\documentclass[12pt]{amsart}
\usepackage{amssymb}
\newcommand{\Ring}{\mathsf{R}}
\newcommand{\Sring}{\mathcal{R}}
\newcommand{\F}{\mathbb{F}}
\newcommand{\kf}{\mathsf{k}}
\newcommand{\p}{\mathfrak{p}}
\newcommand{\q}{\mathfrak{q}}
\newcommand{\T}{\mathcal{T}}

\newcommand{\Cat}{\mathcal{C}}
\newcommand{\OCat}{\mathsf{O}}
\newcommand{\Str}{\mathcal{O}}
\newcommand{\End}{\operatorname{End}}
\newcommand{\Hom}{\operatorname{Hom}}
\newcommand{\Ext}{\operatorname{Ext}}
\newcommand{\Loc}{\operatorname{\mathsf{L}}}
\newcommand{\h}{\mathfrak{h}}
\newcommand{\SBim}{\operatorname{SBim}}
\newcommand{\C}{\mathbb{C}}
\newcommand{\g}{\mathfrak{g}}
\newcommand{\Vfun}{\mathbb{V}}
\newcommand{\Hecke}{\mathcal{H}}

\newcommand{\Coh}{\operatorname{Coh}}

\newcommand{\U}{\mathcal{U}}
\newcommand{\Nilp}{\mathcal{N}}
\newcommand{\St}{\mathsf{St}}

\newcommand{\Acal}{\mathcal{A}}
\newcommand{\Tilt}{\mathcal{E}}
\newcommand{\Z}{\mathbb{Z}}
\newcommand{\Br}{\mathsf{Br}}

\newcommand{\UU}{\mathbf{U}}

\newcommand{\Wh}{\mathsf{Wh}}

\newcommand{\im}{\operatorname{im}}
\newcommand{\Zcal}{\mathcal{Z}}
\newtheorem{Thm}{Theorem}[subsection]
\newtheorem{Prop}[Thm]{Proposition}
\newtheorem{Cor}[Thm]{Corollary}
\newtheorem{Lem}[Thm]{Lemma}
\theoremstyle{definition}
\newtheorem{Ex}[Thm]{Example}
\newtheorem{defi}[Thm]{Definition}
\newtheorem{Rem}[Thm]{Remark}

\numberwithin{equation}{section}
\title{Quantum category O vs affine Hecke category}
\author{Ivan Losev}
\begin{document}
\begin{abstract}
The goal of this paper is to relate the quantum category $\mathcal{O}$ (known also as the category of modules over the mixed quantum group)
at an odd root of unity to the affine Hecke category. Namely, we prove equivalences of highest weight categories between integral blocks 
of the affine category $\mathcal{O}$ and the heart of the so called ``new'' t-structure on the affine Hecke category. In order to 
do this we deform our categories over the formal neighborhood of $0$ in the dual affine Cartan and show that the categories 
of standardly filtered objects in the deformations are equivalent. For this, we construct functors from the deformed categories 
to the category of bimodules over the formal power series on the affine Cartan. Then we use what we call the  Rouquier-Soergel 
theory, also developed in this paper, to show that on the categories of standardly filtered objects, these functors are full embeddings with the same image. 
\end{abstract}
\maketitle
\tableofcontents

\section{Introduction}
\subsection{Representations of quantum groups at roots of unity}
This paper describes certain categories of representations of quantum groups at roots of $1$. Let $\g$ be a semisimple Lie algebra. 
Fix an indeterminate $v$ and consider the Drinfeld-Jimbo quantum group $\UU_v$ for $\g$. One can define various forms of $\UU_v$
over the ring $\C[v^{\pm 1}]$ (we can replace $\C$ with $\Z$ but we will not need this in the present paper). In \cite[Section 4]{Lusztig_quantum_deformation} Lusztig introduced the form with divided powers commonly referred to as the {\it Lusztig form}.
In \cite{DCK}, De Concini and Kac considered the ``default'' form of $\UU_v$ known as the De Concini-Kac form. Finally, 
in \cite{Gaitsgory}, Gaitsgory introduced yet another form called ``mixed'' (or ``hybrid''), which is intermediate 
between the Lusztig and De Concini-Kac forms. This is the form we are mostly concerned with in this paper. 

With all these three forms, we can set $v$ to a nonzero number $q$ getting algebras over $\C$. The most interesting case 
is when $q$ is a root of unity, $\epsilon$, in which case the representations of the three algebras behave differently.
The representations of the Lusztig form at a root of unity are understood best (we'll briefly recall the known results
below), while there are still several basic open questions concerning the irreducible representations of the De Concini-Kac
form. The goal of this paper is to completely describe a reasonable category of representation of the mixed form,
which we refer to as the quantum category $\OCat$. 

Now we explain some reasons to be interested in the representation theory of quantum groups at roots of $1$. 

First, there is a connection to the representations of affine Lie algebras pioneered in \cite{KL}: there is a monoidal equivalence 
between the category of finite dimensional representations of the Lusztig form for $\g$ and the Kazhdan-Lusztig category of representations
of the affine Lie algebra $\hat{\g}$ at the negative level (the value of $q$ is read of the level). Gaitsgory, \cite{Gaitsgory}, conjectured an extension of the Kazhdan-Lusztig equivalence 
to a derived equivalence between the quantum category $\OCat$ and the full affine category $\OCat$.      

Second, there is a connection/ an analogy between the representations of (in fact, all three) forms of $\UU_v$ specialized to $\epsilon$ and the representations of various forms of the universal enveloping algebra of $\g$ over $\Z$ specialized to a field $\F$ of characteristic $p>0$. In order to have a formal connection (rather than an analogy) one needs to take $\epsilon$ to be a root of $1$ of order $p^k$ for some $k\in \Z_{>0}$.
This connection has been first observed by Lusztig, \cite{Lustzig_modular_quantum}, in the context of the Lusztig form 
(the corresponding object on the classical side is the hyperalgebra of $\g_\F$) and since then was explored extensively
in the same context. The analogy between representations of the De Concini-Kac form and $U(\g_\F)$ was observed 
in \cite{DCK}, but a formal connection (existing when the order of $\epsilon$ is $p^k$) was established only relatively recently, see, e.g., \cite{Tanisaki_abelian} and references there. The analogy/ connection for the mixed form will be discussed below.  

Third, categories of modules over various forms of quantum groups at roots of $1$ appear in various context of Geometric representation
theory. For the Lusztig form (at odd roots of $1$) this is a subject on \cite{ABG}. 

\subsection{Quantum category $\OCat$}
Let $\g$ be a semisimple Lie algebra and $\epsilon$ be a root of $1$ of odd degree $d$ (coprime to $3$ if $\g$ has summands of type $G_2$).
Let $I$ be the indexing set for the simple roots of $\g$ and let $\Lambda$ be the weight lattice of $\g$. Let $\h$
and $W$ be the Cartan subalgebra and the Weyl group of $\g$. On the space $\h^*$ we consider the $W$-invariant 
symmetric bilinear form $(\cdot,\cdot)$ on $\h^*$ normalized in such a way that the square of a short root for any 
of the simple summands of $\g$ is $2$. 

First, we explain what one means by the mixed form $U_v^{mix}$. This is the $\C[v^{\pm 1}]$-form with generators $K_\nu, \nu\in \Lambda,$
$F_i$ and the divided powers $E_i^{(\ell)}$ (which is why this form is ``mixed''). 

Now we introduce the integral part of the what we call the quantum category $\OCat$ (one can give a more general definition by 
removing ``integral'' and also consider a version over more general rings, this will be done in Section \ref{SS_O_quant}).
Namely, notice that the specialized $\C$-algebra $U_\epsilon^{mix}$ is naturally graded by the root lattice $\Lambda_0\subset \Lambda$. 
The category $\OCat_\epsilon$, by definition, consists of finitely generated $\Lambda_0$-graded $U^{mix}_\epsilon$-modules
$M=\bigoplus_{\lambda\in \Lambda_0}M_\lambda$ where $K_\nu$ acts on  $M_\lambda$ by $\epsilon^{(\nu,\lambda)}$ and the set
$\{\lambda\in \Lambda_0| M_\lambda\neq 0\}$ is bounded from the above with respect to the usual dominance order on 
$\Lambda_0$.

The category $\OCat_\epsilon$ was suggested by Gaitsgory in \cite{Gaitsgory} and studied in detail by 
Situ in \cite{Situ1,Situ2,Situ3}. We will review some results of Situ's papers below. For now, note that 
$\OCat_\epsilon$ is highest weight in a suitable sense (this is basically because $U^{mix}_\epsilon$
has triangular decomposition), see \cite{Situ1} for details. The standard objects in 
$\OCat_\epsilon$ are Verma modules $\Delta_\epsilon(\lambda),\lambda\in \Lambda_0,$ defined in the usual way. 
This structure is very important for our approach to describing $\OCat_\epsilon$. 
Also note that $\OCat_\epsilon$ decomposes into the sum of its infinitesimal blocks. 
Let $W^a:=W\ltimes \Lambda_0$ be the affine Weyl group. It acts on $\Lambda_0$ as follows:
$w\cdot \lambda=w(\lambda+\rho)-\rho, t_\mu\cdot \lambda=\lambda+d\mu$ for $w\in W, \lambda,\mu\in \Lambda_0$.
For a $W^a$-orbit $\Xi$ in $\Lambda_0$ we write $\OCat_\epsilon^\Xi$ for the Serre span of 
$\Delta_\epsilon(\lambda)$ with $\lambda\in \Xi$. Then we have 
$$\OCat_\epsilon=\bigoplus \OCat_\epsilon^\Xi,$$
where the sum is taken over all $W^a$-orbits in $\Lambda_0$, see, e.g., \cite[Proposition 3.7]{Situ1}.

Finally, we note that the modular analog of $\OCat_\epsilon$ is the category studied in 
\cite[Section 6]{ModHCO}. This is a suitable full subcategory in the category of $(\g_\F,B_\F)$-modules
(here $\F$ is an algebraically closed field of characteristic $p$ and $B_\F$ is a Borel subgroup
in the adjoint semisimple algebraic group $G_\F$ with Lie algebra $\g_\F$).

\subsection{Main result}\label{SS:main_result}
Our description of the category $\mathcal{O}_\epsilon$ will be in terms of a t-structure on the affine Hecke category. 
Let us explain the details.  Note that $W^a$ is a Coxeter group, let $I^a$
denote the set of simple reflections. Pick a proper subset $J\subset I^a$. Associated to $J$ we have the additive category 
of singular Soergel modules $\,_J \operatorname{SMod}(W^a)$. Consider its homotopy category $\,_J\Hecke_0:=K^b(\,_J \operatorname{SMod}(W^a))$.
This is what we mean here by the affine Hecke category. This category has a number of other realizations. For example, one can realize it 
as the Whittaker derived category of constructible sheaves on the affine flag variety for the Langlands dual group $G^\vee$, see
\cite[Section 4]{BY}.  In particular, the category  $\,_J\Hecke_0$ comes with a distinguished t-structure called the {\it perverse t-structure}
(see also \cite[Section 6]{EL}).
Its heart is a highest weight category with standard objects labeled by the elements of $W_J\backslash W^a$, for $x\in W^a,$
we write $\Delta^-(W_Jx)$ for the corresponding standard object. We also note that  $\,_J\Hecke_0$ is a right module category
over the affine braid group $\Br^a$. For $\lambda\in \Lambda_0$, we write $J_\lambda$ for the corresponding element of $\Br^a$. 
The corresponding endo-functor of $\,_J\Hecke_0$ will be denoted by $?*J_\lambda$.

We will need another t-structure on $\,_J\Hecke_0$ that can be called {\it Frenkel-Gaitsgory} (it was originally introduced in a related
but different context in \cite{FG}), {\it new} (as it was called in \cite{BLin}, where a special case was studied) or {\it stabilized}
t-structure, which is how we call it in the present paper. This t-structure is introduced and studied in detail in Section 
\ref{S_stabilized_Hecke}. For now, we will need the following properties of its heart that is denoted by $\,_J\OCat^{st}$:

\begin{itemize}
\item[(S1)] $\,_J\OCat^{st}$ is a highest weight category. Its standard objects $\Delta^{st}(W_J x)$ are labelled by the elements 
$W_J\backslash W^a$ as follows: pick $x=wt_\mu$ and $\lambda$ with $w\in W, \lambda,\mu\in \Lambda_0$ such that 
$\mu-\lambda$ is anti-dominant. Then $\Delta^{st}(W_Jx):=\Delta^-(xt_{-\lambda})*J_\lambda$.
\item[(S2)] We have $D^b(\,_J\OCat^{st})\xrightarrow{\sim} \,_J\Hecke_0$. 
\end{itemize}

Our claim is that each block $\OCat^\Xi_\epsilon$ is equivalent to one of the categories $_J\,\OCat^{st}$. Let $\alpha^\vee_i,i\in I,$
denote the simple coroots, and $\alpha_0^\vee$ denote the maximal  coroot. The locus 
$$\{\lambda\in \Lambda_0| \langle\alpha^\vee_i,\lambda+\rho\rangle\leqslant 0, \langle\alpha_0^\vee,\lambda+\rho\rangle\geqslant -d\}$$
is a fundamental domain for the action of $W^a$ on $\Lambda_0$. Let $\lambda^\circ$ be the unique representative of $\Xi$
in the fundamental domain. The stabilizer of $\lambda^\circ$ in $W^a$ is a standard parabolic subgroup, for $J$ we take 
its set of simple reflections. With this preparation, we can state the main result of the paper. 

\begin{Thm}\label{Thm:main_basic}
There is an equivalence $\OCat_\epsilon^\Xi\xrightarrow{\sim} \,_J\OCat^{st}$ sending $\Delta_\epsilon(x^{-1}\cdot \lambda^\circ)$
to $\Delta^{st}(W_Jx)$ for each $x\in W^a$.
\end{Thm}

\subsection{Description of approach}\label{SS_approach_descr}
The first crucial idea is that Theorem \ref{Thm:main_basic} follows from an equivalence of deformed categories: instead of $\C$-linear 
categories we consider $\Ring$-linear categories, where $\Ring$ is the completion at $0$ of the symmetric algebra of the affine Cartan. 
To define a deformation $\Hecke$ of $\Hecke_0$, we replace the Soergel modules with a suitable version of Soergel bimodules. 
To deform $\OCat_\epsilon$  we follow an old idea of Jantzen: the deformation comes, roughly, from perturbing the actions
of the elements $K_\nu$ and $v$, it is the deformation of $v$ that gives the imaginary direction in the affine Cartan.
The resulting deformation of $\OCat_\epsilon$ will be denoted by $\OCat_\Ring$. 

Now we address the issue of how we define the stabilized t-structure on $\,_J\Hecke$ and how we establish a highest weight structure
on the heart (and on $\,_J\OCat^{st}$). The question relatively easily reduces to the case when $J=\varnothing$. Here a ``coherent realization''
of the stabilized t-structure on $\Hecke_0$ is known thanks to \cite{BLin}: the category $\OCat^{st}$ is equivalent to the category 
of $G$-equivariant coherent sheaves over a certain sheaf of algebras (denoted by $\pi^*\Acal$ in the main body of the paper) on the Springer resolution. Here $G$ is the adjoint group with Lie algebra $\g$. The sheaf of algebras in question is defined over a finite localization of $\Z$ so we can base change it to $\F$, an algebraically closed field of large enough positive characteristic. It turns out that the category $\operatorname{Coh}^{G_\F}(\pi^*\Acal_\F)$ is highest weight: by \cite[Section 6]{ModHCO} it is equivalent to the principal block of the modular category $\OCat$. One can describe the images of the Verma modules in 
$\operatorname{Coh}^{G_\F}(\pi^*\Acal_\F)$, this is also done in \cite[Section 6]{ModHCO}. The description makes sense over 
a finite localization of $\Z$ and hence over $\C$, this gives a highest weight structure on $\OCat^{st}$. The standard objects for this highest weight structure are in fact in $\OCat^-$, the heart of the perverse t-structure. They uniquely deform to $\Ring$-flat objects in 
the heart of the perverse t-structure in $\Hecke$. 
Using this, and some formal nonsense regarding highest weight structures, one can 
then define a full subcategory $\OCat^{st}_\Ring$ of $\Hecke$ that is highest weight, and prove that it is the heart of 
a t-structure (and also an analog of (S2)). The overall procedure is indirect and quite technical, so one should wonder whether 
one define $\OCat^{st}_\Ring$ together with its highest weight structure directly using the Soergel or constructible realization
of $\Hecke$.   

So, now we have two highest weight categories over $\Ring$: $\OCat_\Ring^\Xi$ and $\,_J\OCat^{st}_\Ring$. We want to show that they
are equivalent as highest weight categories. For this we use what we call the {\it Rouquier-Soergel} theory, it is developed in 
Section \ref{S_RS}. The ``Soergel'' part is 
that we construct functors to the category of $\Ring^{W_J}$-$\Ring$-bimodules (here $W_J$ is the standard parabolic subgroup
of $W^a$ with simple reflections $J$). The functor from $\,_J\OCat^{st}_\Ring$ comes from the forgetful functor 
$\,_J\Hecke\rightarrow D^b(\Ring^{W_J}\operatorname{-}\Ring\operatorname{-bimod})$, while the functor from 
$\OCat^\Xi_\Ring$ is given by taking the Whittaker coinvariants. The ``Rouquier'' part is that instead of looking at the image of 
the subcategory of projective/tilting objects (which would be a ``Soergel'' thing to do) we show that our functors are fully faithful on the standardly filtered objects (which is why we need to deform to $\Ring$) and the images in $\Ring^{W_J}\otimes \Ring\operatorname{-mod}$ of the subcategories of standardly filtered objects of $\OCat_\Ring^\Xi$ and $\,_J\OCat^{st}_\Ring$  coincide. A related approach to proving equivalences of highest weight categories  was pioneered in \cite{rouqqsch} and its variants were used in a number of subsequent papers. Our situation is quite different from those papers, for example, the functors we consider are not quotient functors and they do not induce isomorphisms between $\operatorname{Ext}^1$'s.     

\subsection{Related work}
An extended version of the Kazhdan-Lusztig conjecture of Gaitsgory from \cite{Gaitsgory} 
was established by Chen and Fu in \cite{CF}. We note that Gaitsgory has also made a prediction about the behavior of the t-structures 
under the equivalences, \cite[Remark 0.1.5]{Gaitsgory}, and \cite{CF} verifies this prediction. This description agrees with Theorem \ref{Thm:main_basic}\footnote{Recently, Dhillon informed us that he combined results of \cite{CF} with the affine localization
theorem to  get a geometric version of  Theorem \ref{Thm:main_basic}, see \cite[Section 1.3]{DL}}. Overall, 
techniques used in the present paper are very different from \cite{CF} and results are different as well (\cite{CF} does not care about the 
parity of $d$, on the other hand \cite{CF} imposes inequalities on $d$ that do not appear in our approach). Another related 
development (where the techniques are also very different from ours) 
is  \cite{Yang}, where Yang establishes an equivalence between the quantum category $\mathcal{O}$
and the category of twisted Whittaker D-modules on the affine flag variety.  

A more closely related series of works are \cite{Situ2,Situ3}. In \cite{Situ2}, Situ proves an equivalence between 
the block of $\OCat_\epsilon$ corresponding to the orbit of $-\rho$ and the category $\operatorname{Coh}^G(\tilde{\Nilp})$.
Using results of \cite{BLin}, it is easy to see that the latter category is equivalent to $\,_I\OCat^{st}$
(where $I$ is still the set of simple roots for $\g$ and hence a subset of $I^a$). In \cite{Situ3}, published when the present paper was in preparation, Situ 
proves an equivalence $\OCat^\Xi_\epsilon\xrightarrow{\sim}\Coh^G(\pi^*\Acal)$, where $\Xi$ is the orbit of 
$0$, under several assumptions on $d$, most notably that $d>h$ (the Coxeter number; and so the source category is a regular block, which does not exist when $d\leqslant h$) and $d$
is a prime power. 
The method of proof has some similarities and some differences with ours:
\begin{itemize}
\item Situ uses a Soergel type functor to relate the regular and the singular block and then shows that the functor is fully faithful
on the projectives after a suitable deformation. The target categories for the Soergel functors coincide thanks to \cite{Situ2},
and the images of projectives of the deformations of $\OCat^\Xi_\epsilon, \Coh^G(\pi^*\Acal)$ coincide. For our approach, 
the target categories are the same on the nose, but instead of the projectives we deal with the standardly filtered objects. 
Also our base of deformation is larger than Situ's.    
\item To show that the images of projectives coincide, Situ uses Tanisaki's localization theorem, \cite{Tanisaki_abelian}.
This imposes the additional restriction that $d$ is a prime power that does not arise in our approach\footnote{As Situ points out in 
\cite[Remark 5.14]{Situ3} one may be able to remove this restriction.}. 
\item Overall, the coherent realization of the affine Hecke category seems to be of crucial importance for Situ's approach.
For us, it plays a purely technical role: it allows to define $\OCat^{st}_\Ring$ and its highest weight structure. 
Relying on the coherent realization brings some difficulties, at least with our approach: it's harder to deal 
with deformations over $\Ring$ (see Section \ref{SS_coherent_full_deformation} for a brief discussion), and one only easily sees blocks of $\,_J\Hecke_0$ with $J\subset I$. 
\end{itemize} 

\subsection{Variations and future developments}
\subsubsection{More general settings}
One can consider the following modifications of $\OCat_\epsilon$:
\begin{itemize}
\item[(a)] Non-integral parts $\OCat^\zeta_\epsilon$, as defined in Section \ref{SS_O_quant}: we relax the condition on the action of $K_\nu$'s
on $M_\lambda$'s.
\item[(b)] Analogs of $\OCat_\epsilon$, where $\epsilon$ is an even root of $1$ (we ignore the issue of $G_2$ here).
\item[(c)] Analogs of $\OCat_\epsilon$, where we replace $\C$ with an algebraically closed characteristic $p$ field. 
\end{itemize}

Analogs of Theorem \ref{Thm:main_basic} should hold in all these settings. For (a), we expect that one just replaces the affine Weyl
group of $G$ with the affine Weyl group of a suitable pseudo-Levi (one also needs to replace the direct analog of $\OCat_\epsilon^\Xi$
with an actual block). For (b) one likely needs to replace the mixed quantum group with its suitably defined ``even part''
and then, in the integral case, one also needs to replace the affine Weyl group for $G$ with that for the adjoint form of 
the Langlands dual $G^\vee$. In (c) one likely needs to replace the usual Soergel bimodules with 
their modification due to Abe. In all these settings we expect that after relatively minor modifications,
our approach will still work.

\subsubsection{Quantum Harish-Chandra bimodules}
In the remainder of this section we outline a program aimed at understanding other aspects of the representation theory 
of quantum groups at roots of $1$, where Theorem \ref{Thm:main_basic} serves as the first step.

Now we would like to sketch a conjectural analog of Theorem \ref{Thm:main_basic} for a close relative of the category 
$\OCat_\epsilon$, the category of quantum Harish-Chandra bimodules. Below we give a sketch of the definition. 
For simplicity, we assume that $d>h$ and $d$ is odd (and coprime to $3$ if $\g$ has summands of type $G_2$).
We will concentrate on the principal block.  

Let $U^{DK}_\epsilon, 
\dot{U}_\epsilon$ denote the De Concini-Kac and Lusztig forms of the quantum group. Inside $U^{DK}_\epsilon$ we can consider 
a certain ``even'' subalgebra $U_\epsilon^{ev}$ (informally, $U^{DK}_\epsilon$ is an unramified $2^{|I|}$-fold cover of 
$U^{DK}_\epsilon$). One can show that $\dot{U}_\epsilon$ acts on $U_\epsilon^{ev}$ by the adjoint action. Let $U^{lf}_\epsilon$
denote the sum of all finite dimensional $\dot{U}_\epsilon$-submodules in $U_\epsilon^{ev}$, this is a subalgebra. 
The subalgebra of invariants of $\dot{U}_\epsilon$ in $U^{ev}_\epsilon$ coincides with the Harish-Chandra center and is a central
subalgebra of $U^{lf}_{\epsilon}$. Let $U^{lf}_{\epsilon,0}$ denote the quotient of $U^{lf}_{\epsilon}$ by the 
annihilator of the trivial representation of $U^{ev}_\epsilon$ in the Harish-Chandra center. 

We can consider the category  $U^{lf,opp}_{\epsilon,0}\operatorname{-mod}^{\dot{U}_\epsilon}$ of (weakly) $\dot{U}_\epsilon$-equivariant 
right $U^{lf}_{\epsilon,0}$-modules. One can show that modules in this category come with a left $U^{lf}_\epsilon$-action 
that commutes with the right $U^{lf}_{\epsilon,0}$-action. The action of the Harish-Chandra center decomposes 
$U^{lf,opp}_{\epsilon,0}\operatorname{-mod}^{\dot{U}_\epsilon}$ into the sum of infinitesimal blocks, still parameterized 
by the $W^a$-orbits in $\Lambda_0$. Denote the block corresponding to $\Xi$ by $\operatorname{HC}_{\epsilon,\Xi}$. 
One can recover $J\subset I^a$ from $\Xi$ as in Section \ref{SS:main_result}.

We expect the following statement to hold, see  \cite[Theorem 5.9]{ModHCO} for the analogous statement for 
modular Harish-Chandra bimodules:

\begin{itemize}
\item[(*)] We have an equivalence $D^b(\operatorname{HC}_{\epsilon,\Xi})\xrightarrow{\sim} \,_J\Hecke_0$. This equivalence is t-exact 
with respect to the perverse t-structures. 
\end{itemize}

We also expect to have singular analogs of (*) (with singularities on both sides). The general strategy should be similar to the approach
from \cite{ModHCO} as well as a functor from the deformed version of $\operatorname{HC}_{\epsilon,\Xi}$ to the category 
of $\Ring^{W_J}$-$\Ring$-bimodules analogous to the Whittaker coinvariants. Some things are expected to be more complicated than in the 
modular case: for a example, the localization theorem was used in \cite{ModHCO} to control the $K_0$ of the HC category, 
such a theorem is not available in the quantum case in the form/ generality we need (in fact, one could try to prove it using (*)
and Theorem \ref{Thm:main_basic} as explained in the next part).   

\subsubsection{Towards quantum derived localization}
Combining (*) and Theorem \ref{Thm:main_basic} one should get the following derived equivalence:
\begin{equation}\label{eq:quantum_HC_O_equiv}
\bullet\otimes^L_{U^{lf}_\epsilon}\Delta_\epsilon(-2\rho): D^b(\operatorname{HC}_{\epsilon,\Xi})\xrightarrow{\sim}
D^b(\OCat^\Xi_\epsilon).
\end{equation}
A similar equivalence for the modular categories was interpreted in \cite[Section 6]{ModHCO} as the equivariant 
version of the derived localization theorem from \cite{BMR}. One could wonder if one can then interpret 
(\ref{eq:quantum_HC_O_equiv}) in a similar fashion, and then ``de-equivariantize'' to get the actual derived localization
theorem.  We note that under some additional restrictions on $d$ (essentially, that it is a prime power)
the derived localization theorem in the quantum setting was proved by Tanisaki, \cite{Tanisaki_derived1,Tanisaki_derived2}.

\subsubsection{Equivariant irreducibles}
An important problem in the representation theory of the De Concini-Kac forms is to describe the $K_0$ classes of their (finite dimensional)
irreducible representations. This setting is analogous to the representation theory of $U(\g_\F)$, where $\F$ is a characteristic $p$
field. There are two approaches to the problem in the latter setting. 
\begin{enumerate}
\item To work with ordinary irreducible representations, \cite{BM}. The disadvantages here is the absence of ``standard objects''
as well as the explicit combinatorial parametrization of the irreducibles. So one does not expect any explicit (Kazhdan-Lusztig
 type) formulas for the $K_0$-classes.   
\item To work with a suitable category of equivariant modules and study irreducibles there. This is an approach 
taken in \cite{BL}. One now has standard objects, and reasonably explicit parameterizations of simples as well
as formulas for the $K_0$-classes. The price to pay is that it is hard to relate the equivariantly simple objects to the 
actual simples. 
\end{enumerate}

Theorem \ref{Thm:main_basic} (and its deformed version) should allow one to establish the analog of (2) for quantum 
groups -- which we expect to be different from \cite{BL}. 

%\subsubsection{Towards double affine categories}
%This part is not directly related to the story about quantum groups at roots of $1$ but it was a motivation for 
%developing the Rouquier-Soergel theory, which is the main tool in the proof of Theorem \ref{Thm:main_basic}
%so it seems worth to have a brief discussion here. 

\smallskip

{\bf Acknowledgements}: I would like to thank Pramod Achar, Roman Bezrukavnikov, Gurbir Dhillon, Pavel Etingof, Yuchen Fu, Quan Situ, and James Tao
for stimulating discussions and Peng Shan for bringing \cite{Situ1}-\cite{Situ3} to my attention. My work was supported by the NSF under grant DMS-2001139. 

\section{Highest weight categories}
\subsection{Finite posets}\label{SS_hw_finite}
\subsubsection{Definition}\label{SSS_hw_finite_def}
In this section we recall the classical definition of a highest weight category over a ring following
\cite{rouqqsch}.

Let $\Ring$ be a Noetherian ring. For a prime ideal $\p\subset \Ring$ let $\kf(\p)$ denote the fraction field of $\Ring/\p$.

By {\it finite $\Ring$-algebra} we mean an associative unital $\Ring$-algebra $A_\Ring$ that is a finitely generated $\Ring$-module. We say that $A_{\Ring}$ is {\it finite projective} if, in addition, it is projective as an $\Ring$-module.
Let $A_\Ring\operatorname{-mod}$ denote the category of finitely generated $A_\Ring$-modules.
If $\Ring'$ is an $\Ring$-algebra, we write $A_{\Ring'}$ for $\Ring'\otimes_{\Ring}A_\Ring$.

Following \cite[Definition 4.11]{rouqqsch} we can define the notion of a highest weight category over $\Ring$. It is an $\Ring$-linear abelian category $\Cat_R$ equivalent to $A_\Ring\operatorname{-mod}$ for some finite projective $\Ring$-algebra $A_\Ring$ with additional structures that are supposed to satisfy certain axioms. The structures are as follows:
\begin{enumerate}
\item  a finite poset $\T$
\item and a family of objects $\Delta_\Ring(\tau)\in \Cat_\Ring$ to be called {\it standard}.
\end{enumerate}

To state the axioms we need some notation. For a subset $\T'\subset \T$, by
a $\mathcal{T}'$-standardly filtered object we mean an object $M_\Ring\subset \Cat_\Ring$
that admits a filtration with successive quotients of the form $\Delta_\Ring(\tau)\otimes_\Ring P$,
where $P$ is a finitely generated projective $\Ring$-module. Let $\Cat_\Ring^{\Delta,\T'}$
denote the full subcategory of $\Cat_\Ring$ consisting of all such objects. For
$\tau\in \T$ define $\T(>\tau)$ as $\{\tau'\in \T| \tau'> \tau\}$. We define
$\T(\geqslant \tau), \T(\leqslant \tau),\T(<\tau)$ similarly. For $\tau\leqslant \tau'$,
set $\T([\tau,\tau']):=\T(\leqslant \tau')\cap \T(\geqslant \tau)$. We write
$\Cat_\Ring^\Delta$ for $\Cat_\Ring^{\Delta,\T}$ and $\Cat_\Ring^{\Delta,>\tau}$
for $\Cat_\Ring^{\Delta,\T(>\tau)}$.

The axioms of a highest weight category over $\Ring$ are as follows.

\begin{itemize}
\item[(i)] The objects $\Delta_\Ring(\tau)$ are flat over $\Ring$.
\item[(ii)] $\Ring\xrightarrow{\sim}\End_{\Cat_\Ring}(\Delta_\Ring)$ for all
$\tau\in \T$.
\item[(iii)] If $\Hom_{\Cat_{\Ring}}(\Delta_\Ring(\tau_1),\Delta_\Ring(\tau_2))\neq 0$,
then $\tau_1\leqslant \tau_2$.
\item[(iv)] If $M_\Ring$ is a nonzero object in $\Cat_\Ring$, then there is $\tau\in \T$
such that $\Hom_{\Cat_\Ring}(\Delta_\Ring(\tau),M_\Ring)\neq 0$.
\item[(v)] For each $\tau\in \T$ there is a projective object in $\Cat_\Ring$ that admits
an epimorphism onto $\Delta_\Ring(\tau)$ with kernel in $\Cat_\Ring^{\Delta,>\tau}$.
\end{itemize}

%\begin{defi}\label{defi:hw_equiv}
%By an equivalence of highest weight categories over $\Ring$ we mean an $\Ring$-linear equivalence of %abelian categories that sends standard objects to standard objects.
%\end{defi}

\subsubsection{Standardly filtered objects}
Here we study some properties of standardly filtered objects. The following is a direct consequence of
\cite[Lemma 4.21]{rouqqsch}.

\begin{Lem}\label{Lem:stand_exact_prop}
The following claims hold:
\begin{enumerate}
\item $\Cat_\Ring^\Delta$ is closed under taking direct summands.
\item The kernel of an epimorphism of standardly filtered objects is standardly filtered.
\end{enumerate}
\end{Lem}

\begin{Rem}\label{Rem:exact_cat_stand}
It follows from (2) of Lemma \ref{Lem:stand_exact_prop} that $\Cat_\Ring^\Delta$ is an exact category in the sense of Quillen.
\end{Rem}

Further, we claim that every $M\in \Cat_\Ring^{\Delta}$ admits a natural filtration
indexed by poset coideals in $\T$.

%Now fix a total
%order $\prec$ on $\T$ refining $<$. 
%Then on every $M\in \Cat_\Ring^\Delta$ we have
%an ascending filtration $M_{\succeq \tau},\tau\in \T,$ such that $M_{\succeq \tau}/M_{\succ \tau}
%\cong P_{M,\tau}\otimes_\Ring \Delta_\Ring(\tau)$ for a projective $\Ring$-module $P_{M,\tau}$.

\begin{Lem}\label{Lem:stand_filtration}
The following claims are true:
\begin{enumerate}
\item For each poset coideal $\T^0$ there is a unique subobject $M_{\T^0}\subset M$ with
the following property: $M_{\T^0}\in \Cat_\Ring^{\Delta,\T^0}$ and $M/M_{\T^0}\in \Cat_\Ring^{\Delta,\T\setminus \T^0}$.  
\item If $M,N\in \Cat_\Ring^\Delta$ and $\varphi:M\rightarrow N$ is a morphism, then
$\varphi(M_{\T^0})\subset N_{\T^0}$.
\item If $\tau$ is an epimorphism, then $$\varphi(M_{\T^0})= N_{\T^0}.$$
\end{enumerate}
\end{Lem}
\begin{proof}
By definition, $M$ admits a filtration $\{0\}=M_0\subsetneq M_1\subsetneq\ldots\subsetneq
M_k=M$ such that $M_i/M_{i-1}\cong P_i\otimes_\Ring \Delta_\Ring(\tau_i)$ for some $\tau_1,\ldots,\tau_k\in \T$.
By \cite[Proposition 4.13]{rouqqsch}, we have $\Ext^i_{\Cat_\Ring}(\Delta_\Ring(\tau),\Delta_\Ring(\tau'))\neq 0\Rightarrow
\tau<\tau'$.  The existence of $M_{\T^0}$ with the required properties follows. The uniqueness follows from axiom 
(iii). 

(2) is also an immediate consequence of axiom (iii). Now we proceed to (3).  Let $\tau$ be a maximal element in $\T$
with $M_{\T^0}\neq \{0\}$, where $\T^0:=\T(\geqslant \tau)$. Once we know that $\varphi(M_{\T^0})=N_{\T^0}$ we can replace $M$ and $N$
with $M/M_{\T^0}$ and $N/N_{\T^0}$ and argue by induction on the cardinality of $\{\tau\in \T| M_{\T(\geqslant \tau)}\neq 0\}$.
%It is enough to prove (3) in the case when $\T^0=\{\tau\}$, where is minimal such that
%$N_{\succeq \tau}\neq 0$: once we know this case, we can replace $M,N$ with $M/M_{\T^0\setminus \{\tau\}}, N/N_{\succeq \tau}$
%and  use the induction on the total order $\prec$. 

Take a projective object $P$ with $P\twoheadrightarrow \Delta_\Ring(\tau)$ as in axiom (v).
We have $M_{\T^0}=P_{M,\tau}\otimes_{\Ring}\Delta_\Ring(\tau), N_{\T^0}=P_{N,\tau}\otimes_{\Ring}\Delta_\Ring(\tau)$
for projective $\Ring$-modules $P_{M,\tau},P_{N,\tau}$.
We have $\Hom_{\Cat_\Ring}(P,M)=P_{M,\tau}\twoheadrightarrow P_{N,\tau}=\Hom_{\Cat_\Ring}(P,N)$,
that implies $M_{\T_0}\twoheadrightarrow N_{\T^0}$.
\end{proof}

\begin{Rem}\label{Rem:coarsest_hw_order}
Suppose $\Delta_\Ring(\tau),\tau\in \T,$ are the standard objects for some highest weight structure.
The coarsest order making $\Cat_\Ring$ into a highest weight category with these standard objects
is given by as the transitive closure of the following relation: $\tau\preceq \tau'$ if
there is $i$ with $\operatorname{Ext}^i_{\Cat_\Ring}(\Delta_\Ring(\tau),\Delta_\Ring(\tau'))\neq 0$.
\end{Rem}

\subsubsection{Projective objects}\label{SSS_proj_hw_finite}
Now we discuss the structure of projective objects in a highest weight category $\Cat_\Ring$.
We write $\Cat_\Ring\operatorname{-proj}$ for the full subcategory of projective
objects in $\Cat_\Ring$, this is an additive category. Thanks to
\cite[Proposition 4.13]{rouqqsch}, we have $\Cat_\Ring\operatorname{-proj}\subset \Cat_\Ring^\Delta$.

Next, we have the following result, \cite[Lemma 4.22]{rouqqsch}.

\begin{Lem}\label{Lem:proj_charact}
Suppose $M\in \Cat_\Ring^{\Delta}$. Then $M\in\Cat_\Ring\operatorname{-proj}$ if and only if
$\operatorname{Ext}^1_{\Cat_\Ring}(M,\Delta_\Ring(\tau))=0$ for all $\tau\in \T$.
\end{Lem}

\begin{Rem}\label{Rem:stand_filt_recover}
Thanks to Lemma \ref{Lem:proj_charact} we can recover $\Cat_\Ring$ from the exact category
$\Cat_\Ring^\Delta$. Namely, Lemma \ref{Lem:proj_charact} recovers the additive category
$\Cat_\Ring\operatorname{-proj}$. Then one recovers $\Cat_\Ring$ from $\Cat_\Ring\operatorname{-proj}$
in a standard way. Namely, for each $\tau\in \T$ choose a projective object $P_{\Ring,\tau}$ as
in axiom (v). Then the functor $\Hom_{\Cat_\Ring}(\bigoplus_\tau P_{\Ring,\tau},\bullet)$
is an equivalence of $\Cat_\Ring$ and $A_\Ring\operatorname{-mod}$, where $A_\Ring$
is the opposite endomorphism algebra of $\bigoplus_\tau P_{\Ring,\tau}$.

In particular, if $\Cat^1_\Ring,\Cat^2_\Ring$ are two highest weight
categories, and $\varphi:\Cat^{1,\Delta}_\Ring\xrightarrow{\sim} \Cat^{2,\Delta}_\Ring$
is an equivalence of exact categories, then $\varphi$ uniquely extends to an equivalence
of highest weight categories $\Cat^1_\Ring\xrightarrow{\sim} \Cat^2_\Ring$.
\end{Rem}

\subsubsection{Costandard and tilting objects}
Let $\Cat_\Ring$ be a highest weight category over $\Ring$ with poset $\mathcal{T}$.
According to \cite[Proposition 4.19]{rouqqsch}, for all $\tau\in \T$, there is a unique (up to isomorphism) object $\nabla_\Ring(\tau)$ that is flat over $\Ring$ and satisfies
\begin{equation}\label{eq:Ext_stand_costand}
\operatorname{Ext}^i_{\Cat_\Ring}(\Delta_\Ring(\tau'),\nabla_\Ring(\tau))\cong \Ring^{\oplus \delta_{i,0}\delta_{\tau,\tau'}},
\end{equation}
where $\delta_{?,\bullet}$ is the Kronecker symbol.
These objects are called {\it costandard}.

\begin{Rem}\label{Rem:hw_field}
Suppose $\Ring$ is a field $\kf$. Then there is a bijection between the set of isomorphism classes of irreducible objects in $\Cat_\kf$ and $\T$: to $\tau\in \T$ we assign the unique irreducible quotient
of $\Delta_\kf(\tau)$, denote it by $L_\kf(\tau)$. The kernel of $\Delta_\kf(\tau)\rightarrow L_\kf(\tau)$
is filtered by $L_\kf(\tau')$ with $\tau'<\tau$, this follows from axiom (v). Also note that
$L_\kf(\tau)\hookrightarrow \nabla_\kf(\tau)$ and the cokernel is filtered by $L_\kf(\tau')$
with $\tau'<\tau$.
\end{Rem}

The category $\Cat_\Ring^{opp}$ is highest weight with respect to the poset $\T^{opp}$
(the same set as $\T$ but with the opposite order) and the standard objects $\nabla_\Ring(\tau)$
for $\tau\in \T$, see \cite[Proposition 4.19]{rouqqsch}. By $\Cat_\Ring^\nabla$ we denote
the full subcategory of costandardly filtered objects in $\Cat_\Ring$.

%Let $A_\Ring$ be a finite projective $\Ring$-algebra. Recall that a finitely generated $A_\Ring$-module
%$I_\Ring$ is called {\it relatively $\Ring$-injective} if it is projective over $\Ring$ and $\operatorname{Hom}_{\Ring}(I_\Ring,\Ring)$ is a %projective right $A_\Ring$-module. Of course,
%if $\Ring$ is a field, then this is just an injective $A_\Ring$-module.
%
%So, it makes sense to speak about relatively $\Ring$-injective objects in $\Cat_\Ring$.
%By \cite[Proposition 4.20]{rouqqsch}, these objects are costandardly filtered.
%The projective objects
%in $\Cat_\Ring^{opp}$ are exactly the relatively injective objects in $\Cat_\Ring$,
%one way to see this is to use that they lie in $\Cat_\Ring^\nabla$ in combination with
%Lemma \ref{Lem:proj_charact} (applied to the highest weight category $\Cat_\Ring^{opp}$).

By a {\it tilting object} in $\Cat_\Ring$ we mean an object in the full subcategory
$\Cat_\Ring\operatorname{-tilt}:=\Cat_{\Ring}^\Delta\cap \Cat_\Ring^\nabla$. According
to \cite[Proposition 4.26]{rouqqsch}, every $M\in \Cat_\Ring^\Delta$ admits a monomorphism
into a tilting object whose cokernel is standardly filtered. Similarly, every costandardly
filtered object admits an epimorphism whose kernel is costandardly filtered. Moreover,
according to \cite[Proposition 4.26]{rouqqsch}, for each $\tau\in \T$, we can find
a tilting object $T_\Ring(\tau)$ that admits
\begin{itemize}
\item a monomorphism from $\Delta_\Ring(\tau)$ with cokernel in $\Cat_\Ring^{\Delta,<\tau}$,
\item and an epimorphism onto $\nabla_\Ring(\tau)$ with kernel in $\Cat_\Ring^{\nabla,<\tau}$.
\end{itemize}

\subsubsection{Highest weight subcategories}\label{SSS_hw_sub}
Let $\T_0\subset \T$ be a poset ideal. Let $\Cat_{\T_0,\Ring}$ denote the Serre span of
$\Delta_\Ring(\tau), \tau\in \T_0$. Then $\Cat_{\T_0,\Ring}$ is a highest weight category
with poset $\T_0$ and the standard objects $\Delta_\Ring(\tau), \tau\in \T_0$, see
\cite[Proposition 4.13]{rouqqsch}. A subcategory of the form $\Cat_{\T_0,\Ring}$ will be called a {\it highest weight subcategory} of
$\Cat_\Ring$. Note that the costandard objects in $\Cat_{\T_0,\Ring}$
are $\nabla_{\Ring}(\tau),\tau\in \T_0$. 
%Note also that if $\Cat_\Ring=A_\Ring\operatorname{-mod}$,
%then $\Cat_{\T_0,\Ring}=A_\Ring/I\operatorname{-mod}$ for a uniquely determined two-sided ideal
%$I\subset A_\Ring$.

\begin{Lem}\label{Lem:full_derived_embedding}
The inclusion functor $D^b(\Cat_{\T_0,\Ring})\rightarrow D^b(\Cat_\Ring)$ is full.
\end{Lem}
\begin{proof}
The tilting objects $T_\Ring(\tau), \tau\in \T_0,$ generate the triangulated category $D^b(\Cat_{\T_0,\Ring})$. Note that there are no higher self-extensions between the
objects $T_\Ring(\tau)$ in $\Cat_\Ring$. The claim of the lemma is a standard consequence of
these two observations.
\end{proof}

\begin{Rem}\label{Rem:proj_construction}
One can construct projective objects satisfying (v) as follows. Fix $\tau\in \T$ and a poset ideal $\T_0$ containing $\tau$. 
If $\tau$ is maximal in $\T_0$, then 
take $P^{\T_0}_{\Ring,\tau}:=\Delta_\Ring(\tau)$. Otherwise, let $\tau_1,\ldots,\tau_k$ be the maximal elements in $\T_0(>\tau)$.
Set $\T'_0:=\T_0\setminus \{\tau_1,\ldots,\tau_k\}$, this is a poset ideal in $\T_0$. 
Suppose that we have constructed a projective object $P^{\T'_0}_{\Ring,\tau}$ in $\Cat_{\Ring, \T'_0}$
satisfying axiom (v).  Form the $\Ring$-modules
$\operatorname{Ext}^1_{\Cat_\Ring}(P_{\Ring,\tau}^{\T'_0},\Delta_\Ring(\tau_i)), i=1,\ldots,k$. 
Observe that $$\operatorname{Ext}^1_{\Cat_\Ring}(P_{\Ring,\tau}^{\T'_0},\Delta_\Ring(\tau_i))\xrightarrow{\sim}
\operatorname{Ext}^1_{\Cat_\Ring}(P_{\Ring,\tau}^{\T'_0}/(P_{\Ring,\tau}^{\T'_0})_{\T'_0\setminus \T_0(<\tau_i)} ,\Delta_\Ring(\tau_i)).$$ Pick  projective covers $U_{\tau,\tau_i}
\twoheadrightarrow \operatorname{Ext}^1_{\Cat_\Ring}(P_{\Ring,\tau}^{\T'_0}/(P_{\Ring,\tau}^{\T'_0})_{\T'_0\setminus \T_0(<\tau_i)} ,\Delta_\Ring(\tau_i))$. 
For $P_{\Ring,\tau}^{\T_0}$ we take the extension of $P_{\Ring,\tau}^{\T_0'}$ by $\bigoplus_{i=1}^k U_{\tau,\tau_i}\otimes_{\Ring}\Delta_\Ring(\tau_i)$ corresponding to the fixed projective covers.

The construction yields epimorphisms $\eta_{\T_0^1,\T_0^2}: P^{\T_0^2}_{\Ring,\tau}\twoheadrightarrow 
P^{\T_0^1}_{\Ring,\tau}$ for poset ideals $\T_0^1\subset \T_0^2$ with $\tau\in \T_0^1$. Moreover, 
$\eta_{\T_0^1,\T_0^3}= \eta_{\T_0^1,\T_0^2}\circ \eta_{\T_0^2,\T_0^3}$ for $\T_0^1\subset \T_0^2\subset \T_0^3$. 
%
%With a careful choice of the projective covers, we can achieve that for every poset ideal $\T_0\subset \T$ containing, we get
%a natural epimorphism $P_{\Ring,\tau}\rightarrow P_{\Ring,\tau}^{\T_0}$, where $P_{\Ring,\tau}^{\T_0}$ is an analog of 
%$P_{\Ring,\tau}$ in $\Cat_{\T_0,\Ring}$. The construction of   
%Then 
%$P_{\Ring,1}(\tau_j)$ is projective, compare to the proof of \cite[Lemma 4.12]{rouqqsch}.  
%Pick a total order on $\T$ refining $\preceq$, let $\tau_1,\ldots,\tau_k$
%be the elements of $\T$ written in decreasing order. For $i\leqslant j$ construct $P_{\Ring,i}(\tau_j)$ inductively as follows.
%We set $P_{\Ring,j}(\tau_j):=\Delta_\Ring(\tau_j)$. Once $P_{\Ring,i}(\tau_j)$ with $i>1$ is constructed, 
\end{Rem}

\begin{Rem}\label{Rem:hw_sub_recovery}
Choose projective objects $P_{\Ring,\tau}:=P^\T_{\Ring,\tau}$ as in Remark \ref{Rem:proj_construction}. 
The epimorphism $\pi_{\T_0,\T}:P_{\Ring,\tau}\twoheadrightarrow P^{\T_0}_{\Ring,\tau}$ identifies the target
with $P_{\Ring,\tau}/(P_{\Ring,\tau})_{\T\setminus \T^0}$. 
The objects $P^{\T_0}_{\Ring,\tau},\tau\in \T_0,$ generate $\Cat_{\T_0,\Ring}$.
Form the opposite endomorphism algebras $A_{\Ring}$ of $\bigoplus_{\tau\in \T}P_{\Ring,\tau}$
and $A_{\Ring,\T_0}$ of $\bigoplus_{\tau\in \T_0}P^{\T_0}_{\Ring,\tau}$. Then there is a natural 
homomorphism $\varpi_{\T_0,\T}:A_\Ring\rightarrow A^{\Ring,\T_0}$ induced by $\eta_{\T^0,\T}$ as every endomorphism of $\bigoplus_{\tau\in \T}P_{\Ring,\tau}$
preserves the sum of kernels of $P_{\Ring,\tau}\twoheadrightarrow P^{\T_0}_{\Ring,\tau}$ and so induces
an endomorphism of $\bigoplus_{\tau\in \T_0}P^{\T_0}_{\Ring,\tau}$. It is easy to see that every endomorphism
of $\bigoplus_{\tau\in \T_0}P^{\T_0}_{\Ring,\tau}$ lifts to an endomorphism of 
$\bigoplus_{\tau\in \T}P_{\Ring,\tau}$, and so $A_\Ring\twoheadrightarrow A_{\Ring,\T_0}$.
The kernel of this epimorphism can be described as follows. Let $e_\tau\in A_\Ring$ denote the 
projection to the direct summand $P_{\Ring,\tau}$, it is an idempotent. 
The kernel is generated by the idempotents $e_\tau$ with $\tau\not\in \T_0$. 
\end{Rem}

\subsubsection{Highest weight quotients}\label{SSS_hw_quot}
Now set $\T^0:=\T\setminus \T_0$, this is a poset coideal. Set $\Cat_{\T^0,\Ring}:=\Cat_\Ring/\Cat_{\T_0,\Ring}$. An equivalent description of
$\Cat_{\T^0,\Ring}$ is as follows. Thanks to axiom (v) we can choose a projective object
$P_{\T^0}$ in $\Cat_\Ring^{\Delta,\T^0}$ that admits an epimorphism onto
$\Delta_\Ring(\tau)$ for all $\tau\in \T^0$. Then $\Cat_{\T^0,\Ring}\cong
A_{\T^0,\Ring}\operatorname{-mod}$, where $A_{\T^0,\Ring}=\operatorname{End}_{\Cat_\Ring}(P_{\T^0})^{opp}$.
The quotient functor $\pi:\Cat_\Ring\twoheadrightarrow \Cat_{\T^0,\Ring}$ can be realized
as $\Hom_{\Cat_\Ring}(P_{\mathcal{T}^0},\bullet)$. It has left adjoint (and right inverse) functor
$\pi^!=P_{\mathcal{T}^0}\otimes_{A_{\T^0,\Ring}}\bullet$.

For $\tau\in \T^0$, set $\Delta_{\T^0,\Ring}(\tau):=\pi(\Delta_{\Ring}(\tau))$.

\begin{Lem}\label{Lem:adjoint_standard}
The natural morphism $\pi^!(\Delta_{\T^0,\Ring}(\tau))\rightarrow \Delta_\Ring(\tau)$ is an isomorphism.
Moreover, $L^i\pi^!\Delta_{\T^0,\Ring}(\tau)=0$ for $i>0$.
\end{Lem}
\begin{proof}
By (v), we can write a projective resolution of $\Delta_\Ring(\tau)$ by objects in
$\Cat_\Ring^{\Delta,\T^0}$. For such an object $P$, we have that $\pi(P)$ is a projective
in $\Cat_{\T^0,\Ring}$, and $L\pi^!\circ \pi(P)\xrightarrow{\sim} P$. The claim of the lemma follows.
\end{proof}

It follows that $\pi$ and $\pi^!$ restrict to mutually inverse equivalences
between $\Cat_\Ring^{\Delta,\T^0}$ and $\Cat_{\T^0,\Ring}^\Delta$. It is easy to see that
$\Cat_{\T^0,\Ring}$ is a highest weight category with poset $\T^0$ and standard objects
$\Delta_{\T^0,\Ring}(\tau)$.

We also note that $D^b(\Cat_{\T^0,\Ring})$ is the quotient category of $D^b(\Cat_\Ring)$
by $D^b(\Cat_{\T_0,\Ring})$.

\subsubsection{Base change}
Note that if $\Ring'$ is a Noetherian $\Ring$-algebra then the base change $\Cat_{\Ring'}$
of $\Cat_\Ring$ (defined as $A_{\Ring'}\operatorname{-mod}$) is also a highest weight
category, \cite[Proposition 4.14]{rouqqsch}. It has the same poset $\T$. The standard
(resp., costandard) objects are $\Delta_{\Ring'}(\tau):=\Ring'\otimes_{\Ring}\Delta_\Ring(\tau)$,
(resp., $\nabla_{\Ring'}(\tau):=\Ring'\otimes_{\Ring}\nabla_\Ring(\tau)$).

%\subsubsection{Example}
%Let $\Ring$ be a ring and $z\in \Ring$.
%
%Consider the free $\Ring$-algebra $A_\Ring$
%with basis $e_1,e_2,x_{12},x_{21},x_{22}$ and multiplication recovered from
%$$e_1e_2=0, e_1+e_2=1, e_1x_{12}=x_{12}e_2=x_{12}, e_2x_{21}=x_{21}e_1=x_{22},
%x_{21}x_{12}=x_{22}, x_{12}x_{21}=z e_{11}.$$
%The $A_\Ring$-modules $P_\Ring(\tau_i):=A_\Ring e_i$ are projective. Note that the
%right multiplication by $x_{12}$ defines an $A_\Ring$-module embedding
%$P_\Ring(\tau_1)\hookrightarrow P_\Ring(\tau_2)$. The cokernel is a free rank
%one $\Ring$-module to be denoted by $\Delta_\Ring(\tau_2)$. Set $\Delta_\Ring(\tau_1):=P_\Ring(\tau_1)$.
%Note that the endomorphisms of $\Delta_\Ring(\tau_i)$ are identified with $\Ring$.
%The embedding $P_\Ring(\tau_1)\hookrightarrow P_\Ring(\tau_2)$ is a basis in
%$\Hom_{A_\Ring}(P_\Ring(\tau_1),P_\Ring(\tau_2))$ and so $\Hom_{A_\Ring}(\Delta_\Ring(\tau_1),\Delta_\Ring(\tau_2))=\{0\}$. From here it is easy to
%see that $A_\Ring\operatorname{-mod}$ is a highest weight category with poset $\{\tau_1>\tau_2\}$.
%
%We will mostly be interested in the situation when $\Ring$ is a DVR and $z$ is a parameter.

\subsection{Coideal finite posets}\label{SS_hw_coideal_finite}
In this section we consider a generalization of highest weight categories from Section \ref{SS_hw_finite}:
highest weight categories associated to coideal finite posets following \cite[Section 6.1.2]{EL}.

\subsubsection{Definition}\label{SSS_defi_coideal_finite_hw}
Let $\Cat_\Ring$ be an $\Ring$-linear abelian category. Suppose that $\Cat_\Ring$ satisfies the following conditions.

\begin{itemize}
\item[(I)] The Hom modules in $\Cat_\Ring$ are finitely generated over $\Ring$, and the Hom modules
between projective objects are projective $\Ring$-modules;
\item[(II)] and $\Cat_\Ring$ is Noetherian.
\end{itemize}

We say that an object $M_\Ring\in \Cat_\Ring$ is {\it projective over $\Ring$}
if $\Hom_{\Cat_\Ring}(P,M)$ is projective over $\Ring$ for all projectives $P$.
For example, every projective in $\Cat_\Ring$ is projective over $\Ring$.

\begin{defi}\label{defi:coideal_finite}
We say that a poset $\T$ is {\it coideal finite} if for all $\tau\in \T$, the poset coideal
$\{\tau'\in \T| \tau'\geqslant \tau\}$ is finite. 
\end{defi}

\begin{defi}\label{defi:hw_coideal_finite}
Let $\Cat_\Ring$ be a category satisfying (I) and (II) above and $\T$ be a coideal finite poset.
By the structure of a highest weight category on $\Cat_\Ring$ we mean a collection
of {\it standard objects} $\Delta_\Ring(\tau)\in \Cat_\Ring$ satisfying axioms
(i)-(v) from Section \ref{SSS_hw_finite_def}.
\end{defi}

Thanks to (iv) and (v), a highest weight category automatically has enough projectives.

\begin{Rem}\label{Rem:BS1}
Similar categories (under the name ``upper finite highest weight categories'') were considered in 
\cite{BS}. However, there are some differences: for example, our categories are Noetherian, 
unlike in \cite{BS}. Also, we work over general rings, while the categories in \cite{BS}
are over (algebraically closed) fields. Because of this differences, we decided to give 
self-contained proofs of various technical statements, many of which have analogs in 
\cite{BS}.
\end{Rem}

\subsubsection{Locally unital algebras and their modules}\label{SSS_loc_unit}
By definition, highest weight categories associated to finite posets are equivalent to
the categories of modules over finite projective $\Ring$-algebras. We need an analog of
this result in the coideal finite setting. 

Let $A_\Ring$ be an associative (but not necessarily unital)
$\Ring$-algebra. Let $\T$ be a set. Assume that $A_\Ring$ comes with a collection of orthogonal idempotents $e_\tau, \tau\in \T$. For a finite subset $\T'\subset \T$, we write $e_{\T'}$
for $\sum_{\tau\in \T'}e_\tau$.
Assume that the following conditions are satisfied:
\begin{itemize}
\item[($i$)] For each $a\in A_\Ring$, there is a finite subset $\T'\subset\T$ depending
on $a$ such that $e_{\T'}a=ae_{\T'}=a$. We say that $A_\Ring$ is {\it locally unital}.
\item[($ii$)] for each finite subset $\T'\subset \T$, the $\Ring$-module $e_{\T'}A_{\Ring}e_{\T'}$
is a finitely generated projective $\Ring$-module. We say that $A_\Ring$ is {\it locally finite projective}.
\end{itemize}
Note that if $\T$ itself is finite, then $A_\Ring$ is just a finite projective $\Ring$-algebra.

\begin{Ex}\label{Ex:loc_fin_proj_algebra}
Let $\Cat_\Ring$ satisfy (I)-(II). Let  $P_{\tau}$, where $\tau$ is in
some indexing set $\T$, be a collection of projective objects. Take $A_\Ring:=\left(\bigoplus_{\tau,\tau'}\Hom_{\Cat_\Ring}(P_{\Ring,\tau},P_{\Ring,\tau'})\right)^{opp}$,
this is an associative $\Ring$-algebra satisfying (i) and (ii) (where $e_\tau$ is the identity endomorphism of $P_{\Ring,\tau}$).
\end{Ex}

We get back to the situation of a general locally finite and locally projective $\Ring$-algebra
$A_\Ring$. We consider the category $A_\Ring\operatorname{-Mod}_{lf}$ (where ``lf'' is for ``locally finite'') 
consisting of all $A_\Ring$-modules $M$
satisfying the following two conditions:
\begin{itemize}
\item[($i'$)] $M$ is {\it locally unital} meaning that for each $m\in M$, there is a finite subset
$\T'\subset\T$ such that $e_{\T'}m=m$.
\item[($ii'$)] $M$ is {\it locally finite} meaning that for every finite subset $\T'$, the $\Ring$-module
$e_{\T'}M$ is finitely generated.
\end{itemize}

By $A_\Ring\operatorname{-mod}$ we denote the category of Noetherian objects in $A_\Ring\operatorname{-Mod}_{lf}$. The category $A_\Ring\operatorname{-mod}$
satisfies conditions (I) and (II) from Section \ref{SSS_defi_coideal_finite_hw}. It  has enough projectives
provided $A_\Ring e_\tau$ is a Noetherian $A_\Ring$-module for all $\tau$. The following lemma
is a partial converse of this claim. 

Let $\Cat_\Ring$ be as in Section \ref{SSS_defi_coideal_finite_hw} and  $A_\Ring$ be as in Example \ref{Ex:loc_fin_proj_algebra}.
Suppose that every object in $\Cat_\Ring$ is a quotient of a finite direct sum of $P_\tau$'s. 
We have the functor $\mathcal{F}: \Cat_\Ring\rightarrow A_\Ring\operatorname{-Mod}_{lf}$
sending $M\in \Cat_\Ring$ to $\mathcal{F}(M):=\bigoplus_{\tau\in \T}\Hom_{\Cat_\Ring}(P_\tau,M)$.

\begin{Lem}\label{Lem:full_embedding_coideal_finite}
The functor $\mathcal{F}$ is a full embedding whose essential image is $A_\Ring\operatorname{-mod}$.
In particular, $A_\Ring e_\tau$ is Noetherian for each $\tau$.
\end{Lem}
\begin{proof}
We first show that the essential image of $\mathcal{F}$ is contained in $A_\Ring\operatorname{-mod}$.
It is enough to show that $A_\Ring e_\tau=\bigoplus_{\tau'}\Hom_{\Cat_\Ring}(P_{\tau'},P_{\tau})$ is Noetherian. Let $N\subset A_\Ring e_\tau$ be an $A_\Ring$-submodule. Note that
$N=\bigoplus_{\tau'\in \T}e_{\tau'}N$. Since $P_{\tau}$ is Noetherian, we can find a finite subset $\T'\subset \T$ such that
$\im \varphi\subset Q:=\sum_{\psi\in e_{\T'}N}\im \psi$ for all $\tau'\in \T, \varphi\in \Hom_{\Cat_\Ring}(P_{\tau'},N)=e_{\tau'}N$.
We have an epimorphism $\widetilde{\varphi}:\bigoplus_{\tau'\in \T'}P_{\tau'}
\twoheadrightarrow Q$, and $\varphi\in \Hom_{\Cat_\Ring}(P_{\tau'},P_{\tau})$
factors through $\widetilde{\varphi}$. Equivalently, $\varphi$ lies in the $A_\Ring$-submodule
of $N$ generated by $e_{\T'}N$. The latter is a submodule of $e_{\T'}P_{\tau}$, a finitely
generated $\Ring$-module by (I), so is finitely generated. It follows that $N$ is a Noetherian $A_\Ring$-module.

Now we are going to construct a quasi-inverse functor $\mathcal{G}:A_\Ring\operatorname{-mod}
\rightarrow \Cat_\Ring$. It is given by $M\mapsto (\bigoplus_{\tau\in \T}P_{\tau})\otimes_{A_\Ring}M$. The target is viewed as an object of $\Cat_\Ring$ as follows.
If $M=A_\Ring e_\tau$, then it is $P_{\tau}$. In general, $M$ is the cokernel of a morphism
of objects of the form $\bigoplus_{i=1}^k A_\Ring e_{\tau_i}$. Then $\mathcal{G}(M)$ is the cokernel of
the corresponding morphism of the objects of the form $\bigoplus_{i=1}^k P_{\tau_i}$. One defines
$\mathcal{G}$ on morphisms in a similar way. A check that $\mathcal{F}$ and $\mathcal{G}$
are quasi-inverse is standard and is left as an exercise.
\end{proof}

\subsubsection{Deformations and Noetherian property}
Let $\Ring$ be a complete regular local ring and $\kf$ be its residue field. Let $A_\Ring,\T$ be such as in  
Section \ref{SSS_loc_unit}. Suppose that (i) and (ii) hold. Set $A_\kf:=\kf\otimes_{\Ring}A_\Ring$. 

\begin{Lem}\label{Lem:deformation_Noetherian}
If $A_\kf e_\tau$ is a Noetherian $A_\kf$-module, then $A_\Ring e_{\tau}$ is a Noetherian $A_\Ring$-module. 
\end{Lem}
\begin{proof}
Let $N_\Ring$ be an $A_\Ring$-submodule of $A_\Ring e_\tau$. 
Take $x\in \mathfrak{m}\setminus \mathfrak{m}^2$, where $\mathfrak{m}$ is the maximal ideal in $\Ring$,
 and set $\Ring_0:=\Ring/(x)$, a regular local ring of dimension one less.
By induction, we can assume that $A_{\Ring_0}e_\tau$ is a Noetherian $A_{\Ring_0}$-module. 
For $i\geqslant 0$, set $N^i_\Ring:=\{m\in A_\Ring e_\tau| x^i m\in N_\Ring\}$.
Let $N^i_{\Ring_0}$ denote the image of $N^i_\Ring$ in $A_{\Ring_0}e_\tau$. This is an ascending 
chain of submodules in $A_{\Ring_0}e_\tau$ so it terminates: there is $j$ such that $N^i_{\Ring_0}:=N^j_{\Ring_0}$
for all $i>j$. By the inductive assumption, each  $N^i_{\Ring_0}$ is finitely generated.  More precisely, 
we can find elements $m_k^i\in N^i_\Ring, k=1,\ldots,d_i,$ such that, for each $i=0,\ldots,j$, the images of 
$m_k^i$ in $A_{\Ring_0}e_\tau$ generate $N^i_{\Ring_0}$. Note that the elements $x^i m_k^i$ lie in $N_\Ring$. 
We claim that they generate the $A_\Ring$-module $N_\Ring$. Let $n\in N_\Ring\cap x^\ell A_\Ring e_\tau$. Choose a finite subset $\T'\subset \T$
such that $e_{\T'}n=n, e_{\T'}m_k^i=m_k^i$.
By the construction, we can find an $e_{\T'}A_\Ring e_{\T'}$-linear combination $n'$ of the elements $x^i  m_k^i$ such that 
$n-n'\in x^{\ell+1} e_{\T'}A_\Ring e_\tau$. Moreover, for $\ell>j$, we can assume that all coefficients are in 
$x^{\ell-j}e_{\T'}A_\Ring e_{\T'}$. Since $\Ring$ is complete, and $e_{\T'}A_\Ring e_{\T'}$
is a finitely generated $\Ring$-module, this shows that the element $n$ lies in the 
$A_\Ring$-linear span of the elements $x^i m_k^i$. From here we easily deduce that $A_\Ring e_\tau$
is Noetherian.
%Assume the contrary. Let $\ell$ be minimal such that there is $n\in N_\Ring\cap t^\ell (A_\Ring e_\tau)$ that does not lie in
%$N'_\Ring$. 
\end{proof}

\subsubsection{Highest weight quotients}
Let $\Cat_\Ring$ be a highest weight category with a coideal finite poset $\T$. For each $\tau\in \T$, choose a
projective object $P_\tau$ as in axiom (v) from Section \ref{SS_hw_finite}. Form the algebra $A_\Ring$ with idempotents $e_\tau$ as
in Example \ref{Ex:loc_fin_proj_algebra}. We identify $\Cat_\Ring$ with $A_\Ring\operatorname{-mod}$
as in Lemma \ref{Lem:full_embedding_coideal_finite}.

Let $\T^0$ be a coideal of $\T$. The algebra
$A_{\T^0,\Ring}:=\bigoplus_{\tau,\tau'\in \T^0}e_\tau A_{\Ring}e_{\tau'}$  satisfies conditions (i) and (ii) from Section 
\ref{SSS_loc_unit}, moreover, $A_{\T^0,\Ring}e_\tau$ is a Noetherian $A_{\T^0,\Ring}$-module for all $\tau\in \T^0$.
Set $\Cat_{\T^0,\Ring}:=A_{\T^0,\Ring}\operatorname{-mod}$. Then we get an exact functor
$\pi(=\pi_{\T^0}):\Cat_\Ring\rightarrow \Cat_{\T^0,\Ring}, M\mapsto \bigoplus_{\tau\in \T^0}e_\tau M$. It admits a
left adjoint (and right inverse) functor $\pi^!: \Cat_{\T^0,\Ring}\rightarrow \Cat_\Ring$, it is given by
$(\bigoplus_{\tau\in \T^0} A_\Ring e_\tau)\otimes_{A_{\T^0,\Ring}}\bullet$.

For $\tau\in \T^0$, set $\Delta_{\T^0,\Ring}(\tau):=\pi(\Delta_\Ring(\tau))$.

\begin{Lem}\label{Lem:adjoint_standard_coideal_finite}
The following claims are true:
\begin{enumerate}
\item $\pi(\Delta_\Ring(\tau))=0$ if $\tau\not\in \T^0$.
\item
The natural morphism $\pi^!(\Delta_{\T^0,\Ring}(\tau))\rightarrow \Delta_\Ring(\tau)$ is an isomorphism.
Moreover, $L^i\pi^!\Delta_{\T^0,\Ring}(\tau)=0$ for $i>0$.
\end{enumerate}
\end{Lem}
\begin{proof}
(1) follows from axioms (iii) and (v) in Section \ref{SSS_hw_finite_def}. (2) is proved in the same way as
Lemma \ref{Lem:adjoint_standard}.
\end{proof}

As in Section \ref{SSS_hw_quot}, from this lemma we deduce that:
\begin{itemize}
\item The category $\Cat_{\T^0,\Ring}$ is highest weight with poset $\T^0$
and standard objects $\Delta_{\T^0,\Ring}(\tau), \tau\in \T^0$.
\item The functor $\pi$ and $\pi^!$ are mutually quasi-inverse equivalences
between $\Cat_\Ring^{\Delta, \T^0}$ (defined analogously to the finite case, see Section \ref{SSS_hw_finite_def}) and $\Cat_{\T_0,\Ring}^\Delta$.
\end{itemize}

\subsubsection{Standardly filtered objects}\label{SSS_coideal_finite_standard}
Passing to the highest weight subquotients associated to finite poset coideals is our main tool to
study highest weight categories with coideal finite posets. We start with the following lemma.

\begin{Lem}\label{Lem:stand_charact}
Let $M\in \Cat_\Ring$. The following conditions are equivalent:
\begin{itemize}
\item[(a)] $M\in \Cat_\Ring^\Delta$.
\item[(b)] $\pi_{\T^0}(M)\in \Cat_{\T^0,\Ring}^\Delta$ for all finite coideals $\T^0\subset \T$.
\item[(c)] There is a finite coideal $\T^0(M)$ such that $\pi_{\T^0}(M)$ is standardly filtered
for all finite coideals $\T^0$ containing $\T^0(M)$.
\end{itemize}
Moreover, we can find  $\T^0(M)$ as in (c), such that $\pi_{\T_0}^!\pi_{\T_0} M\xrightarrow{\sim} M$
for all $\T^0$ containing $\T^0(M)$.
\end{Lem}
\begin{proof}
(a) implies (b) thanks to the definition of $\Delta_{\T^0,\Ring}(\tau)$ and (1) of Lemma
\ref{Lem:adjoint_standard_coideal_finite}. (b)$\Rightarrow$(c) is a tautology.

Now we prove (c)$\Rightarrow$(a). Thanks to axioms (iv) and (v), we can find an exact sequence
$P_1\rightarrow P_0\rightarrow M\rightarrow 0$, where $P_1,P_2$ are direct sums of projectives $P_\tau$. We can enlarge $\T^0(M)$ and assume that it contains all labels
of standards that occur in $P_1,P_0$. Set $\T^0:=\T^0(M)$ and consider the corresponding
functors $\pi,\pi^!$. Thanks to (2) of Lemma \ref{Lem:adjoint_standard_coideal_finite},
$\pi^!\pi(P_i)\cong P_i, i=0,1$. By the 5-lemma, $\pi^!\pi(M)\xrightarrow{\sim} M$.
The source is standardly filtered, again by (2) of Lemma \ref{Lem:adjoint_standard_coideal_finite}.
\end{proof}

Using Lemma \ref{Lem:stand_charact}, we can carry over results about standardly filtered objects from the case of finite
posets to the case of coideal finite posets:

\begin{itemize}
\item[(A)] $\Cat_\Ring^\Delta$ is closed under taking direct summands, see Lemma \ref{Lem:stand_exact_prop}.
\item[(B)] The kernel of an epimorphism of standardly filtered objects is standardly filtered,
Lemma \ref{Lem:stand_exact_prop}.
\item[(C)] On any standardly filtered
object $M$, we have a canonical filtration $M_{\T^0}$ indexed by finite poset coideals. A direct analog of Lemma
\ref{Lem:stand_filtration} holds.
\end{itemize}

\subsubsection{Extensions}
We have the following lemma that reduces the computation of $\Ext's$ in $\Cat_\Ring$ to
the computation in highest weight quotients associated to finite coideals. Fix $M,N\in \Cat_\Ring$ and $i\geqslant 0$.
Take a finite coideal $\T^0$. Then we have an $\Ring$-linear map
$\psi_{\T_0}: \Ext^i_{\Cat_\Ring}(M,N)\rightarrow \Ext^i_{\Cat_{\T^0,\Ring}}(\pi_{\T^0}M,\pi_{\T^0}N)$.

\begin{Lem}\label{Lem:Ext_comput}
Let $M\in \Cat_\Ring$ and $i\geqslant 0$. Then there is a finite coideal $\T^{0}(M,i)\subset \T$
such that $\psi_{\T_0}$ is an isomorphism for all $N\in \Cat_\Ring$ and all finite coideals $\T^0$
that contain $\T^0(M,i)$.
\end{Lem}
\begin{proof}
The category $\Cat_\Ring$ is Noetherian. So, thanks to axiom (v), we can find an epimorphism
$P_0\rightarrow M$, where $P_0$ is a finite direct sum of objects of the form $P_\tau$. Of course,
we can find a resolution $\rightarrow P_i\rightarrow P_{i-1}\rightarrow\ldots\rightarrow P_0$
of $M$ by the same kind of objects. Take $\T^0(M,i):=\T^0(P_{i+1})\cup \T^0(P_i)\cap \T^0(P_{i-1})$,
where in the right hand side we have coideals from (c) of Lemma \ref{Lem:stand_charact}.
For $j\in \{i-1,i,i+1\}$, and any $\T^0$ containing $\T^0(M,i)$, we have
$\pi^!_{\T^0}\pi_{\T^0}(P^j)\xrightarrow{\sim} P^j$. Hence
$$\Hom_{\Cat_\Ring}(P^j,N)\xrightarrow{\sim} \Hom_{\Cat_{\T^0,\Ring}}(\pi_{\T^0}P^j,\pi_{\T^0}N).$$
It follows that $\psi_{\T^0}$ is an isomorphism.
\end{proof}

\subsubsection{Projectives}
Using Lemma \ref{Lem:Ext_comput}, we can deduce some properties of projective objects
in $\Cat_\Ring$, some of which mirror the corresponding properties in the finite case.

\begin{Cor}\label{Cor:coideal_finite_projectives}
The following claims are true:
\begin{enumerate}
\item An object $P\in \Cat_\Ring$ is projective if and only if there is a finite coideal $\T^1(P)\subset \T$ such that $\pi_{\T^0}(P)\in \Cat_{\T^0,\Ring}$ is projective for all finite coideals $\T^0$
    containing $\T^1(P)$.
\item Every projective in $\Cat_\Ring$ is standardly filtered.
\item A standardly filtered object $P\in \Cat_\Ring$ is projective if and only if $\Ext^1_{\Cat_\Ring}(P,\Delta_\Ring(\tau))=0$ for all $\tau\in \T$.
\end{enumerate}
\end{Cor}
\begin{proof}
To prove (1) apply Lemma \ref{Lem:Ext_comput} to $P$ and $i=1$. We get $\Ext^1_{\Cat_\Ring}(P,N)\xrightarrow{\sim}\Ext^1_{\Cat_{\T^0,\Ring}}(\pi_{\T^0}P,\pi_{\T^0}N)$
for all $N\in \Cat_\Ring$ and all $\T^0$ containing $\T^0(P,1)$ from Lemma \ref{Lem:Ext_comput}.
If $P$ is projective,
then the source is zero for all $N$, hence the target is zero for all $N$. Since $\pi_{\T^0}$
is essentially surjective, this implies $\pi_{\T^0}P$ is projective. Conversely, if $\pi_{\T^0}P$
is projective for all $\T^0$ containing $\T^1(P)$, then the target is zero for all $\T^0$
containing $\T^1(P)\cup\T^0(P,1)$. It follows that $P$ is projective, finishing the proof of (1).

We proceed to (2). Take $\T^0$ containing $\T^0(P,0)$ from Lemma \ref{Lem:Ext_comput}
and $\T^1(P)$ from (a). The object $\pi_{\T^0}P$ is projective by (a), and the objects
$\pi_{\T^0}^!\pi_{\T^0}P$ and $P$ represent the same functor, $\Hom_{\Cat_{\T^0,\Ring}}(\pi_{\T^0}P,\pi_{\T^0}\bullet)$. Hence they are isomorphic.
Since $\pi_{\T^0}^!\pi_{\T^0}P$ is standardly filtered, we get (2).

(3) is left as an exercise using Lemma \ref{Lem:Ext_comput} and Lemma \ref{Lem:proj_charact}.
\end{proof}

\begin{Rem}\label{Rem:coideal_finite_stand_to_all}
Note that part (3) of Corollary \ref{Cor:coideal_finite_projectives} allows to recover the category of projective objects $\Cat_\Ring$
from $\Cat_\Ring^\Delta$. Then one can recover $\Cat_\Ring$ from the category of projectives,
see Lemma \ref{Lem:full_embedding_coideal_finite}. 
A precise statement to be used below is as follows. Suppose that $\Cat^1_\Ring,\Cat^2_\Ring$ are two highest weight categories with poset $\T$ and standard
objects $\Delta^1_\Ring(\tau),\Delta^2_\Ring(\tau),\tau\in \T$. Suppose that we have
an $\Ring$-linear equivalence of exact categories $\Cat_\Ring^{1,\Delta}\xrightarrow{\sim}
\Cat_{\Ring}^{2,\Delta}$ that sends $\Delta^1_\Ring(\tau)$ to $\Delta^2_\Ring(\tau)$. Then
there is an equivalence $\Cat^1_\Ring\xrightarrow{\sim}\Cat^2_\Ring$ of $\Ring$-linear abelian
categories extending the equivalence $\Cat^{1,\Delta}_\Ring\xrightarrow{\sim}\Cat^{2,\Delta}_\Ring$.
\end{Rem}

\subsubsection{Highest weight subcategories}
Let $\T^0$ be a coideal in $\T$, and $\T_0:=\T\setminus \T^0$ be its complement.
%Let $A_\Ring$ be a locally unital locally $\Ring$-projective algebra with $\Cat_{\Ring}$
%We can form the locally unital locally projective algebra $A_$

Let $\Cat_{\T_0,\Ring}$ denote the kernel of the quotient functor $\pi_{\T^0}$. This is a
Serre subcategory of $\Cat_\Ring$. The following lemma summarizes properties of $\Cat_{\T_0,\Ring}$.

\begin{Lem}\label{Lem:hw_sub_coideal_finite}
The following claims hold.
\begin{enumerate}
\item $\Cat_{\T_0,\Ring}$ is the Serre span of $\Delta_\Ring(\tau)$ with $\tau\in \T_0$.
\item $\Cat_{\T_0,\Ring}$ is a highest weight category with poset $\T_0$ and standard
objects $\Delta_\Ring(\tau),\tau\in \T_0$.
\item Let $\iota$ denote the inclusion functor $\Cat_{\T_0,\Ring}\hookrightarrow\Cat_{\Ring}$.
Then it has the left adjoint functor $\iota^!$. This functor sends $\Delta_{\Ring}(\tau)$ to itself if $\tau\in \T_0$ and to
$0$ else. Moreover, $\iota^!$ is acyclic on standard objects.
\item The inclusion functor $D^b(\Cat_{\T_0,\Ring})\hookrightarrow D^b(\Cat_\Ring)$
is a full embedding.
\end{enumerate}
\end{Lem}
\begin{proof}
For $\tau\in\T_0$, set $\underline{P}_\tau:=P_\tau/(P_\tau)_{\T^0}$ (see (C)
in Section \ref{SSS_coideal_finite_standard}),
this is an object filtered by $\Delta_\Ring(\tau')$ with $\tau'\in \T_0$. Clearly, it is projective in $\Cat_{\T_0,\Ring}$ and every object in $\Cat_{\T_0,\Ring}$ is a quotient of a finite direct
sum of projects of the form $\underline{P}_\tau$. (1) follows. (2) is easy to check. (3) is proved
as Lemma \ref{SSS_coideal_finite_standard}.

To prove (4), we show that $\Ext^i_{\Cat_{\T_0,\Ring}}(M,N)\xrightarrow{\sim}
\Ext^i_{\Cat_{\Ring}}(M,N)$ for all $M,N\in \Cat_{\T_0,\Ring}$. This follows from
Lemma \ref{Lem:Ext_comput} that reduces the claim to the case of finite posets (note that for each finite coideal $\T_0'\subset \T_0$
there is a finite coideal $\T'\subset \T$ with $\T'\cap\T_0=\T'_0$) combined with Lemma \ref{Lem:full_derived_embedding} treating that case.
\end{proof}

\begin{Rem}\label{Rem:hw_sub_recovery_coideal_finite}
The direct analog of Remark \ref{Rem:hw_sub_recovery} holds. 
\end{Rem}

\subsection{Interval finite posets}
In this section we consider a more general class of highest weight categories (and, in fact, the class
we need in the present paper): those associated to interval finite posets.

\subsubsection{Definitions}
Let $\T$ be a poset.

\begin{defi}\label{defi:interval_finite_poset}
We say that $\T$ is {\it interval finite} if $\T$ can be represented as the union of poset ideals that are coideal finite as posets. 
This is equivalent to the condition that all poset intervals $[\tau_1,\tau_2]:=\{\tau\in \T| \tau_1\leqslant \tau\leqslant \tau_2\}$
are finite. 
\end{defi}

% 
%
%\begin{defi}\label{defi:strong_int_finite}
%We say that $\T$ is {\it strongly interval finite} if it can be represented as $\bigcup_{i\in \Z}\T^0_i$,
%where $\T^0_i$ are coideal finite posets with $\T^0_i\subset \T^0_{i+1}$. 
%\end{defi}

\begin{Ex}\label{Ex:poset_highest_weights}
Let $\T$ be the character lattice of a maximal torus in a  semisimple group $G$. We equip it with the standard dominance
order: $\tau\leqslant \tau'$ if $\tau'-\tau$ is a sum of positive roots. Then $\T$ is  
interval finite. For example, when $G$ is adjoint, then $\T$ is the union of the posets ideals of the form $\T_i:=\{\lambda\in \Lambda_0| \langle\rho^\vee,\lambda\rangle\leqslant i\}$ for $i\in \Z$ (where $\rho^\vee$ is half the sum of positive coroots), and these poset ideals are coideal finite. 
\end{Ex}

We note that the poset ideal $\T(\leqslant \tau)=\{\tau'\in \T| \tau'\leqslant \tau\}$ in an interval finite poset $\T$ is coideal finite for all $\tau$.

Let $\Ring$ be a Noetherian ring and $\Cat_\Ring$ be a Noetherian $\Ring$-linear category whose $\Hom$'s
are finitely generated over $\Ring$. Suppose $\T$ is an interval finite poset and we have a family
$\Delta_\Ring(\tau),\tau\in \T,$ of objects in $\Cat_\Ring$. For a coideal finite poset ideal
$\T_0\subset \T$ define the Serre subcategory $\Cat_{\T_0,\Ring}$ as the Serre span of
$\Delta_\Ring(\tau),\tau\in \T_0$.

\begin{defi}\label{defi:hw_interval_finite}
The category $\Cat_\Ring$ equipped with the objects $\Delta_\Ring(\tau),\tau\in \T,$ is said to be {\it highest weight} if the following conditions are satisfied:
\begin{itemize}
\item
for each coideal finite poset ideal $\T_0\subset \T$, the category $\Cat_{\T_0,\Ring}$
is highest weight with poset $\T_0$ and standard objects $\Delta_\Ring(\tau),\tau\in \T_0$;
\item Every object of $\Cat_\Ring$ is contained in $\Cat_{\T_0,\Ring}$ for some coideal finite
poset ideal $\T_0\subset \T$.
\end{itemize}

\begin{Rem}\label{Rem:BS2}
The notions of interval finite posets and the corresponding highest weight categories also appeared in 
\cite{BS}. Their definition of an interval finite highest weight categories
is too restrictive for our purposes. 
\end{Rem}

By a {\it standardly filtered object} in $\Cat_\Ring$ we mean an object that is standardly filtered in one of $\Cat_{\T_0,\Ring}$. The category of standardly filtered objects in $\Cat_\Ring$ will be denoted by
$\Cat_\Ring^\Delta$. By an {\it equivalence} of highest weight categories $\Cat^1_\Ring,\Cat^2_\Ring$
we mean an equivalence of $\Ring$-linear abelian categories $\Cat^1_\Ring\xrightarrow{\sim}\Cat^2_\Ring$
that sends standard objects to standard objects (automatically inducing a bijection between their labels).
\end{defi}

Below we will see that the two categories studied in this paper: the deformed quantum category $\OCat$ and the heart of the stabilized t-structure on the affine Hecke category are highest weight categories associated to interval finite posets (related to those from 
Example \ref{Ex:poset_highest_weights}).

\begin{Rem}\label{Rem:interval_finite_stand_to_all}
Note that $\Cat_\Ring^\Delta$ is an exact category. Thanks to Remark \ref{Rem:coideal_finite_stand_to_all}, one can recover $\Cat_\Ring$
from $\Cat_\Ring^\Delta$ in the same sense as in Remark \ref{Rem:coideal_finite_stand_to_all}:
an $\Ring$-linear equivalence $\Cat^{1,\Delta}_\Ring\xrightarrow{\sim}\Cat^{2,\Delta}_\Ring$
of exact categories sending
$\Delta^1_\Ring(\tau)$ to $\Delta^2_\Ring(\tau)$ for all $\tau\in \T$ extends to an
$\Ring$-linear equivalence of abelian categories $\Cat^1_\Ring\xrightarrow{\sim}\Cat^2_\Ring$.
\end{Rem}

\begin{Rem}\label{Rem:ideal_finite}
Suppose that $\T$ is {\it ideal finite} meaning that $\{\tau'\in \T|\tau'\leqslant \tau\}$
is finite for all $\tau$. Then we recover the definition of a highest weight category associated
to an ideal finite poset, \cite[Section 6.1.3]{EL}. Note that it makes sense to speak about
costandard and tilting objects in the corresponding highest weight category.
\end{Rem}

\subsubsection{Full embedding}
Below we will need the following result.

\begin{Lem}\label{Lem:derived_full_embedding_interval finite}
In the notation of the previous section, let $\T_0\subset\T$ be a coideal finite poset ideal.
Then the inclusion functor $D^b(\Cat_{\T_0,\Ring})\rightarrow D^b(\Cat_{\Ring})$ is a full
embedding.
\end{Lem}
\begin{proof}
Again, we need to show that $\Ext^i_{\Cat_{\T_0,\Ring}}(M,N)\xrightarrow{\sim}\Ext^i_{\Cat_\Ring}(M,N)$
for all $M,N\in \Cat_{\T_0,\Ring}$ and all $i\geqslant 0$. By Lemma \ref{Lem:Ext_comput},
$\Ext^i_{\Cat_{\T_0,\Ring}}(M,N)\xrightarrow{\sim} \Ext^i_{\Cat_{\T'_0,\Ring}}(M,N)$ for all
coideal finite posets $\T_0'$ containing $\T_0$. On the other hand, $\Ext^i_{\Cat_\Ring}(M,N)$
is the colimit of $\Ext^i_{\Cat_{\T_0',\Ring}}(M,N)$ taken over all coideal finite poset ideals
$\T_0'$ such that $M,N\in \Cat_{\T_0',\Ring}$, this can be seen, for example, from the Yoneda description of $\Ext^i$. We see that indeed $\Ext^i_{\Cat_{\T_0,\Ring}}(M,N)\xrightarrow{\sim}\Ext^i_{\Cat_\Ring}(M,N)$.
\end{proof}

\subsubsection{Modules over algebras}\label{SSS_interval_finite_recovery}
In this part we explain how to realize a highest weight category with an interval finite poset in terms of modules 
over a family of algebras. 

Using  Remark \ref{Rem:proj_construction}, we produce projective objects $P^{\T_0}_{\Ring,\tau}\in \Cat_{\T_0,\Ring}$
for every coideal finite poset ideal $\T_0\subset \T$ and every $\tau\in \T_0$. Then we have epimorphisms
$\eta_{\T_0^1,\T_0^2}:P^{\T_0^2}_{\Ring,\tau}\twoheadrightarrow P^{\T_0^1}_{\Ring,\tau}$
for each pair $\T_0^1\subset \T_0^2$ of coideal finite poset ideals.

Form the locally unital algebra $A_{\Ring,\T_0}:=\left(\bigoplus_{\tau,\tau'\in \T_0}\Hom_{\Cat_\Ring}(P^{\T_0}_{\Ring,\tau},P_{\Ring,\tau'}^{\T_0})\right)^{opp}$. Similarly to 
Remark \ref{Rem:hw_sub_recovery}, we get algebra epimorphisms
$\varpi_{\T_0^1,\T_0^2}:A_{\Ring,\T_0^2}\twoheadrightarrow A_{\Ring,\T_0^1}$. There is the following compatibility: for any $\varphi\in A_{\Ring,\T_0^2}$, we have $\eta_{\T_0^1,\T_0^2}\circ \varphi=
\varpi_{\T_0^1,\T_0^2}(\varphi)\circ \eta_{\T_0^1,\T_0^2}$. Also, we write 
$\iota_{\T_0^2,\T_0^1}$ for the inclusion functor $\Cat_{\T_0^1,\Ring}\hookrightarrow \Cat_{\T_0^2,\Ring}$.

Moreover, for $\T_0^1\subset\T_0^2\subset \T_0^3$, we have 
$$\varpi_{\T_0^1,\T_0^3}=\varpi_{\T_0^1,\T_0^2}\circ \varpi_{\T_0^2,\T_0^3}, 
\eta_{\T_0^1,\T_0^3}=\eta_{\T_0^1,\T_0^2}\circ \eta_{\T_0^2,\T_0^3}, \iota_{T_0^3,\T_0^1}=\iota_{\T_0^3,\T_0^2}\circ \iota_{\T_0^2,\T_0^1}.$$ 

Let $\Phi_{\T_0}$ denote the equivalence $\bigoplus_{\tau\in \T_0}\Hom_{\Cat_\Ring}(P^{\T_0}_{\Ring,\tau},\bullet):
\Cat_{\Ring,\T_0}\xrightarrow{\sim} A_{\Ring,\T_0}\operatorname{-mod}$. We have a functor isomorphism 
$\varpi_{\T_0^1,\T_0^2}^*\circ \Phi_{\T_0^1}\xrightarrow{\sim}\Phi_{\T_0^2}\circ \iota_{\T_0^1}$, on an object $M$
it sends an element of $\bigoplus \Hom(P^{\T_0^1}_{\Ring,\tau},M)$ to 
its composition with $\eta_{\T_0^2,\T_0^1}$. 

Consider the following category $\Cat'_\Ring$. Its objects are modules over the associative non-unital algebra  $\varprojlim A_{\Ring,\T_0}$ 
that factor through a locally unital and locally finite (conditions (i') and (ii') in Section \ref{SSS_loc_unit}) Noetherian $A_{\Ring,\T_0'}$-module for some coideal finite poset ideal $\T_0'\subset\T$.  The morphisms 
are $\varprojlim A_{\Ring,\T_0}$-linear maps. The discussion above shows that the functors
$\Phi_{\T_0}:\Cat_{\T_0,\Ring}\xrightarrow{\sim} A_{\Ring,\T_0}\operatorname{-mod}$ extend
to an equivalence $\Phi:\Cat_\Ring\xrightarrow{\sim} \Cat'_\Ring$ of $\Ring$-linear abelian categories.

\begin{Rem}\label{Rem:interval_finite_base_change}
This construction, in particular, allows to define the base change $\Cat_{\bar{\Ring}}$ for a Noetherian $\Ring$-algebra 
$\bar{\Ring}$: we take the analog of the category $\Cat'_\Ring$ for   $\varprojlim (\bar{\Ring}\otimes_{\Ring}A_{\Ring,\T_0})$. 
The category  $\Cat_{\bar{\Ring}}$ is highest weight with standard objects $\bar{\Ring}\otimes_{\Ring}\Delta_{\Ring}(\tau), \tau\in \T$. 
To check the axioms is left as an exercise. 
\end{Rem}

\subsubsection{Example}\label{SSS:ex_hw_interval}
We will now give an example of a highest weight category with interval finite poset. This example is taken from 
\cite[Section 6]{ModHCO}. Take an adjoint semisimple algebraic $G$ and an algebraically closed field $\F$ of characteristic $p$
bigger than the maximum of the Coxeter numbers of the simple summands of
$\g$. Then we can consider the classical category $\OCat^{cl}$ for $\g_\F$: by definition, it consists of 
strongly $B_\F$-equivariant finitely generated $\g_\F$-modules, where $B_\F\subset G_\F$ is a Borel subgroup. An example of an object in $\OCat^{cl}$ is as follows: to a weight $\lambda$ in the character lattice of $B_\F$ we assign the Verma module $\Delta^{cl}_\F(\lambda)$ with highest weight $\lambda$. 
It is a standard observation that $\OCat^{cl}$ is a highest weight category with poset
from Example \ref{Ex:poset_highest_weights}. The standard objects are $\Delta^{cl}_\F(\lambda)$. Compare to Section \ref{SSS_quantum_hw_structure} that treats the quantum case. 

Recall that $\Lambda_0$ denotes the root lattice of $G$. 
Consider the action of the extended affine Weyl group $W^a:=W\ltimes \Lambda_0$ on $\Lambda_0$ given as follows.
As usual, we write $\rho$ for half the sum of positive roots. For $\mu\in \Lambda$, let $t_\mu$ denote 
the corresponding element of $W^a$. Then we have the unique action of $W^a$ on $\Lambda$ given by 
$w\cdot \lambda:=w(\lambda+\rho)-\rho, t_\mu\cdot \lambda=\lambda+p\mu$. The Serre span of the standard
objects $\Delta^{cl}_\F(x^{-1}\cdot (-2\rho))$ for $x\in W^a$ will be called the {\it principal block} of 
$\OCat^{cl}$  and is denoted by
$\OCat^{[0]}$. It is direct summand in $\OCat^{cl}$ and hence also a highest weight category with an interval finite poset.  
 
\section{Equivalences of highest weight categories}\label{S_RS}
We use the notation and conventions of Section \ref{SS_hw_finite}.
\subsection{Rouquier-Soergel functors}
\subsubsection{Definition}\label{SSS_RS_defn}
Let $\Ring$ be a Noetherian ring, and $\Cat_\Ring$ be a highest weight category over $\Ring$.
Let $A_\Ring$ be an a finite projective associative $\Ring$-algebra such that $\Cat_\Ring\cong A_\Ring\operatorname{-mod}$.
Let $\underline{\Cat}_\Ring$ be another $\Ring$-linear abelian category equivalent to the category
of modules over a finite $\Ring$-algebra $\underline{A}_\Ring$. Let $\pi_\Ring:\Cat_\Ring\rightarrow \underline{\Cat}_\Ring$
be a right exact $\Ring$-linear functor. It is given by $B_\Ring\otimes_{A_\Ring}\bullet$
for a uniquely determined $\underline{A}_\Ring$-$A_\Ring$-bimodule $B_\Ring$
such that the left and the right $\Ring$-actions on $B_\Ring$ coincide.
In particular, for any Noetherian $\Ring$-algebra $\Ring'$ we
have the induced functor $\pi_{\Ring'}:\Cat_{\Ring}\rightarrow \Cat_{\Ring'}$,
it is given by $(\Ring'\otimes_\Ring B_\Ring)\otimes_{A_{\Ring'}}\bullet$.

The following definition is inspired by \cite[Section 4.2.2]{rouqqsch}, compare also
to \cite[Section 3.1]{VV_proof}.

\begin{defi}\label{defi:faithful}
We say that the functor $\pi_\Ring$ is
\begin{itemize}
\item {\it $(-1)$-faithful} if it is faithful on $\Cat_\Ring^\Delta$, the category of standardly
filtered objects in $\Cat_\Ring$;
\item {\it 0-faithful} if it is fully faithful on $\Cat_\Ring^\Delta$;
\item {\it 1-faithful} if it is exact, 0-faithful, and, moreover,
$$\Ext^1_{\Cat_\Ring}(M_\Ring, N_\Ring)\xrightarrow{\sim}
\Ext^1_{\underline{\Cat}_\Ring}(\pi_\Ring(M_\Ring),\pi_\Ring(N_\Ring))$$
for any standardly filtered objects $M_\Ring, N_\Ring\in \Cat_\Ring$.
\end{itemize}
\end{defi}

\begin{defi}\label{defi:RS_functor} Suppose that $\Ring$ is a normal Noetherian domain. Let $\F$ stand for its field of fractions.
We say that $\pi_\Ring$ is a {\it Rouquier-Soergel} (shortly, {\it RS}) functor if the following conditions hold:
\begin{itemize}
\item[(RS1)] $\pi_\Ring$ is exact on $\Cat_\Ring^\Delta$.
\item[(RS2)] $\pi_\Ring(\Delta_\Ring(\tau))$ is projective over $\Ring$ for all $\tau\in \T$.
\item[(RS3)] The categories $\Cat_{\F},\underline{\Cat}_\F$ are split semisimple and
$\pi_\F$ is their equivalence.
\item[(RS4)] For each prime ideal $\p\subset \Ring$ of height not exceeding $2$, the functor
$\pi_{\kf(\p)}:\Cat_{\kf(p)}\rightarrow \underline{\Cat}_{\kf(\p)}$ is $(-1)$-faithful.
\end{itemize}
\end{defi}

\begin{Rem}\label{Rem:RS2_equiv}
Note that (RS1) implies that the higher derived functors $L_i\pi_\Ring$ vanish on standardly filtered objects thanks to condition (v) in the definition of a highest weight category.
\end{Rem}

\begin{Rem}\label{Rem:RS4}
Suppose $\Ring$ is, in addition, local, with residue field $\kf$. We claim that if $\pi_{\kf}$ is $(-1)$-faithful, then (RS4) holds.
Take a quotient $\Ring'$ of $\Ring$. 
We have the following implications (a)$\Rightarrow$(b)$\Rightarrow$(c)$\Rightarrow$(d): 
\begin{itemize}
\item[(a)] $\pi_{\kf}$ is $(-1)$-faithful,
\item[(b)] for any finite length quotient $\Ring_0$ of $\Ring$, the functor $\pi_{\Ring_0}$
is $(-1)$-faithful,
\item[(c)] $\pi_{\hat{\Ring}'}$ is $(-1)$-faithful, where $\hat{\Ring}'$ is the completion of $\Ring'$
at the maximal ideal; this is because of the inclusion 
$$\Hom_{\Cat_{\hat{\Ring}'}}(\hat{\Ring}'\otimes_\Ring M_\Ring, \hat{\Ring}'\otimes_\Ring N_\Ring)
\hookrightarrow \varprojlim \Hom_{\Cat_{\Ring_0}}(\Ring_0\otimes_\Ring M_\Ring, \Ring_0\otimes_\Ring N_\Ring),$$
for all $\Ring$-flat objects $M_\Ring, N_\Ring\in \Cat_\Ring$, where the inverse limit is taken over finite length quotients
$\Ring_0$ of $\Ring'$, and the similar inclusion for the category $\underline{\Cat}_{\hat{\Ring}'}$.
\item[(d)] $\pi_{\Ring'}$ is $(-1)$-faithful. 
\end{itemize}
In particular, we can take $\Ring'$ with full fraction ring $\kf(\p)$, and then (d) implies (RS4) for the given $\p$.
\end{Rem}

\subsubsection{Examples}
One of the main results of this paper is the construction of RS functors from two categories
(the deformed quantum categories $\mathcal{O}$ and the heart of the new t-structure on the affine Hecke category) to a certain category of bimodules (this doesn't quite fit in the above setting because these categories are not finite). Here we will sketch two families of examples already in the
literature. We are not going to use these examples in what follows.

The first example is essentially from \cite{Soergel}. Let $\g$ be a semisimple Lie algebra over $\C$ with Cartan subalgebra $\h$. Choosing simple roots in $\h^*$, one can talk about the BGG category $\mathcal{O}$
and its principal block $\mathcal{O}^{pr}$. This is a highest weight category over $\C$. Its standard objects are labelled by elements of the Weyl group $W$: to $w\in W$ we assign the Verma module $\Delta(w)$
with highest weight $w\cdot (-2\rho)$, the order on $W$ is the Bruhat order. Let $\Ring$ denote the completion of $\C[\h^*]$ at $0$. Then we can consider a deformation  $\mathcal{O}^{pr}_\Ring$, a highest weight category over $\Ring$, compare to \cite[Section 3.1]{Soergel} and 
Section \ref{SS_O_quant}. Consider the ``anti-dominant'' projective object $P_\Ring(w_0)$. Its endomorphisms are the algebra $\Ring\otimes_{\Ring^W}\Ring$, compare to \cite[Theorem 9]{Soergel}. Then we consider the functor
$\pi_\Ring:=\operatorname{Hom}_{\mathcal{O}^{pr}_\Ring}(P_\Ring(w_0),\bullet): \mathcal{O}^{pr}_\Ring\rightarrow \Ring\otimes_{\Ring^W}\Ring\operatorname{-mod}$.
It is an RS functor. The axioms (RS1) and (RS2) are manifest. (RS3) is easy to check. It is enough to check (RS4) only for the specialization to the closed point, see Remark \ref{Rem:RS4}. There it follows because the socle of every Verma module is the anti-dominant Verma that occurs with multiplicity $1$.

Some other families of examples can be found in \cite[Sections 5,6]{rouqqsch} they have to do
with categories $\mathcal{O}$ over rational Cherednik algebras (with KZ functors $\pi_\Ring$)
and higher level Schur algebras (with Schur functors $\pi_\Ring$). We note that we impose much weaker
conditions on our functors compared to \cite[Section 4.2]{rouqqsch}.

\subsubsection{Basic properties}
In what follows to simplify the notation, for $M\in \Cat_\Ring$, we write $\underline{M}$ for
$\pi_\Ring(M)$.

\begin{Lem}\label{Lem:RS_0_faith}
Every RS functor $\pi_\Ring: \Cat_\Ring\rightarrow \underline{\Cat}_\Ring$ is 0-faithful.
\end{Lem}
\begin{proof}
Let $A_\Ring$ be a finite projective $\Ring$-algebra. Let $M_\Ring,N_\Ring$ be two objects
in $A_\Ring\operatorname{-mod}$ that are projective over $\Ring$. It is easy to see that
$\Hom_{A_\Ring}(M_\Ring,N_\Ring)$ is a reflexive $\Ring$-module. A homomorphism
$\varphi_\Ring:K_\Ring\rightarrow \underline{K}_\Ring$ between two finitely generated reflexive
$\Ring$-modules is an isomorphism if and only if its localization at every height $1$ prime ideal is so.

Apply this to $K_\Ring:=\Hom_{\Cat_\Ring}(M_\Ring,N_\Ring), \underline{K}_\Ring:=
\Hom_{\underline{\Cat}_\Ring}(\underline{M}_\Ring,\underline{N}_\Ring)$ for standardly filtered objects $M_\Ring,N_\Ring$ and the homomorphism $\varphi_\Ring$ induced by $\pi_\Ring$. Note that $M_\Ring,N_\Ring$ are projective over
$\Ring$ because of the definition of a highest weight category (see Section \ref{SSS_hw_finite_def}),
while $\underline{M}_\Ring,\underline{N}_\Ring$ are projective over $\Ring$ thanks to (RS1),(RS2). So $K_\Ring,\underline{K}_\Ring$
are indeed finitely generated and reflexive.

By replacing $\Ring$ with $\Ring_\p$ for height $1$ prime ideals, we can reduce the proof to the case when $\Ring$ is a local Dedekind domain with residue field $\kf$. This case can be handled as in the proofs
of \cite[Proposition 3.1]{VV_proof} or \cite[Proposition 4.42]{rouqqsch}. But we provide the details
as the settings are a bit different. Let $t$ denote a uniformizer in $\Ring$.
The sequence $$0\rightarrow \underline{M}_\Ring\xrightarrow{t}\underline{M}_\Ring\rightarrow
\underline{M}_\kf\rightarrow 0,$$
where we write $M_\kf$ for $\kf\otimes_\Ring M_\Ring$ and $\underline{M}_\kf:=\pi_\Ring(M_\kf)$,
is exact thanks to (RS1) and (RS2), and similarly for $N_\Ring$. It follows that
$$K_\Ring/t K_\Ring\hookrightarrow \Hom_{\Cat_{\kf}}(M_\kf,N_\kf),
\underline{K}_\Ring/t \underline{K}_\Ring\hookrightarrow \Hom_{\underline{\Cat}_\kf}(\underline{M}_\kf,\underline{N}_\kf).$$
Thanks to (RS4), we have $ \Hom_{\Cat_\kf}(M_\kf,N_\kf)\hookrightarrow \Hom_{\kf}(
\underline{M}_\kf,\underline{N}_\kf)$. Hence the homomorphism
$\varphi_\kf:K_\Ring/ t K_\Ring\rightarrow \underline{K}_\Ring/ t \underline{K}_\Ring$
induced by $\varphi_\Ring$ is injective. By (RS3), $\varphi_{\F}$ is an isomorphism.
Combining these two observations, we see that $\varphi_\Ring$ is an isomorphism.
This finishes the proof.
%In this case the argument
%of the proof of \cite[Proposition 3.1]{VV_proof} (for $y$ there we take the uniformizer)
%finishes the proof. Compare also to the proof of \cite[Proposition 4.42]{rouqqsch}.
\end{proof}

%\begin{Rem}\label{Rem:RS_0_faith}
%Similarly, one can prove the following claim. Let $\Ring$ be  a normal Noetherian domain. Suppose $\Cat_\Ring,\underline{\Cat}_\Ring$ be %$\Ring$-linear abelian categories where Hom's are finitely generated $\Ring$-modules.  
%Let $\pi_\Ring$ be a right exact functor. 
%Let $M_\Ring,N_\Ring\in \Cat_\Ring$ be such that 
%\begin{itemize}
%\item $M_\Ring,N_\Ring$ are flat over $\Ring$ and the functor $\pi_\Ring$ is acyclic on $M_\Ring,N_\Ring$,
%\item $\pi_\Ring(M_\Ring), \pi_\Ring(N_\Ring)$ are flat over $\Ring$,
%\item $\F\otimes_{\Cat_\Ring}\Hom_\Ring(M_\Ring,N_\Ring)\xrightarrow{\sim} \F\otimes_\Ring\Hom_\Ring(M_\Ring,N_\Ring)$
%\end{itemize}   
%\end{Rem}

\subsection{Equivalence theorem}
Let $\Ring$ be a regular local ring. In this section we are going to prove an equivalence
between two highest weight categories over $\Ring$ with RS functors to the same category
that are subject to some compatibility condition.

\subsubsection{Main result}
Let $\Cat^i_\Ring$ be two highest weight categories over $\Ring$ with (finite) posets
$\T^i,i=1,2$. Let $\underline{\Cat}_\Ring$
have the same meaning as in Section \ref{SSS_RS_defn}. Let $\pi^i_{\Ring}:\Cat^i_\Ring\rightarrow
\underline{\Cat}_\Ring$ be RS functors.

We have the following result.

\begin{Thm}\label{Thm:hw_equiv}
Suppose that
\begin{itemize}
\item[(a)] there is a bijection $\iota:\T^1\xrightarrow{\sim}\T^2$ such that
$\pi^1_{\Ring}(\Delta^1_\Ring(\tau))\cong \pi^2_{\Ring}(\Delta^2_\Ring(\iota(\tau)))$
for all $\tau\in \T^1$,
\item[(b)] and for every height $1$ prime ideal $\p\subset \Ring$, the essential
images of $\Cat^{1,\Delta}_{\Ring_\p}$ and $\Cat^{2,\Delta}_{\Ring_\p}$
in $\underline{\Cat}_{\Ring_\p}$ coincide.
\end{itemize}
Then there is a unique equivalence of highest weight categories $\Cat^1_\Ring\xrightarrow{\sim}\Cat^2_\Ring$
intertwining $\pi^1_\Ring$ and $\pi^2_\Ring$. It automatically sends $\Delta^1_\Ring(\tau)$ to
$\Delta^2_\Ring(\iota(\tau))$ for all $\tau\in \T^1$.
\end{Thm}

This theorem will be proved below in this section. Here is how the proof works:
\begin{itemize}
\item By Lemma \ref{Lem:RS_0_faith}, the functor $\pi^i_\Ring$ is 0-faithful.
By Remark \ref{Rem:stand_filt_recover}, we can recover $\Cat^i_\Ring$ from its
full subcategory $\Cat^{i,\Delta}_\Ring$. So Theorem boils down to proving that
\begin{equation}\label{eq:image_stand_coincide}
\pi^1_\Ring(\Cat^{1,\Delta}_\Ring)=\pi^2_\Ring(\Cat^{2,\Delta}_\Ring).
\end{equation}
\item Since $\pi^i_\Ring$ is 0-faithful and (Remark \ref{Rem:RS2_equiv}) acyclic on the standard objects, it induces embeddings
$\Ext^1_{\Cat^i_\Ring}(M^i_\Ring, N^i_\Ring)\hookrightarrow
\Ext^1_{\underline{\Cat}_\Ring}(\pi^i_\Ring(M^i_\Ring),\pi^i_\Ring(N^i_\Ring))$ for $i=1,2$.
If we know that for all $M^i_\Ring, N^i_\Ring$ such that $\pi^1_\Ring(M^1_\Ring)=\pi^2_\Ring(M^2_\Ring),\pi^1_\Ring(N^1_\Ring)=\pi^2_\Ring(N^2_\Ring)$,
the images of $\Ext^1$'s for $i=1,2$, coincide, we can argue inductively on the number of standards in a filtration, to deduce (\ref{eq:image_stand_coincide}).
\item We will show that the images of $\Ext^1$'s indeed coincide.
\end{itemize}

\subsubsection{Description of $\Ext^1$}
Let $\Cat_\Ring,\underline{\Cat}_\Ring$ be as in Section \ref{SSS_RS_defn}, and $\pi_\Ring:\Cat_\Ring
\rightarrow \underline{\Cat}_\Ring$ be an RS functor. For $M\in \Cat_\Ring$ we will write
$\underline{M}$ for $\pi_\Ring(M)$.

Our goal in this section is, for $M_\Ring,N_\Ring\in \Cat_\Ring^\Delta$, to describe the image of $\Ext^1_{\Cat_\Ring}(M_\Ring,N_\Ring)$ in $\Ext^1_{\underline{\Cat}_\Ring}(\underline{M}_\Ring,\underline{N}_\Ring)$.

In order to state the description we need some preparation. The $\Ring$-module $\Ext^1_{\underline{\Cat}_\Ring}(\underline{M}_\Ring,\underline{N}_\Ring)$ is finitely generated.
By (RS3) it is torsion. Fix an element $a\in \Ring$ annihilating this module. Then $a$
also annihilates the submodule $\Ext^1_{\Cat_\Ring}(M_\Ring,N_\Ring)$. We write $\Ring^1$ for $\Ring/(a)$.
Let $\p_1,\ldots,\p_k\subset \Ring$ be the minimal prime ideals of $a$.
Set $\Loc(\Ring):=\bigoplus_{i=1}^k \Ring_{\p_i}$. For an $\Ring$-module $Q_\Ring$, we write
$Q_{\Ring^1}:=Q_\Ring/a Q_\Ring$ and $Q_{\Loc(\Ring)}$ for $\Loc(\Ring)\otimes_\Ring Q_\Ring$.
The same notation is used for objects in
$\Ring$-linear abelian categories.

Note that $\Loc(\Ring^1)$ is the full
fraction ring of $\Ring^1$.

Let $\Ring'$ be a Noetherian ring.
We can talk about $\Ring'$-modules of positive depth -- these are finitely generated
$\Ring'$-modules $Q_{\Ring'}$ such that every non-zero divisor $b\in \Ring'$ is not a zero divisor
in $Q_{\Ring'}$. For example, the Hom module to every positive depth (e.g. free) $\Ring'$-module
is also positive depth. Equivalently, $Q_{\Ring'}$ is of positive depth if the natural homomorphism
$Q_{\Ring'}\rightarrow Q_{\operatorname{Frac}(\Ring')}$ is injective, where $\operatorname{Frac}(\Ring')$ stands for the full quotient ring of $\Ring'$.

The following proposition is a key result of this section that will yield the description we need.

\begin{Prop}\label{Prop:injectivity}
The homomorphism
\begin{equation}\label{eq:Hom_homomorphism}
\Hom_{\Cat_{\Ring^1}}(M_{\Ring^1},N_{\Ring^1})\rightarrow
\Hom_{\underline{\Cat}_{\Ring^1}}(\underline{M}_{\Ring^1},\underline{N}_{\Ring^1})
\end{equation}
is injective. Moreover, its cokernel is of positive depth.
\end{Prop}
\begin{proof}
{\it Step 1}. Let $\Ring'$ be a quotient of $\Ring$ that is a complete intersection of codimension $\leqslant 2$
(for example, we can take $\Ring'=\Ring^1$) and $\F'$ be its full fraction ring.
We claim that
\begin{equation}\label{eq:Hom_homomorphism1}
\Hom_{\Cat_{\Ring'}}(M_{\Ring'},N_{\Ring'})\rightarrow
\Hom_{\underline{\Cat}_{\Ring'}}(\underline{M}_{\Ring'},\underline{N}_{\Ring'})
\end{equation}
is injective.

Since $M_{\Ring'},N_{\Ring'},\underline{M}_{\Ring'},\underline{N}_{\Ring'}$
are projective over $\Ring$, hence of positive depth,  the $\Ring'$-modules $\Hom_{\Cat_{\Ring'}}(M_{\Ring'},N_{\Ring'}),
\Hom_{\underline{\Cat}_{\Ring'}}(\underline{M}_{\Ring'},\underline{N}_{\Ring'})$
are of positive depth. So their natural maps to $\Hom_{\Cat_{\F'}}(M_{\F'},N_{\F'}),
\Hom_{\underline{\Cat}_{\F'}}(\underline{M}_{\F'},\underline{N}_{\F'})$
are embeddings. To prove that (\ref{eq:Hom_homomorphism1}) is injective it is enough to show that
\begin{equation}\label{eq:localization_embedding}
\Hom_{\Cat_{\F'}}(M_{\F'},N_{\F'})\hookrightarrow
\Hom_{\underline{\Cat}_{\F'}}(\underline{M}_{\F'},\underline{N}_{\F'}),
\end{equation}
here $\Ring'=\Ring^1$ and so $\F'=\Loc(\Ring^1)$.

Let $\q_1,\ldots,\q_\ell$ be the prime ideals in $\Ring$ containing $\ker(\Ring\twoheadrightarrow \Ring')$ whose images in $\Ring'$ are the minimal prime ideals. The heights of the ideals $\q_i$ are all $\leqslant 2$. Note that $\F'$ has finite length and it is filtered by $\kf(\q_i)$'s. Thanks to (RS4), we know that
$$\Hom_{\Cat_{\kf(\q_i)}}(M_{\kf(\q_i)},N_{\kf(\q_i)})\hookrightarrow
\Hom_{\underline{\Cat}_{\kf(\q_i)}}(\underline{M}_{\kf(\q_i)},\underline{N}_{\kf(\q_i)}).$$
(\ref{eq:localization_embedding}) follows. This shows that  (\ref{eq:Hom_homomorphism1})
(and its special case (\ref{eq:Hom_homomorphism})) are injective.

{\it Step 2}.
Now we show that the cokernel of (\ref{eq:Hom_homomorphism}) is of positive depth. Equivalently,
we need to show that for every nonzero divisor $b\in \Ring^1$, the map
\begin{equation}\label{eq:Hom_homomorphism2}
\Hom_{\Cat_{\Ring^1}}(M_{\Ring^1},N_{\Ring^1})/b \Hom_{\Cat_{\Ring^1}}(M_{\Ring^1},N_{\Ring^1})\rightarrow
\Hom_{\underline{\Cat}_{\Ring^1}}(\underline{M}_{\Ring^1},\underline{N}_{\Ring^1})/
b\Hom_{\underline{\Cat}_{\Ring^1}}(\underline{M}_{\Ring^1},\underline{N}_{\Ring^1}).
\end{equation}
is injective.

Set $\Ring':=\Ring^1/(b)$. The source of (\ref{eq:Hom_homomorphism2}) embeds into
$\Hom_{\Cat_{\Ring'}}(M_{\Ring'},N_{\Ring'})$ and the target embeds into
$\Hom_{\underline{\Cat}_{\Ring'}}(\underline{M}_{\Ring'},\underline{N}_{\Ring'})$ in such a way that
(\ref{eq:Hom_homomorphism2}) gets intertwined with (\ref{eq:Hom_homomorphism1}).
So  (\ref{eq:Hom_homomorphism2}) is injective by Step 1.
\end{proof}

Consider the natural homomorphisms
\begin{equation}\label{eq:hom_Ext_comput}
\begin{split}
&\kappa:\Hom_{\Cat_\Ring}(M_\Ring,N_\Ring)\rightarrow
\Hom_{\Cat_{\Ring^1}}(M_{\Ring^1},N_{\Ring^1}),\\
&\underline{\kappa}:\Hom_{\underline{\Cat}_\Ring}(\underline{M}_\Ring,\underline{N}_\Ring)\rightarrow
\Hom_{\underline{\Cat}_{\Ring^1}}(\underline{M}_{\Ring^1},\underline{N}_{\Ring^1}).
\end{split}
\end{equation}
Recall, Lemma \ref{Lem:RS_0_faith}, that $\pi_\Ring$ is $0$-faithful, so $\Hom_{\Cat_\Ring}(M_\Ring,N_\Ring)\xrightarrow{\sim}
\Hom_{\underline{\Cat}_\Ring}(\underline{M}_\Ring,\underline{N}_\Ring)$. Let $\mathcal{H}(M_\Ring, N_\Ring):=\Hom_{\Cat_\Ring}(M_\Ring,N_\Ring)/a\Hom_{\Cat_\Ring}(M_\Ring,N_\Ring)$.

By the long exact sequence for $R\Hom_{\Cat_\Ring}(M_\Ring, \bullet)$ applied to
$$0\rightarrow N_\Ring\xrightarrow{a\cdot}N_\Ring\rightarrow N_{\Ring^1}\rightarrow 0$$
we get that $\kappa$ factors through the inclusion
$$\mathcal{H}(M_\Ring, N_\Ring)\hookrightarrow
\Hom_{\Cat_{\Ring^1}}(M_{\Ring^1},N_{\Ring^1})$$
while $\Ext^1_{\Cat_\Ring}(M_\Ring,N_\Ring)\cong \operatorname{coker}\kappa$. Similarly,
$\underline{\kappa}$ factors through the inclusion
$$\mathcal{H}(M_\Ring, N_\Ring)\hookrightarrow
\Hom_{\underline{\Cat}_{\Ring^1}}(\underline{M}_{\Ring^1},\underline{N}_{\Ring^1})$$
while $\Ext^1_{\underline{\Cat}_\Ring}(\underline{M}_\Ring,\underline{N}_\Ring)\cong \operatorname{coker}\underline{\kappa}$. Note that the notation $\underline{M}_{\Ring^1}$ is unambiguous -- this module is identified with
$\Ring^1\otimes_\Ring \underline{M}_\Ring$ and with $\pi_{\Ring}(M_{\Ring^1})$.

Now we are going to define an $\Ring^1$-submodule $\widetilde{\mathcal{H}}(M_\Ring, N_\Ring)
\subset \Hom_{\underline{\Cat}_{\Ring^1}}(\underline{M}_{\Ring^1},\underline{N}_{\Ring^1})$.

Since the $\Ring^1$-module $\Hom_{\underline{\Cat}_{\Ring^1}}(\underline{M}_{\Ring^1},\underline{N}_{\Ring^1})$
is of positive depth, we get
\begin{equation}\label{eq:Ext1_embedding1}
\Hom_{\underline{\Cat}_{\Ring^1}}(\underline{M}_{\Ring^1},\underline{N}_{\Ring^1})\hookrightarrow
\Hom_{\underline{\Cat}_{\Loc(\Ring^1)}}(\underline{M}_{\Loc(\Ring^1)},\underline{N}_{\Loc(\Ring^1)}).
\end{equation}
By Proposition \ref{Prop:injectivity}, we also have
\begin{equation}\label{eq:Ext1_embedding2}
\Hom_{\Cat_{\Loc(\Ring^1)}}(M_{\Loc(\Ring^1)},N_{\Loc(\Ring^1)})
\hookrightarrow \Hom_{\underline{\Cat}_{\Loc(\Ring^1)}}(\underline{M}_{\Loc(\Ring^1)},\underline{N}_{\Loc(\Ring^1)})
\end{equation}
For $\widetilde{\mathcal{H}}(M_\Ring, N_\Ring)$ we take the intersection of the images of embeddings
(\ref{eq:Ext1_embedding1}) and (\ref{eq:Ext1_embedding2}).

\begin{Cor}\label{Cor:Ext_description}
The $\Ring^1$-submodule $\mathcal{H}(M_\Ring,N_\Ring)\subset \Hom_{\underline{\Cat}_{\Ring^1}}(\underline{M}_{\Ring^1},\underline{N}_{\Ring^1})$ is contained
in $\widetilde{\mathcal{H}}(M_\Ring, N_\Ring)$. Moreover,
$$\widetilde{\mathcal{H}}(M_\Ring, N_\Ring)/\mathcal{H}(M_\Ring, N_\Ring)
\subset \Hom_{\underline{\Cat}_{\Ring^1}}(\underline{M}_{\Ring^1},\underline{N}_{\Ring^1})/\mathcal{H}(M_\Ring, N_\Ring)\cong \Ext^1_{\underline{\Cat}_\Ring}(\underline{M}_\Ring,\underline{N}_\Ring)$$
coincides with the image of the embedding $\Ext^1_{\Cat_\Ring}(M_\Ring,N_\Ring)\hookrightarrow \Ext^1_{\underline{\Cat}_\Ring}(\underline{M}_\Ring,\underline{N}_\Ring)$.
\end{Cor}
\begin{proof}
The containment $\mathcal{H}(M_\Ring,N_\Ring)\subset \widetilde{\mathcal{H}}(M_\Ring, N_\Ring)$
follows from $\mathcal{H}(M_\Ring,N_\Ring)\subset \Hom_{\Cat_{\Ring^1}}(M_{\Ring^1},N_{\Ring^1})$.
Proving the description of $\Ext^1_{\Cat_\Ring}(M_\Ring,N_\Ring)$ amounts to checking that $\Hom_{\Cat_{\Ring^1}}(M_{\Ring^1},N_{\Ring^1})
\subset \Hom_{\underline{\Cat}_{\Ring^1}}(\underline{M}_{\Ring^1},\underline{N}_{\Ring^1})$
coincides with $\widetilde{\mathcal{H}}(M_\Ring, N_\Ring)$. This is a straightforward consequence
of the claim that $\Hom_{\underline{\Cat}_{\Ring^1}}(\underline{M}_{\Ring^1},\underline{N}_{\Ring^1})/\Hom_{\Cat_{\Ring^1}}(M_{\Ring^1},N_{\Ring^1})$
is of positive depth (Proposition \ref{Prop:injectivity}), equivalently, embeds into its base change to
$\Loc(\Ring^1)$.
\end{proof}

\subsubsection{Completion of proof}
\begin{proof}[Proof of Theorem \ref{Thm:hw_equiv}]
For $\ell>0$, let $\Cat_\Ring^{i,\Delta}(\ell)$ denote the full subcategory of standardly filtered objects, where the length of the filtration is at most $\ell$. It is enough to prove
that
\begin{itemize}
\item[($A_\ell$)]
$\pi^1_\Ring(\Cat_\Ring^{1,\Delta}(\ell))=\pi^2_\Ring(\Cat_\Ring^{2,\Delta}(\ell))$.
\end{itemize}
We prove ($A_\ell$) by induction on $\ell$. The base, $\ell=1$, follows from condition (a).
To prove the induction step, notice that every $L^i_\Ring\in \Cat_\Ring^{i,\Delta}(\ell+1)$
fits into a short exact sequence
\begin{equation}\label{eq:SES_stand}
0\rightarrow N_\Ring^i\rightarrow L^i_\Ring\rightarrow M_\Ring^i\rightarrow 0,\quad M_\Ring^i\in \Cat_\Ring^{i,\Delta}(\ell),\,N_\Ring^i\in \Cat_\Ring^{i,\Delta}(1).
\end{equation}

Pick $L^1_\Ring\in \Cat_\Ring^{2,\Delta}(\ell+1)$ and let $\underline{L}_\Ring$ be its image in
$\underline{\Cat}_\Ring$. Take $N^1_\Ring$ and $M^1_\Ring$ as in (\ref{eq:SES_stand}).
By the induction assumption, there are $N^2_\Ring\in \Cat_\Ring^{2,\Delta}(\ell),
M^2_\Ring\in \Cat_\Ring^{2,\Delta}(1)$ with $\underline{N}^2_\Ring\cong \underline{N}^1_\Ring$
and $\underline{M}^2_\Ring\cong \underline{M}^1_\Ring$. We will write
$\underline{N}_\Ring$ and $\underline{M}_\Ring$ for these objects. Our goal is to construct
an extension $L^2_\Ring$ of $M^2_\Ring$ by $N^2_\Ring$ with $\underline{L}^2_\Ring\cong \underline{L}^1_\Ring$, this will finish the induction step.
This will follow if we show that the inclusions $\Ext^1_{\Cat_\Ring}(M_\Ring^i, N_\Ring^i)\rightarrow \Ext^1_{\underline{\Cat}_\Ring}(\underline{M}_\Ring,\underline{N}_\Ring)$
have the same image. Using Corollary \ref{Cor:Ext_description}, the  coincidence of the images will follow 
once we know that the following property holds:
\begin{itemize}
\item[(*)]
the images of $$\Hom_{\Cat^i_{\Loc(\Ring^1)}}(M^i_{\Loc(\Ring^1)},N^i_{\Loc(\Ring^1)})\hookrightarrow
\Hom_{\underline{\Cat}_{\Loc(\Ring^1)}}(\underline{M}_{\Loc(\Ring^1)},\underline{N}_{\Loc(\Ring^1)})$$
coincide for $i=1,2$.\end{itemize}
 
Note that $\Loc(\Ring^1)=\bigoplus_{j=1}^k \Loc_j(\Ring^1)$, where
$\Loc_j(\Ring^1):=\Ring_{\p_j}/(a)$. 

Thanks to (b) and Remark \ref{Rem:stand_filt_recover}, for all $j$ we get an equivalence $\Cat^1_{\Ring_{\p_j}}\xrightarrow{\sim}\Cat^2_{\Ring_{\p_j}}$ intertwining the functors
$\pi^i_{\Ring_{\p_j}}$. Note that 
$$\pi^i_{\Ring_{\p_j}}(M_{\Loc_j(\Ring^1)})\cong \underline{M}_{\Loc_j(\Ring^1)}, 
\pi^i_{\Ring_{\p_j}}(N_{\Loc_j(\Ring^1)})\cong \underline{N}_{\Loc_j(\Ring^1)}.$$
(*) holds thanks to the equivalence  $\Cat^1_{\Ring_{\p_j}}\xrightarrow{\sim}\Cat^2_{\Ring_{\p_j}}$. This finishes the proof.
%Thanks to Corollary \ref{Cor:Ext_description}, the . For $L^2_\Ring$ we take the extension of $M^2_\Ring$ by $N^2_\Ring$ whose class in
%$\Ext^1_\Ring(M^2_\Ring,N^2_\Ring)$ maps to the image of the class of the extension
%$L^1_\Ring$. By the construction, $\pi^2_\Ring(L^2_\Ring)\cong \pi^1_\Ring(L^1_\Ring)$.
%This finishes the induction step and the proof.
\end{proof}

\section{Stabilized t-structure on the affine Hecke category}\label{S_stabilized_Hecke}
\subsection{Affine Hecke category}
In this section we recall various things related to the affine Hecke category. It essentially contains
no new results.

\subsubsection{Setting}\label{SSS_Hecke_setting}
Fix a finite type root system. Let $W,\Lambda_0,\h$ be the corresponding Weyl group, root lattice and the Cartan space
(so that $\Lambda_0\subset \h^*$).   Then we can form the affine Weyl
group $W^{a}=W\ltimes \Lambda_0$. We write $\hat{\h}$ for $\h\oplus \C \hbar$, where $\hbar$ is the notation for an
additional basis element. 

The group $W^{a}$ acts on $\hat{\h}$ as follows: $\hbar$ is invariant, $W$ acts on $\h$ by the default action,
and for $\lambda\in \Lambda_0, \xi\in \h$, we have $t_\lambda x=x+\langle\lambda,\xi\rangle\hbar$. Here and below $t_\lambda$
denotes $\lambda$ viewed as an element of $W^a$. Identify $\hat{\h}^*$ with $\h^*\oplus \C$ by
$$\langle(\mu,z), \xi\rangle=\langle\mu,\xi\rangle, \langle(\mu,z),\hbar\rangle=-z.$$
The action of $W^a$ on $\hat{\h}^*$ becomes 
$$w.(\mu,z)=(w\lambda,z), t_\lambda(\mu,z)=(\mu+z\lambda,z).$$
We note that $W^a$ is an affine reflection group: the reflections are given by the elements 
$s_\alpha t_{k\alpha}=t_{-k\alpha}s_\alpha$. This element acts on $\hat{\h}^*$ by 
$(\lambda,z)\mapsto (\lambda-(zk+\langle\lambda,\alpha^\vee\rangle)\alpha,z)$ and so the corresponding reflection 
hyperplane given by the linear function $\alpha^\vee-k\hbar:\hat{\h}^*\rightarrow \C$.   

%Here $\mu\in \h^*,z\in \C, w\in W, \lambda\in \Lambda_0$ and we write
%$t_\lambda$ for the corresponding element in $\Lambda_0\subset W^{a}$.
%Note that $W^a$ is an affine Coxeter group. 

Set $\Ring:=\C[\hat{\h}^*]^{\wedge_0}$, the completion of $\C[\hat{\h}^*]=S(\hat{\h})$ at $0$.
This algebra is acted on by $W^{a}$.
Consider the category $\Ring\operatorname{-bimod}$ of $\Ring$-bimodules. An example of
an object is the bimodule $\Ring_x$ for $x\in W^{a}$, it is a free rank one module
as a right $\Ring$-module, and the action of $\Ring$ on the left is twisted by $x^{-1}$.
In other words, this is the completed version of the functions on the graph of
$x$ under its action on $\hat{\h}^*$, i.e., the locus $\{(x\xi,\xi)| \xi\in \hat{\h}^*\}$.

\subsubsection{The category $\SBim$}
Inside $\Ring\operatorname{-bimod}$ we have the full subcategory of Soergel bimodules to be denoted
by $\SBim(W^{a})$. By definition, it is the full Karoubian monoidal
additive subcategory of $\Ring\operatorname{-bimod}$ generated by the elementary Bott-Samelson bimodules $\Ring\otimes_{\Ring^s}\Ring$, where $s$ runs over the set of simple affine reflections.

We will denote the tensor product functor by $*$. 

This is the additive version of the Hecke category for $W^{a}$. We also consider the triangulated version, $\Hecke:=K^b(\SBim)$. By the construction, $\Hecke$ is a monoidal category.

\subsubsection{Perverse t-structure}
\begin{Lem}\label{Lem:usual_t_structure}
The category $\Hecke$ comes with a unique t-structure (to be called perverse) with the following properties:
\begin{enumerate}
\item  its heart is
a highest weight category over $\Ring$, whose poset is
$W^{a}$ with respect to the Bruhat order (an ideal finite poset, see Remark \ref{Rem:ideal_finite}).
\item the indecomposable Soergel bimodule
$B_x\in \Hecke$ corresponds to the tilting object labelled by $x\in W^{a}$.
\end{enumerate}
\end{Lem}
\begin{proof}
One can argue as in \cite[Section 6]{EL} to formally
recover the highest weight category from $\SBim$, by the construction it is the heart of
a unique t-structure with the prescribed properties. Note that most of \cite[Section 6]{EL}
deals with the Ringel dual of the highest weight category we need, the category we need
is explained in \cite[Section 6.6.7]{EL} in a more general case.
\end{proof}

The heart of the t-structure will be denoted by $\OCat^-_\Ring$. 
Let $\Delta^-_\Ring(x)$ denote the standard object in $\OCat^-_\Ring$ labelled by $x\in W^{a}$.
The specialization of $\OCat^-_\Ring$ to the closed
point of $\operatorname{Spec}(\Ring)$ will be denoted by $\OCat^-$. This is a highest weight category over $\C$
with poset $W^a$ and standard objects $\Delta^-(x):=\C\otimes_\Ring \Delta_\Ring^-(x)$. Note that by condition (2) in
the lemma, we have
\begin{equation}\label{eq:derived_O_minus}
D^b(\OCat^-_\Ring)\xrightarrow{\sim}\Hecke.
\end{equation}

\begin{Rem}\label{Rem:BY}
One can also extract Lemma \ref{Lem:usual_t_structure} from \cite{BY}, which gave a geometric construction of $\OCat^-_\Ring$
and $\OCat^-$. In particular, $\OCat^-$ is nothing else but the category of Iwahori-equivariant
perverse sheaves on the (thin) affine flag variety.
\end{Rem}

We would like to give a corollary of Remark \ref{Rem:BY} that will be used later. It is proved in \cite[Lemma 2.1]{BBM},
note that while that lemma deals with the case of finite type flag varieties, the proof carries over to
the case of the affine flag variety verbatim.

\begin{Cor}\label{Cor:simple_socle}
The socle of any standard object in $\OCat^-$ is isomorphic to the simple object $\Delta^-(1)$.
\end{Cor}

\subsubsection{Exactness of the convolution}\label{SSS_conv_exact}
\begin{Cor}\label{Cor:tens_t_exact}
Tensoring with a standardly (resp., costandardly) filtered object in $\OCat^-_\Ring$ is left (resp., right)
t-exact endo-functor of $\Hecke$.
\end{Cor}
\begin{proof}
Tensoring with the generators of
$\SBim$ is t-exact by the construction in \cite{EL},
and then we use that every costandardly filtered object
$M$ admits a resolution
$$0\rightarrow T_k\rightarrow\ldots \rightarrow T_0\rightarrow M,$$
hence tensoring with $M$ is right t-exact. Dually, tensoring with a standardly filtered object is
left t-exact.
\end{proof}

We have a group homomorphism  from the  affine braid group $\Br^a$ to $\Hecke$,
\cite{R_braid}. The positive braid $\mathsf{T}_x, x\in W^{a},$ maps to $\nabla^-_\Ring(x)$,
while $\mathsf{T}_{x^{-1}}^{-1}$ maps to $\Delta^-_\Ring(x)$. Recall that we have a group embedding
$\Lambda_0\hookrightarrow \Br^a$ uniquely characterized by $\lambda\mapsto \mathsf{T}_{t_\lambda}$
if $\lambda$ is dominant. Let $J_\lambda$ denote the image of $\lambda\in \Lambda$
under this embedding. The following statement is a straightforward consequence of
the discussion above in this paragraph and Corollary \ref{Cor:tens_t_exact}.

\begin{Cor}\label{Cor:transl_exact}
For every dominant $\lambda\in \Lambda$, the functor $\bullet*J_\lambda$ is right t-exact,
while the functor $\bullet*J_{-\lambda}$ is left t-exact.
\end{Cor}

\subsubsection{Soergel functor}\label{SSS_Soergel_fun}
We consider the tautological exact functor of triangulated categories
$\Hecke\rightarrow D^b(\Ring\operatorname{-bimod})$ to be denoted by
$\Vfun$. Here are basic properties of this functor.

\begin{Lem}\label{Lem:V_O_negative_properties}
The functor $\Vfun$ has the following properties.
\begin{enumerate}
\item It is monoidal.
\item $\Vfun(\Delta^-_\Ring(x))\cong \Vfun(\nabla^-_\Ring(x))\cong \Ring_x$.
\item The functor $\Vfun$ is t-exact (w.r.t. the perverse t-structure
on $\Hecke$ and the usual t-structure on $D^b(\Ring\operatorname{-bimod})$).
\item The functor $\Vfun$ is faithful on standardly filtered objects in $\OCat^-$.
\end{enumerate}
\end{Lem}
\begin{proof}
(1) is a tautology. (2) is true for the generators (the simple affine reflections)
and then for the general $x$ we deduce the claim using the left action of $\Br^a$.  
In more detail, for $x=s$, a simple affine reflection, we have an exact
sequence $0\rightarrow \nabla^-_\Ring(s)\rightarrow T_s\rightarrow \nabla^-_\Ring(1)\rightarrow 0$, where $T_s$
stands for the tilting object labelled by $s$. Note that
$\Ring_s$ is the kernel of the natural epimorphism $\Ring\otimes_{\Ring^s}\Ring\rightarrow \Ring$.
It follows that $\Vfun(\nabla^-_\Ring(s))\cong \Ring_s$. Note that since $\mathsf{T}_x$ maps to $\nabla_\Ring^-(x)$, we 
get that $\nabla^-_\Ring(x)*\nabla^-_\Ring(s)\cong \nabla_\Ring^-(xs)$ if $\ell(xs)=\ell(x)+1$. Now we use that $\Vfun$ is a monoidal functor
to deduce that $\Vfun(\nabla^-_\Ring(x))\cong \Ring_x$ for all $x\in W^a$. The proof for standard objects is analogous.

Let us prove (3). We use that $\OCat^-_\Ring$ is a highest weight category with an ideal finite poset
(and so is the union of highest weight subcategories with finite posets).
By (2), $\Vfun(\Delta^-_\Ring(x))$ has no
higher cohomology. Every object in $\OCat^-_\Ring$ admits a resolution by standardly filtered objects,
so $\Vfun$ is right t-exact.
Also by (2), $\Vfun(\nabla^-(x))$ has no higher homology (where
$\nabla^-(x)=\C\otimes_{\Ring}\nabla^-_\Ring(x)$). It follows that $\Vfun$ is acyclic on every object in
$\OCat^-_\Ring$ annihilated by the maximal ideal  $\mathfrak{m}\subset \Ring$. We claim that this implies that $\Vfun$
is acyclic on all objects in $\OCat^-_\Ring$. Indeed, let $A_\Ring\operatorname{-mod}$ be a highest weight
subcategory of $\OCat^-_\Ring$ associated to a finite poset ideal. The restriction of $\Vfun$ to
$A_\Ring\operatorname{-mod}$ is isomorphic to a functor of the form $B_\Ring\otimes_{A_\Ring}\bullet$,
where $B_\Ring$ is an $\Ring$-$A_\Ring$-bimodule finitely generated over $\Ring\subset A_\Ring$.  Note that $B_\Ring/B_{\Ring}\mathfrak{m}$ is 
a flat $A_\Ring/A_\Ring \mathfrak{m}$-module. Also, since $B_\Ring\otimes^L_{A_\Ring}(A_\Ring/A_{\Ring}\mathfrak{m})=B_\Ring\otimes^L_{\Ring}(\Ring/\mathfrak{m})$ has no higher homology, $B_\Ring$
is flat over $\Ring$. It follows that $B_\Ring$ is flat over $A_\Ring$, which finishes the proof of (3).

Finally, we prove (4). We need to show that a nonzero homomorphism between standard objects
is sent to a nonzero homomorphism.
Part (2) implies that $\Vfun(\Delta^-(x))\neq 0$ for all $x\in W^{a}$.
Corollary \ref{Cor:simple_socle} implies that any homomorphism between standard objects is injective. 
This, together with the exactness of $\Vfun$, imply (4).
\end{proof}

\subsubsection{Singular Soergel bimodules}\label{SSS_sing_Soergel}
Let $I^a$ denote the system of affine simple roots for $W^{a}$ and let $J$ be its proper subset. 
We write $W_J\subset W^a$ for the corresponding parabolic subgroup, it is a finite Weyl group.
We write $\Ring^J$ for the ring of invariants $\Ring^{W_J}$. We consider the category $\,_J\SBim$ of singular 
Soergel $\Ring^J$-$\Ring$-bimodules, by definition, these are direct sums of direct $\Ring^J$-$\Ring$-bimodule summands of objects 
in $\SBim$. It is a right $\SBim$-module category. Note that we have $\Ring^J\otimes\Ring$-linear functors
$\pi_J:\SBim\rightarrow \,_J\SBim$ of restricting the left $\Ring$-action to $\Ring^J$, and its 
biadjoint $\pi^J:\,_J\SBim\rightarrow \SBim$ sending $B$ to $\Ring\otimes_{\Ring^J}B$.
We have the following isomorphism
\begin{equation}\label{eq:piJ_identity}
\pi_J\circ \pi^J\cong \operatorname{id}^{\oplus |W_J|}.
\end{equation}

The indecomposable objects in $\,_J\SBim$ are labelled by the left cosets for the action 
of $W_J$ on $W^{a}$,  the indecomposable object labelled by $W_Jx$ will be labelled by
$B_{W_Jx}$.

Consider the category $\,_J\Hecke:=K^b(\,_J\SBim)$. We have the following analog of Lemma 
\ref{Lem:usual_t_structure}, that also follows from the results of \cite[Section 6]{EL}. 

\begin{Lem}\label{Lem:usual_t_structure_singular}
The category $\,_J\Hecke$ comes with a unique t-structure (to be called perverse) with the following properties:
\begin{enumerate}
\item its heart is
a highest weight category over $\Ring$, whose poset is
$W_J\backslash W^{a}$ with respect to the order induced by the Bruhat order.
\item the indecomposable singular Soergel bimodules
$B_{W_Jx}\in \Hecke$ correspond to the tilting object labelled by $W_Jx$.
\item The standard object labelled by $W_Jx$ is $\pi_J(\Delta_\Ring(x))$ for all
$x\in W^{ea}$. 
\end{enumerate}
\end{Lem}

%Similar to the regular case we have the functor $\Vfun: \,_J\Hecke\rightarrow D^b(\Ring^J\operatorname{-}\Ring\operatorname{-bimod})$.

\subsubsection{Main result on the stabilized t-structure}
Now let us state our main result: the existence and properties of another t-structure on
$\Hecke$. It can be called the {\it Frenkel-Gaitsgory} t-structure -- it is closely
related to the t-structure introduced in \cite{FG}, the {\it new} t-structure, following
\cite{BLin}, where the restriction of the Frenkel-Gaitsgory t-structure to $D^b(\OCat^-)$
was studied, or the {\it stabilized} t-structure, which is the term that we are going to use.

\begin{Thm}\label{Thm:stabilized_t_structure}
The following claims are true:
\begin{enumerate}
\item
There is a unique t-structure on $\Hecke$
such that its negative part $\Hecke^{st,\leqslant 0}$ coincides with
the full subcategory of $\Hecke$ consisting of all objects $\mathcal{F}$ such that
$\mathcal{F}*J_\lambda\in \Hecke^{\leqslant 0}$ (for the perverse t-structure).
This t-structure is bounded.
\item
Moreover, the heart of the t-structure is a highest weight category with the following interval finite poset and
standard objects:
\begin{itemize}
\item The poset is $W^{a}$ with order $x\preceq y$ if $xt_{-\lambda}\leqslant yt_{-\lambda}$
for all $\lambda$ sufficiently dominant, where $\leqslant$ is the usual Bruhat order (below we will see that this gives a well-defined
order).
\item The standard object corresponding to $x\in W^{a}$ is $\Delta^-_\Ring(xt_{-\lambda})*J_{\lambda}$ for all $\lambda$
    sufficiently dominant (again, below we will see that this is well-defined).
\end{itemize}
\item Let $\OCat^{st}_\Ring$ denote the heart of the stabilized t-structure. Then $D^b(\OCat^{st}_\Ring)\xrightarrow{\sim} \Hecke$.
\end{enumerate}
\end{Thm}

In the subsequent sections we prove this theorem. And in Section \ref{SS_stab_tstr_singular}, we extend it to the singular blocks,
which is fairly standard. 

\subsection{Stabilized t-structure on $D^b(\OCat^-)$}
The goal of this section is to review the new t-structure on $\Hecke_0:=D^b(\OCat^-)$
following \cite{BLin} and to show that the heart is a highest weight category
(with poset and standard objects basically as explained in Theorem \ref{Thm:stabilized_t_structure})
based on results of \cite{ModHCO}.

\subsubsection{New t-structure following \cite{BLin}}\label{SSS_BLin}
Here we recall a few results on the new t-structure on $\Hecke_0$  following
\cite{BLin} and using constructions from \cite{B_Hecke,BM}.

Consider the ring $\Sring$ that is the localization
of $\Z$ by sufficiently many primes (it is enough to take $\Sring=\Z[\frac{1}{h!}]$, where $h$
is the maximum of the Coxeter numbers of the simple summands of $\g$,
but we will not need this). Let $G$ be the adjoint group for our fixed root system. 
The group $G$ is defined over $\mathbb{Z}$, hence over $\Sring$. In particular, we can consider the Lie algebra $\g_\Sring$, the Springer resolution $\tilde{\Nilp}_\Sring$, and the simultaneous Grothendieck resolution $\tilde{\g}_{\Sring}$. Then we can form
the Steinberg $\Sring$-scheme, $\St_\Sring:=\tilde{\g}_{\Sring}\times_{\g_\Sring^*}\tilde{\Nilp}_\Sring$,
it is a complete intersection in the regular $\Sring$-scheme $\tilde{\g}_{\Sring}\times \tilde{\Nilp}_\Sring$. The group scheme $G_\Sring$ acts on $\St_{\Sring}$ and we can consider
the equivariant derived category $D^b(\Coh^{G_\Sring}\St_\Sring)$.

Bezrukavnikov and Mirkovic in \cite{BM} have constructed a tilting generator $\Tilt_\Sring$ on
$\tilde{\g}_\Sring$. This is a vector bundle on $\tilde{\g}_\Sring$ without higher self-extensions
whose endomorphism algebra, to be denoted by $\Acal_\Sring$, has finite homological dimension.
We then have a derived equivalence
\begin{equation}\label{eq:derived_equiv_NCS_ring}
R\Gamma(\Tilt_\Sring\otimes \bullet):D^b(\Coh \tilde{\g}_{\Sring})\xrightarrow{\sim}
D^b(\Acal_\Sring\operatorname{-mod}).
\end{equation}
We can also consider the restriction of $\Tilt_{0,\Sring}$ of $\Tilt_{\Sring}$
to $\tilde{\Nilp}_\Sring\subset \tilde{\g}_\Sring$. We note that $\Acal_{\Sring}$
is naturally an algebra over $\g^*_{\Sring}\times_{\h^*_{\Sring}/W}\h^*_{\Sring}$
(the algebra of global functions on $\tilde{\g}_{\Sring}$).
The endomorphism algebra of $\Tilt_{0,\Sring}$ coincides with the specialization $\Acal_{0,\Sring}:=\Acal_{\Sring}\otimes_{\Sring[\h^*]}\Sring_0$, where $\Sring_0$
stands for the rank $1$ free $\Sring$-module, where $\h$ acts by $0$. Similarly
to (\ref{eq:derived_equiv_NCS_ring}) we have a derived equivalence
\begin{equation}\label{eq:derived_equiv_NCS_ring_specialized}
R\Gamma(\Tilt_{0,\Sring}^*\otimes \bullet):D^b(\Coh \tilde{\Nilp}_{\Sring})\xrightarrow{\sim}
D^b(\Acal_{0,\Sring}^{opp}\operatorname{-mod}).
\end{equation}

The bundle $\Tilt_{\Sring}$ is $G_\Sring$-equivariant. It follows that
(\ref{eq:derived_equiv_NCS_ring}) and  (\ref{eq:derived_equiv_NCS_ring_specialized})
lift to equivalences between the equivariant derived categories.
As argued in \cite[Section 1.5.3]{BM}, from (\ref{eq:derived_equiv_NCS_ring}) and
(\ref{eq:derived_equiv_NCS_ring_specialized}) one can deduce
that we have a derived equivalence
\begin{equation}\label{eq:derived_equiv_NCS_ring_St}
R\Gamma(\Tilt_{\Sring}\otimes\Tilt_{0,\Sring}^*\otimes \bullet):D^b(\Coh^{G_\Sring} \St_{\Sring})\xrightarrow{\sim}
D^b(\Acal_\Sring\otimes_{\Sring[\g^*]}\Acal_{0,\Sring}^{opp}\operatorname{-mod}^{G_\Sring}).
\end{equation}

We will need a related derived equivalence that appeared in \cite{BLin}.
Let $\pi:\tilde{\Nilp}_{\Sring}\rightarrow \g^*_{\Sring}$ denote the Springer
map. Consider the sheaf of algebras $\pi^*\Acal_{\Sring}$ on $\tilde{\Nilp}$.
It is $G_{\Sring}$-equivariant so we can consider the equivariant derived
category $\Coh^{G_\Sring}(\pi^* \Acal_{\Sring})$. Note that we have
the functor
\begin{equation}\label{eq:another_equiv}
D^b(\Coh^{G_{\Sring}}\St_{\Sring})\rightarrow D^b(\Coh^{G_{\Sring}}(\pi^*\Acal_{\Sring})),
\mathcal{F}\mapsto R\pi_{2*}(\pi_1^*\Tilt\otimes\bullet),
\end{equation}
where we write $\pi_1,\pi_2$ for the projections  $\St_{\Sring}\rightarrow \tilde{\g}_\Sring,
\tilde{\Nilp}_\Sring$. Then (\ref{eq:another_equiv}) is an equivalence, see \cite[Section 2.2]{BLin}.

Now we return to the situation when the base field is $\C$ (and we drop the subscript). One of the main result of \cite{B_Hecke}, see \cite[Theorem 1]{B_Hecke}, is an equivalence $\Hecke_0\xrightarrow{\sim}D^b(\Coh^{G}\St)$.
Composing with the base change to $\C$ of (\ref{eq:another_equiv}), we get an
equivalence $\Hecke_0\xrightarrow{\sim}  D^b(\Coh^{G}(\pi^*\Acal))$.
It was proved in \cite[Corollary 1]{BLin} that the transfer of the default t-structure on
$D^b(\Coh^{G}(\pi^*\Acal))$ to $\Hecke_0$ satisfies the conditions
analogous to those in Theorem \ref{Thm:stabilized_t_structure}.
This is the new t-structure on $\Hecke_0$.

\subsubsection{Reminder on braid group actions}
The goal of this section is to recall left and right braid group actions on
$D^b(\Coh^{G_{\Sring}}\St_{\Sring})$ following \cite{BR_braid}, see, in particular, \cite[Theorem 1.3.2]{BR_braid}.

The first step is to construct a homomorphism from the affine braid group
$\mathsf{Br}^{a}$ to the group of isomorphism classes of invertible objects
in $D^b(\Coh^{G_{\Sring}}(\tilde{\g}_{\Sring}\times_{\g_{\Sring}^*}\tilde{\g}_{\Sring}))$.
Let $\tilde{\g}_{\Sring,\Delta}$ denote the diagonal in $\tilde{\g}_{\Sring}\times_{\g_{\Sring}^*}\tilde{\g}_{\Sring}$.
The homomorphism sends $J_\lambda,\lambda\in \Lambda_0,$
to the sheaf $\Str(\lambda)$ on $\tilde{\g}_{\Sring,\Delta}$, where the normalization 
is such that the line bundle $\Str(\lambda)$ on the flag variety is ample if and only 
if $\lambda$ is strictly dominant.
The images of the generators $\mathsf{T}_s$ of the finite braid group
inside $\mathsf{Br}^{a}$ are determined as follows. Let
$\g_{\Sring}^{*,reg}$ denote the locus of regular semisimple elements in $\g_\Sring^*$
and
$$(\tilde{\g}_{\Sring}\times_{\g_{\Sring}^*}\tilde{\g}_{\Sring})^{reg}\subset
\tilde{\g}_{\Sring}\times_{\g_{\Sring}^*}\tilde{\g}_{\Sring}$$
be its preimage. For a simple reflection $s\in W$ consider the locus in
$(\tilde{\g}_{\Sring}\times_{\g_{\Sring}^*}\tilde{\g}_{\Sring})^{reg}$, where the
two Borel subalgebras are in relative position $s$. Denote it by $Z_{s,\Sring}^{reg}$.
Let $Z_{s,\Sring}$ denote the Zariski closure of $Z_{s,\Sring}^{reg}$ in
$\tilde{\g}_{\Sring}\times_{\g_{\Sring}^*}\tilde{\g}_{\Sring}$.
The element $\mathsf{T}_s$ is sent to $\Str_{Z_{s,\Sring}}$.

Note that $D^b(\Coh^{G_{\Sring}}(\tilde{\g}_{\Sring}\times_{\g^*_{\Sring}}\tilde{\g}_{\Sring}))$
acts on $D^b(\Coh^{G_{\Sring}}\St_{\Sring})$ by convolutions. This gives the first (left) action
of $\Br^{a}$ on $D^b(\Coh^{G_{\Sring}}\St_{\Sring})$.

We also have a commuting right action. Consider the derived scheme $\tilde{\Nilp}_{\Sring}\times^L_{\g^*_{\Sring}}\tilde{\Nilp}_{\Sring}$.
It makes sense to speak about the derived category of $G_{\Sring}$-equivariant
coherent sheaves on this scheme, to be denoted by
$D^b(\Coh^{G_{\Sring}}(\tilde{\Nilp}_{\Sring}\times^L_{\g^*_{\Sring}}\tilde{\Nilp}_{\Sring}))$.
It acts on $D^b(\Coh^{G_{\Sring}}\St_{\Sring})$ by convolutions from the right.
As explained in \cite[Section 1.3]{BR_braid}, there is a homomorphism from $\Br^{a}$ to the group of
invertible objects in $D^b(\Coh^{G_{\Sring}}(\tilde{\Nilp}_{\Sring}\times^L_{\g^*_{\Sring}}\tilde{\Nilp}_{\Sring}))$.
Namely, note that we have a natural morphism from the usual fiber product
$\tilde{\Nilp}_{\Sring}\times_{\g^*_{\Sring}}\tilde{\Nilp}_{\Sring}$ to the derived
tensor product. So we can view objects of $\Coh^{G_{\Sring}}(\tilde{\Nilp}_{\Sring}\times_{\g^*_{\Sring}}\tilde{\Nilp}_{\Sring})$
as objects of
$D^b(\Coh^{G_{\Sring}}(\tilde{\Nilp}_{\Sring}\times^L_{\g^*_{\Sring}}\tilde{\Nilp}_{\Sring}))$ by pushforward.

An element $J_\lambda, \lambda\in \Lambda,$ is sent to $\Str_{\tilde{\Nilp}_{\Sring, \Delta}}(\lambda)$, while $\mathsf{T}_s$ is sent to the structure sheaf of the scheme-theoretic
intersection $Z_{s,\Ring}\cap \St_{\Sring}$ (that is actually a subscheme in
$\tilde{\Nilp}_{\Sring}\times_{\g^*_{\Sring}}\tilde{\Nilp}_{\Sring}$).

The following claim is a consequence \cite[Proposition 5.8]{BL} (for the left action; the proof for the right actions
is analogous). Note that by the realization of
$\Hecke_0$ as the Iwahori constructible perverse sheaves on the affine flag variety 
we have natural actions of the affine braid group $\Br^a$ on this category from the left and from the right.

\begin{Lem}\label{Lem:braid_intertwined}
The equivalence $\Hecke_0\xrightarrow{\sim}D^b(\Coh^{G}\St)$, see Remark
\ref{Rem:BY} and  \cite[Theorem 1]{B_Hecke}, is equivariant with respect to
both $\Br^a$-actions.
\end{Lem}

\subsubsection{Order}
The goal of this section is to equip $W^{a}$ with a new partial order and study its properties. 

We choose simple reflections in $W^{a}$ with respect to the {\it anti-dominant} alcove. So,
the affine simple reflection $s_0$ is $s_{\tilde{\alpha}}t_{\tilde{\alpha}}$, where $\tilde{\alpha}$
is the dominant short root. Correspondingly, the length $w\in W^{a}$ is given by 
\begin{equation}\label{eq:length}
\ell(wt_\lambda)=\sum_{\alpha, w(\alpha)>0}|\langle\lambda,\alpha^\vee\rangle|+\sum_{\alpha,w(\alpha)<0}|\langle\lambda,\alpha^\vee\rangle-1|.
\end{equation}
Here $w\in W,\lambda\in \Lambda_0$, and $\alpha$ runs over the positive roots. 

Define the {\it stabilized length} of $x=wt_\lambda\in W^{a}$
by $\ell^{st}(x)=\ell(w)-2\langle\rho^\vee,\lambda\rangle$. This formula is motivated by the 
observation that 
\begin{itemize}
\item[(*)]
$\ell^{st}(x)=\ell(x)$ if $\lambda$ is anti-dominant. 
\end{itemize}

\begin{defi}\label{defi:stabilized_order}
For $x,y\in W^{a}$ set $x\leqslant^{st} y$ if there is a sequence of roots $\beta_1,\ldots,\beta_k$
such that $x=s_{\beta_1}\ldots s_{\beta_k}y$ and $\ell^{st}(s_{\beta_i}\ldots s_{\beta_k}y)<
\ell^{st}(s_{\beta_{i+1}}\ldots s_{\beta_k}y)$ for all $i$.
\end{defi}  

It is clear that $\leqslant^{st}$ is a partial order on $W^{a}$. Moreover, for each $\mu\in \Lambda$
the map $x\mapsto xt_\mu$ is a poset automorphism. 

The following lemma provides an alternative characterization of $\leqslant^{st}$. Let $\Lambda_0^+$ 
denote the subset of dominant weights in $\Lambda_0$.

\begin{Lem}\label{Lem:stabilized_order_characterization}
For $x,y\in W^{a}$, the following are equivalent.
\begin{enumerate}
\item $x\leqslant^{st}y$.
\item there is $\mu_0\in \Lambda^+_0$ depending on $x,y$ such that for all $\mu\in \mu_0+\Lambda^+_0$, 
we have $xt_{-\mu}\leqslant yt_{-\mu}$.
\end{enumerate}
\end{Lem}
\begin{proof}
(1)$\Rightarrow$(2): we can find $\mu_0\in \Lambda_0^+$ such that all elements $s_{\beta_i}\ldots s_{\beta_k}yt_{-\mu}$
for $\mu\in \mu_0+\Lambda_0^+$ as in Definition \ref{defi:stabilized_order} are of the form $w_i t_{\lambda_i}$ with 
$\lambda_i\in -\Lambda_0^+$. Now the desired implication follows from (*). 

(2)$\Rightarrow$(1): is similar and is also based on (*). It is left as an exercise to a reader. 
\end{proof}

In particular, the order in (2) of Theorem \ref{Thm:stabilized_t_structure} is well-defined. 

Now pick a positive number $p$ and consider the action of $W^{a}$ on $\Lambda$ as in Section 
\ref{SSS:ex_hw_interval}. Consider the action of $W^{a}$ on $\Lambda$ by 
$$w\cdot \mu:=w(\mu+\rho)-\rho, t_\lambda\cdot \mu=\mu+p\lambda, \lambda,\mu\in \Lambda, w\in W.$$

\begin{Lem}\label{Lem:order_implication}
Suppose $p>h$ (the Coxeter number). If $x\leqslant^{st}y$, then $x^{-1}\cdot (-2\rho)\leqslant y^{-1}\cdot (-2\rho)$ with respect to the dominance ordering.
\end{Lem}
\begin{proof}
It is enough to consider the situation when $x=s_\beta y$ with $\ell^{st}(x)<\ell^{st}(y)$. Let 
$y=wt_\lambda$. Then $s_\beta=s_\alpha t_{k\alpha}$ for some finite positive root $\alpha$ and $k\in \Z$. 
We have $s_\beta y=s_\alpha t_{k\alpha} w t_\lambda=(s_\alpha w)t_{k w^{-1}(\alpha)+\lambda}$.
So we know that $\ell(s_\alpha w)- 2k\langle\rho^\vee, w^{-1}(\alpha)\rangle<\ell(w)$. This inequality 
holds precisely when one of the following options holds:
\begin{itemize}
\item[(a)] $w^{-1}(\alpha)>0$ and $k>0$,
\item[(b)] $w^{-1}(\alpha)<0$ and $k\leqslant 0$.
\end{itemize}
Now note that 
\begin{align*}&y^{-1}\cdot (-2\rho)-x^{-1}\cdot (-2\rho)=(w^{-1}\cdot (-2\rho)-p\lambda)-(s_\alpha w)^{-1}\cdot (-2\rho)+ p(\lambda+kw^{-1}(\alpha))=\\
&(w^{-1}\cdot (-2\rho)-(s_\alpha w)^{-1}\cdot (-2\rho))+pkw^{-1}\alpha=(pk-\langle\alpha^\vee,\rho\rangle)w^{-1}\alpha. 
\end{align*}
Both (a) and (b) imply that the difference is a positive multiple of a positive root (note that $\langle\alpha^\vee,\rho\rangle\leqslant h$).  
\end{proof}

\begin{Cor}\label{Cor:order_implication}
$(W^a, \leqslant^{st})$ is an interval finite poset in the sense of 
Definition \ref{defi:interval_finite_poset}.
\end{Cor}
\begin{proof}
The poset structure on $W^{a}$ induced by the usual poset on structure on $W^{a}\cdot (-2\rho)$ is  interval finite. 
Thanks to Lemma \ref{Lem:order_implication}, $(W^a,\leqslant^{st})$ is interval finite.
\end{proof}

\subsubsection{Highest weight structure on $\Coh^G(\pi^*\Acal)$: positive characteristic}
Now we can state a result from \cite{ModHCO} describing a highest weight structure
on $\Coh^{G_\F}(\pi^*\Acal_\F)$, where $\F$ is an algebraically closed
field of large enough positive characteristic  $p$ (it is enough to take $p>h$).

The following claim follows from \cite[Proposition 6.5]{ModHCO}. We identify
$D^b(\Coh^{G_\F}(\pi^*\Acal_\F))$ with $D^b(\Coh^{G_{\F}}(\St_\F))$ via the specialization
of the equivalence (\ref{eq:another_equiv}) to $\F$. Consider the diagonal $\tilde{\mathcal{N}}_{diag,\F}
\subset \St_\F$.

\begin{Prop}\label{Prop:hw_new_posit}
The category $\Coh^{G_\F}(\pi^*\Acal_\F)$ is highest weight with interval finite poset
and standard objects $\Delta^{st}_\F(x):=\mathsf{T}^{-1}_{w^{-1}}\mathcal{O}_{\tilde{\mathcal{N}}_{diag, \F}}(\lambda)$
for $w\in W,\lambda\in \Lambda, x=wt_\lambda$. A highest weight order is given by $x\leqslant y$
if $x^{-1}\cdot (-2\rho)\leqslant y^{-1}\cdot (-2\rho)$.  
\end{Prop}

More precisely, \cite[Proposition 6.5]{ModHCO} produces an equivalence of $\Coh^{G^{(1)}_\F}(\pi^*\Acal^{(1)}_\F)$
with the principal block of the classical category $\OCat^{[0]}$ mentioned in Section \ref{SSS:ex_hw_interval}.

\subsubsection{Highest weight structure on $\Coh^G(\pi^*\Acal)$: result}
Now assume that the base field is $\C$. In this section we are going to state a result concerning the highest weight
structure on $\Coh^G(\pi^*\Acal)$. Recall that thanks to Remark \ref{Rem:BY}
combined with (\ref{eq:another_equiv}), we have an equivalence
\begin{equation}\label{eq:main_equiv_new_specialized}
\Hecke_0\xrightarrow{\sim}D^b(\Coh^G(\pi^*\Acal)).
\end{equation}

Then we have the following result.

\begin{Prop}\label{Prop:new_t_structure_properties}
The following claims hold:
\begin{enumerate}
\item The transfer of the default t-structure from $D^b(\Coh^G(\pi^*\Acal))$
to $\Hecke_0$ is the stabilized (a.k.a. new) t-structure.
\item
Moreover, the heart of the t-structure is a highest weight category with the following interval finite poset and
standard objects:
\begin{itemize}
\item The poset is $W^{a}$ with order $\leqslant^{st}$.
\item The standard object corresponding to $x\in W^{a}$ is $\Delta^-(xt_{-\lambda})*J_{\lambda}$ for all $\lambda$
    sufficiently dominant.
\end{itemize}
\end{enumerate}
\end{Prop}

(1) is \cite[Corollary 1]{BLin}. Part (2) will be proved below in this section.

\subsubsection{Highest weight structure on $\Coh^G(\pi^*\Acal)$: standards}\label{SSS_hw_0_reduction}
Here we define objects in $D^b(\Coh^G(\pi^*\Acal))$ that are candidates to be standard objects 
for the yet to be defined highest weight structure on $\Coh^G(\pi^*\Acal)$.

\begin{Lem}\label{Lem:new_t_structure_reduction}
Suppose $x=wt_\lambda$ with anti-dominant $\lambda$. Then, for any dominant $\mu$ we have,
$$\Delta^-(xt_{-\mu})*J_\mu\cong T^{-1}_{w^{-1}}\mathcal{O}_{\tilde{\mathcal{N}}_{diag}}(\lambda).$$
\end{Lem}
\begin{proof}
The claim that under the equivalence from \cite{B_Hecke}, the object $\Delta^-(1)$ goes to
$\mathcal{O}_{\tilde{\mathcal{N}}_{diag}}$, follows from the construction of that
equivalence. Further, note that, for $\lambda$ dominant, $\mathcal{O}_{\tilde{\mathcal{N}}_{diag}}(-\lambda)\cong \mathsf{T}^{-1}_{t_\lambda}\mathcal{O}_{\tilde{\mathcal{N}}_{diag}}$ by the construction of the
braid group action. The isomorphism we need to check reduces to
$\ell(wt_\lambda)=\ell(w)+\ell(t_\lambda)$. If $\lambda$ is anti-dominant, then for every $u\in W$, we have $\ell(ut_\lambda)=\ell(u)-2\langle\rho^\vee,\lambda\rangle$ and our claim follows.
\end{proof}

For $x\in W^{a}$, we set
$$\Delta^{st}(x):=\Delta^-(xt_{-\mu})*J_{\mu}$$
for $\mu$ sufficiently dominant. If $x=wt_\lambda$, with $\lambda$
anti-dominant, then $\Delta^{st}(x)=\Delta^-(x)$. 

\begin{Cor}\label{Cor:new_t_structure_reduction}
We have
$\operatorname{Ext}^i(\Delta^{st}(x),\Delta^{st}(y))\neq 0\Rightarrow
x\leqslant^{st}y$. Moreover, $\operatorname{End}(\Delta^{st}(x))=\C$.
\end{Cor}

\subsubsection{Preparation for proof of (2) of Proposition \ref{Prop:new_t_structure_properties}}
The proof will be given after two lemmas. Consider the object 
$\Delta^{st}_{\Sring}(x):=\mathsf{T}^{-1}_{w^{-1}}\mathcal{O}_{\tilde{\mathcal{N}}_{\Sring,diag}}(\lambda)$
in $D^b(\Coh^{G_\Sring}(\pi^*\Acal_\Sring))$. 

\begin{Lem}\label{Lem:stand_objects_base}
The following claims are true:
\begin{enumerate}
\item 
We have $\Delta^{st}_{\Sring}(x)\in \Coh^{G_\Sring}(\pi^*\Acal_\Sring)$.
\item We have that $\operatorname{End}(\Delta^{st}_{\Sring}(x))$ is a free rank $1$ module
over $\Sring$ and $\Delta^{st}_\Sring(x)$ has no higher self-extensions for all $x\in W^{a}$.
\end{enumerate} 
\end{Lem}
\begin{proof}
To prove (1) note that $\F\otimes^{L}_{\Sring}\Delta^{st}_{\Sring}(x)\cong \Delta_\F^{st}(x)$ lies in 
$\Coh^{G_\F}(\Acal_\F)$ by Proposition \ref{Prop:hw_new_posit} for all algebraically closed fields
$\F$ that are algebras over $\Sring$. This implies (1). The proof of (2) is analogous.  
\end{proof}

The next technical result we need is as follows. 

\begin{Lem}\label{Lem:Ext_fin_gen}
Let $M_\Sring,N_\Sring\in \Coh^{G_{\Sring}}(\pi^*\Acal_\Sring)$. Then 
$\operatorname{Ext}^i(M_\Sring, N_\Sring)$ (where the Ext is taken in $\Coh^{G_{\Sring}}(\pi^*\Acal_\Sring)$) is a finitely generated 
$\Sring$-module. 
\end{Lem}
\begin{proof}
Let $\underline{\Acal}_\Sring$ denote the pullback of $\Acal_\Sring$ to the fiber of
$\tilde{\Nilp}_\Sring\twoheadrightarrow \mathcal{B}_{\Sring}$ (the target is the flag variety) corresponding 
to the chosen Borel, which can be identified with the maximal nilpotent subalgebra $\mathfrak{n}_\Sring$. Then we have an equivalence $\Coh^{G_{\Sring}}(\pi^*\Acal_\Sring)\xrightarrow{\sim} \underline{\Acal}_\Sring\operatorname{-mod}^{B_\Sring}$. So we can deal with the latter category. 
The proof is in two steps. 

{\it Step 1}. Observe that $\underline{\Acal}_\Sring$ is a finitely generated algebra over $\Sring[\mathfrak{n}]$. It follows that, for all $j$, the $\Sring[\mathfrak{n}]$-module 
$\Ext^j_{\underline{\Acal}_\Sring}(\underline{M}_\Sring, \underline{N}_\Sring)$ is finitely
generated for all finitely generated $\underline{\Acal}_\Sring$-modules  $\underline{M}_\Sring, \underline{N}_\Sring$. This reduces the claim to showing that 
\begin{itemize}
\item[(*)] $H^i_{B_\Sring}(L_{\Sring})$
is finitely generated over $\Sring$ for all $L_{\Sring}\in \Sring[\mathfrak{n}]\operatorname{-mod}^{B_{\Sring}}$. 
\end{itemize}

{\it Step 2}. Here we prove (*). Let
$U$ denote the unipotent radical of $B$. Note that
$H^i_{B_\Sring}(L_\Sring)=H^i_{U_{\Sring}}(N_{\Sring})^{T_{\Sring}}$.
The weights of $T_{\Sring}$ in both $\Sring[\mathfrak{n}], \Sring[U]$
are the nonpositive linear combinations of simple roots and $0$ occurs
only in the scalars. Using this we easily reduce (*) to showing that each $T_{\Sring}$-weight
component in each of $H^i_{U_{\Sring}}(\Sring)$ is finitely generated. For $i=0$, this is
certainly true. Applying the long
exact sequence in cohomology to $0\rightarrow \Sring\rightarrow \Sring[U]\rightarrow
\Sring[U]/\Sring\rightarrow 0$ we see that
$H^i_{U_{\Sring}}(\Sring)\xrightarrow{\sim}H^{i-1}_{U_\Sring}(\Sring[U]/\Sring)$.
Note the weight components in $\Sring[U]/\Sring$ are finite rank free $\Sring$-modules.
For any weight $\lambda$ and any $i$, we can find a $B_{\Sring}$-module quotient $L'_{\Sring}$
of $\Sring[U]/\Sring$ that is a free finite rank $\Sring$-module
such that $H^{i-1}_{U_{\Sring}}(\Sring[U]/\Sring)_\lambda\xrightarrow{\sim}
H^{i-1}_{U_{\Sring}}(L'_{\Sring})_\lambda$. Now the finite generation of
$H^i_{U_{\Sring}}(\Sring)_\lambda$ is proved by induction on $i$. This implies (*)
and finishes the proof.  
\end{proof}

\subsubsection{Proof of (2) of Proposition \ref{Prop:new_t_structure_properties}}\label{SSS_Prop_new_t_hw_proof}

\begin{proof}
The proof is in several steps. For a ring $R$, we write $\mathcal{C}_R$ for 
$\Coh^{G_R}(\pi^*\Acal_R)$ and $\widetilde{\mathcal{C}}_R$ for 
$\operatorname{QCoh}^{G_R}(\pi^*\Acal_R)$.

{\it Step 1}.  Our first goal is to show that every object 
$M\in \mathcal{C}$ is a quotient of an object filtered by $\Delta^{st}(?)$'s. 

The object $M$
is defined over a finitely generated ring $\Sring'$, let $M_{\Sring'}$ be a corresponding lattice.
Replacing $\Sring'$ with its finite localization, we can assume that $M_{\Sring'}$ is flat over 
$\Sring'$. Take a maximal ideal $\mathfrak{m}\subset \Sring'$ with residue field $\kf$, 
let $\hat{\Sring}$ be the corresponding completion. We will assume that the characteristic $p$ of $\kf$
is bigger than $h$.  Set $\Sring_\ell:=\hat{\Sring}/\mathfrak{m}^\ell$. 
Set $M_{\kf}:=\kf\otimes_{\Sring'}M_{\Sring'}$.

The category  $\mathcal{C}_{\kf}$ is highest weight with standards $\Delta^{st}_{\kf}(x)$ because it is highest weight with standard 
objects $\Delta^{st}_{\bar{\kf}}(x)$ after 
base change to the algebraic closure $\bar{\kf}$ (in which case we use Proposition \ref{Prop:hw_new_posit}).
A highest weight poset was described in Proposition \ref{Prop:hw_new_posit}, denote it by $\T$. 

In particular, take  a coideal finite poset ideal $\T_0\subset \T$. We can choose $\T_0$ in such a way 
that there is a projective object $P_{\kf}\in \Cat_{\kf,\T_0}$ such that 
$M_{\kf}$ is a quotient of  $P_{\kf}$. Note that since $P_{\kf}$ is projective, it has zero 1st and 
2nd extensions in $\Cat_{\kf,\T_0}$ and hence, by Lemma \ref{Lem:derived_full_embedding_interval finite}
in $\Cat_{\kf}$. 

Note that the category $\widetilde{\Cat}_{\Sring_\ell}$ has enough injectives:
we can identify it with $\operatorname{QCoh}^{B_{\Sring_\ell}}(\underline{\Acal}_{\Sring_\ell})$, then every object embeds to an induced (as a trivial representation of $B_{\Sring_\ell}$) object from an injective $\underline{\Acal}_{\Sring_\ell}$-module.

It follows that every object in $\Cat_{\kf}$ with vanishing 1st and 2nd self-extensions 
has a unique $\Sring_\ell$-free deformation 
to an object in  $\tilde{\Cat}_{\Sring_\ell}$. It automatically lies in  $\Cat_{\Sring_\ell}$. For $P_{\kf}$, 
denote this deformation by $P_{\Sring_\ell}$. 

{\it Step 2}. Let $M_{\Sring_\ell}$ denote an object in $\Cat_{\Sring_\ell}$
that is flat over $\Sring_\ell$. Write $M_\kf$ for $\kf\otimes_{\Sring_\ell}M_{\Sring_\ell}$. 
Let $N_{\kf}$ denote an object in $\Cat_{\kf}$ (and hence of $\Cat_{\Sring_\ell}$). 
We claim that $\operatorname{Ext}^1_{\Cat_{\Sring_\ell}}(M_{\Sring_\ell}, N_{\kf})
\xrightarrow{\sim} \operatorname{Ext}^1_{\Cat_{\kf}}(M_{\kf}, N_{\kf})$. This is because 
$\kf\otimes_{\Sring_\ell}\bullet: \Cat_{\Sring_\ell}\rightarrow \Cat_{\kf}$ is left adjoint to 
the inclusion functor $\Cat_{\kf}\hookrightarrow \Cat_{\Sring_\ell}$.

Since there are no
1st extensions in $\Cat_{\kf}$ between $P_{\kf}$ and $M_{\kf}$, the epimorphism 
$P_{\kf}\twoheadrightarrow M_{\kf}$ extends to a homomorphism (and automatically an epimorphism)
$P_{\Sring_\ell}\twoheadrightarrow M_{\Sring_\ell}$. Now we use the notation of 
Lemma \ref{Lem:Ext_fin_gen} and identify $\Cat_{\Sring}$ with 
$\Coh^{B_{\Sring_\ell}}(\underline{\Acal}_{\Sring_\ell})$. Note that $T_{\Sring_\ell}$
acts on $P_{\Sring_\ell}$ with weights uniformly (in $\ell$) bounded from below. From here one deduces that 
the $T_{\hat{\Sring}}$-finite part in $\varprojlim P_{\Sring_\ell}$ is an object in 
$\Coh^{B_{\hat{\Sring}}}(\underline{\Acal}_{\hat{\Sring}})$. For the same reason, the epimorphism to
$M_{\hat{\Sring}}$ is still an epimorphism. We can pick an embedding $\hat{\Sring}\hookrightarrow \C$
and base change to $\C$ getting an epimorphism $P\twoheadrightarrow M$.

{\it Step 3}.
We still need to show that $P$ is filtered by $\Delta^{st}(?)$. The projective $P_{\kf}$ is filtered 
by $\Delta^{st}_{\kf}(?)$'s thanks to Proposition \ref{Prop:hw_new_posit}. We claim that the similar assertion is 
true for $P_{\hat{\Sring}}$. Pick a maximal element $\tau$
of $\T_0$. We can assume that $\Delta^{st}_{\kf}(\tau)$ occurs in the standard filtration of 
$P_{\kf}$, otherwise we can shrink $\T_0$. Set $\T_1:=\T_0\setminus \{\tau\}$. Let $P'_{\kf}$ denote the maximal quotient of $P_{\kf}$ lying
in $\Cat_{\kf,\T_1}$. The kernel of $P_{\kf}\twoheadrightarrow P'_{\kf}$
is the direct sum of several copies of $\Delta^{st}_{\kf}(\tau)$. We can deform $P'_{\kf}$ similarly 
to Step 2, we get an object $P'_{\hat{\Sring}}$ together with an epimorphism $P_{\hat{\Sring}}\twoheadrightarrow 
P'_{\hat{\Sring}}$. Its kernel is an $\hat{\Sring}$-flat deformation of the direct sum of some copies of 
$\Delta^{st}_{\kf}(\tau)$. Such a deformation is unique for the same Ext vanishing reason and so is 
isomorphic to the direct sum of several copies of $\Delta^{st}_{\hat{\Sring}}(\tau)$. We reduce 
to the claim that $P'_{\hat{\Sring}}$ is filtered by $\Delta^{st}_{\hat{\Sring}}(?)$'s and this can be done by
induction (say on the number of standards that occur in $P_{\kf}$). 

The claim that $P$ is filtered by $\Delta^{st}(?)$'s follows by base change now that the similar claim about 
$P_{\hat{\Sring}}$ is proved. The goal in the beginning of Step 1 is now accomplished. 
  
{\it Step 4}. Now we prove that $\Cat$ is a highest weight category with interval finite poset $\T$.
The category $\Cat$ is Noetherian because 
$\pi^*\Acal$ is a coherent sheaf of algebras on $\tilde{\mathcal{N}}$. As a module over $\C[\g^*]$,
$\Gamma(\pi^*\Acal)$ is supported on the nilpotent cone that has finitely many $G$-orbits. From here we see that 
Hom's in $\Coh^G(\pi^*\Acal)$ are finite dimensional.
So we need to show that, for every coideal finite poset ideal $\T_0\subset \T$, the Serre span of $\Delta^{st}(x)$
with $x\in \T_0$, to be denoted by $\Cat_{\T_0}$ is a highest weight category with poset 
$\T_0$ and standards $\Delta^{st}(x)$ (thanks to the previous two steps, $\Coh^G(\pi^*\Acal)$ is the union 
of the subcategories $\Cat_{\T_0}$).  In other words we need to  check axioms (i)-(v) from 
Section \ref{SSS_hw_finite_def} for $\Cat_{\T_0}$. (i) is vacuous. (ii) follows from Corollary \ref{Cor:new_t_structure_reduction}.
(iii) follows from Corollary \ref{Cor:new_t_structure_reduction} combined with Lemma \ref{Lem:order_implication}. 
(iv) has been established in the previous two steps. 

So, it remains to establish (v). It will follow once we check that the object $P\in \Cat_{\T_0}$
from Step 2 is projective. This will follow once we show that $\operatorname{Ext}^1_{\Cat_{\hat{\Sring}}}(P_{\hat{\Sring}}, M_{\hat{\Sring}})=0$
for all $M_{\hat{\Sring}}\in \Cat_{\T_0,\hat{\Sring}}$. By Lemma 
\ref{Lem:Ext_fin_gen}, all Ext's in $\Cat_{\Sring}$ are finitely generated
over $\Sring$, and hence all Ext's in $\Cat_{\hat{\Sring}}$ are finitely
generated over $\hat{\Sring}$. Therefore, to show that $\operatorname{Ext}^1_{\Cat_\Sring}(P_{\hat{\Sring}}, M_{\hat{\Sring}})$ vanishes
it is enough to show that $\operatorname{Ext}^1(P_{\hat{\Sring}}, M_{\kf})$ does. It is enough to show this Ext coincides with 
$\operatorname{Ext}_{\Cat_{\kf}}^1(P_{\kf}, M_{\kf})$. 
This is analogous to the first paragraph in Step 2. 

{\it Step 5}. It remains to show that $\Cat$ is a highest weight category with respect to the coarser partial order, namely
$\leqslant^{st}$.  This follows from 
Remark \ref{Rem:coarsest_hw_order} combined with Corollary \ref{Cor:new_t_structure_reduction}. 
\end{proof}

\subsection{Deformation of $\Coh^G(\pi^*\Acal)_{\T_0}$}
Our goal in this section is to define a deformation (over $\Ring$) of the highest weight category $\Coh^G(\pi^*\Acal)_{\T^0}$,
where $\T^0$ is a coideal finite poset ideal in $(W^a,\leqslant^{st})$. 
The deformed category will come with a full embedding (on the level of derived categories) into $\Hecke$
and in the next section we will see that this embedding is essentially surjective and finish the proof of 
Theorem \ref{Thm:stabilized_t_structure}.

\subsubsection{Evaluation}\label{SSS:evaluation_functor}
In what follows we will often use the following construction. Let $\Cat_\Ring$ be an
$\Ring$-linear abelian category. We write $\mathfrak{m}$ for the maximal ideal of $\Ring$. Let $\Cat$ denote the full subcategory 
of $\Cat_\Ring$ that consists of all objects that are killed by $\mathfrak{m}$. 
The inclusion functor $\Cat\hookrightarrow \Cat_{\Ring}$ has left adjoint,
$\C\otimes_\Ring\bullet$. The left adjoint of $D^b(\Cat)\rightarrow D^b(\Cat_\Ring)$
is the functor $\C\otimes^L_\Ring\bullet$, it will be called the {\it evaluation functor}. The same holds for the functors
between the $D^-$ categories. 

%The following claim is standard and is left to the reader. 

\begin{Lem}\label{Lem:evaluation_properties}
Suppose that every object $N$ in $\Cat_\Ring$ is  separated in the $\mathfrak{m}$-adic topology
(meaning that $N\hookrightarrow \varprojlim N/\mathfrak{m}^k N$). Then
\begin{itemize}
\item[(i)] 
for $M\in D^b(\Cat_\Ring)$ the following conditions are equivalent:
\begin{enumerate}
\item $M\in D^{b,\leqslant 0}(\Cat_\Ring)$.
\item $\C\otimes^L_{\Ring}M\in D^{b,\leqslant 0}(\Cat)$
\end{enumerate}
\item[(ii)]
for $M\in D^-(\Cat_\Ring)$, the following conditions are equivalent:
\begin{enumerate}
\item $M\in D^b(\Cat_\Ring)$,
\item $\C\otimes^L_\Ring M\in D^b(\Cat)$.
\end{enumerate}
%\item[(iii)] The direct analog of (ii) for $D^+$ holds. 
%\item[(iii)] If $M\in D^b(\Cat_\Ring)$, and $\C\otimes^L_\Ring M=0$, then $M=0$. 
\end{itemize}
\end{Lem}
\begin{proof}
We will prove (ii), the proof of (i) is similar. We write $\Ring$ as $\C[[x_1,\ldots,x_k]]$. Pick $M\in D^-(\Cat_\Ring)$
and let $i$ be minimal such that taking the derived quotient by $x_1,\ldots,x_i$ sends $M$ to $D^b$. We need to show that $i=0$. Assume the 
contrary. Replace $M$
with its derived quotient by $x_1,\ldots,x_{i-1}$ and set $x:=x_i$. Let $\underline{M}:=\C\otimes^L_{\C[[x]]}M$. Then we have the long 
exact sequence in homology
$$\ldots\rightarrow H_j(M)\xrightarrow{x\cdot} H_j(M)\rightarrow H_j(\underline{M})\rightarrow\ldots$$ 
For $j\gg 0$ we have $x:H_j(M)\twoheadrightarrow H_j(M)$. By the assumptions of the lemma, $H_j(M)$ is 
separated in the $x$-adic topology, and so $H_j(M)=0$. This contradicts the choice of $i$ and finishes the proof. 
\end{proof}

\subsubsection{Objects $\Delta^{st}_\Ring(x)$}\label{SSS:Delta_st}
As in Section \ref{SSS_hw_0_reduction}, the object 
$$\Delta^-_{\Ring}(xt_{-\mu})*J_{\mu}$$
is well-defined for $\mu$ sufficiently dominant. Denote it by $\Delta^{st}_\Ring(x)$.

Note that 
\begin{equation}\label{eq:cl_stand_base_change}
\C\otimes^L_{\Ring}\Delta^{st}_\Ring(x)\cong\Delta^{st}(x),
\end{equation}
an isomorphism of objects of $D^b(\OCat^-)=\Hecke_0$.

\begin{Lem}\label{Lem:new_stand}
We have $\Delta^{st}_\Ring(x)\in  \OCat^-_\Ring$
and $\Delta^{st}(x)\in  \OCat^-$.
\end{Lem}
\begin{proof}
To show that $\Delta^{st}_\Ring(x)\in \OCat^-_\Ring$ we use results from Section \ref{SSS_conv_exact}:
the functor $\Delta^-_\Ring(xt_{-\mu})*?$ is left t-exact, while the functor $?*J_{\mu}$ is right t-exact. 
Since $\Delta^-_\Ring(xt_{-\mu})$ and $J_{\mu}=\nabla_\Ring(t_{\mu})$ are both in $\OCat^-_\Ring$, our claim follows. 

Now we prove $\Delta^{st}(x)\in \OCat^-$. Since $\Delta^{st}(x)$ lies in the negative part of the new t-structure,
we have $\Delta^{st}(x)\in D^{b,\leqslant 0}(\OCat^-)$. The proof of $\Delta^{st}(x)\in D^{b,\geqslant 0}(\OCat^-)$
is similar to the previous paragraph. 
\end{proof}

Here are some other properties of the objects $\Delta_\Ring^{st}(x)$ that will be used below.

\begin{Lem}\label{Lem:cl_stand_properties}
The following claims hold:
\begin{enumerate}
\item $\operatorname{Ext}^i(\Delta^{st}_\Ring(x),\Delta_\Ring^{st}(y))\neq 0\Rightarrow
x\leqslant^{st}y$.
\item $\operatorname{End}(\Delta^{st}_\Ring(x))=\Ring$.
\end{enumerate}
\end{Lem}
\begin{proof}
(1) follows in a standard way from Corollary \ref{Cor:new_t_structure_reduction} combined with 
(\ref{eq:cl_stand_base_change}). (2) follows from the corresponding property of $\Delta^-_\Ring(xt_{-\lambda})$.
\end{proof}

\begin{Rem}\label{Rem:filtr_delta_cl}
Assume that $M_\Ring$ is an object in $\OCat^-_\Ring$ that is filtered by $\Delta^{st}_\Ring(x)$ with 
$x\in \T_0$, where $\T_0$ is a coideal finite poset ideal in $(W^{a},\leqslant^{st})$. Let $\T^0$ be a finite poset 
coideal in $\T_0$. Then, as in (1) of Lemma \ref{Lem:cl_stand_properties}, there is the unique 
subobject $M_{\Ring,\T^0}$ of $M_\Ring$ with the following two properties:
\begin{itemize}
\item $M_{\Ring,\T^0}$ is filtered by $\Delta^{st}_\Ring(x)$ with $x\in \T^0$,
\item $M_\Ring/M_{\Ring,\T^0}$ is filtered by $\Delta^{st}_\Ring(x)$ with $x\in \T_0\setminus \T^0$.
\end{itemize}
\end{Rem}

\subsubsection{Objects $P^{st}_{\T_0,\Ring}(x)$}
Let $\T$ denote the poset $W^{a}$ with respect to the partial order $\leqslant^{st}$. Let $\T_0$ be a coideal finite poset ideal of $\T$.
Pick $x\in \T_0$. Let $P_{\T_0}(x)$ denote a projective object satisfying axiom (v) in Section \ref{SSS_hw_finite_def}
for the label $x\in \T_0$. We can view it as an object in $D^b(\OCat^-)$. Our goal now is to construct 
an $\Ring$-flat object $P^{st}_{\T_0,\Ring}(x)$ with $P^{st}_{\T_0}(x)\cong\C\otimes_{\Ring}P^{st}_{\T_0,\Ring}(x)$.

Thanks to Proposition \ref{Prop:new_t_structure_properties},  $P^{st}_{\T_0}(x)$ is filtered by objects 
$\Delta^{st}(y)$, where $y\geqslant^{st} x$. In particular, Lemma \ref{Lem:new_stand} implies that $P^{st}_{\T_0}(x)\in 
\OCat^-$.

\begin{Lem}\label{Lem:proj_deformation}
The object $P^{st}_{\T_0}(x)\in \OCat^-$ admits a unique $\Ring$-flat deformation in $\OCat^-_\Ring$, to be denoted
by $P^{st}_{\T_0,\Ring}(x)$. This object admits an epimorphism onto $\Delta^{st}_\Ring(x)$ whose kernel is filtered
by $\Delta^{st}_\Ring(y)$ with $y\geqslant^{st}x$.  
\end{Lem}
\begin{proof}
The proof of the existence and the uniqueness of the deformation follows Step 1 of the 
proof in Section \ref{SSS_Prop_new_t_hw_proof}: the object $P^{st}_{\T_0}(x)$ has no higher 
self-extensions in $\Coh^G(\pi^*\Acal)_{\T_0}$, hence (by Lemma \ref{Lem:derived_full_embedding_interval finite})
in $\Coh^G(\pi^*\Acal)$ and hence, thanks to the derived equivalence $D^b(\OCat^-)\xrightarrow{\sim}
D^b(\Coh^G(\pi^*\Acal))$, in $\OCat^-$. So there is the unique deformation.

The proof of the claim about a filtration on $P^{st}_{\T_0,\Ring}(x)$ follows Step 3 of the proof in Section 
\ref{SSS_Prop_new_t_hw_proof}. Let $\T'$ denote the subset of $\T$ consisting of the labels $y$ such that 
$\Delta^{st}(y)$ appears in the standard filtration of $P^{st}_{\T_0}(x)$. We can find a coideal finite poset ideal $\T_1\subset \T_0$
such that $\T_0\setminus \T_1$ is finite, $\T'\cap \T_1$ is a single element, say $\tau$, and this element is maximal in $\T'$.
We have an epimorphism $P^{st}_{\T_0}(x)\rightarrow P^{st}_{\T_1}(x)$ whose kernel is the direct sum of several copies 
of $\Delta^{st}(\tau)$. The latter object has no higher self-extensions, and hence admits a unique $\Ring$-flat 
deformation in $\OCat_\Ring^-$. Thanks to Lemma \ref{Lem:new_stand}, this deformation is the direct sum of
several copies of $\Delta^{st}_\Ring(\tau)$. So the kernel of the epimorphism $P^{st}_{\T_0,\Ring}(x)\rightarrow P^{st}_{\Ring,\T_1}(x)$
is the direct sum of several copies of $\Delta^{st}_\Ring(\tau)$. Now we can replace $\T_0$ with $\T_1$
and argue by induction on the cardinality of $\T'$.
\end{proof}

\subsubsection{Algebra $A_{\Ring,\T_0}$}
We define the non-unital $\Ring$-algebra $A_{\Ring,\T_0}$ as $$\left(\bigoplus_{x,y\in \T_0}\Hom_{\OCat_\Ring^-}(P^{st}_{\Ring,\T_0}(x),P^{st}_{\Ring,\T_0}(y))\right)^{opp}.$$ 

Set
$A_{\T_0}:=\C\otimes_{\Ring}A_{\Ring,\T_0}$. Since there are no higher Ext's between the 
objects $P^{st}_{\T_0}(x)$, we see that $A_{\T_0}=\left(\bigoplus_{x,y\in \T_0}\Hom_{\OCat^-}(P^{st}_{\T_0}(x),P^{st}_{\T_0}(y)\right)^{opp}$.
It follows that the category $A_{\T_0}\operatorname{-mod}$ is identified with the highest weight 
subcategory $\Coh^G(\pi^*\Acal)_{\T_0}$, an equivalence is defined by sending $A_{\T_0}e_x$ to $P^{st}_{\T_0}(x)$
for all $x\in \T_0$. 

\begin{Lem}\label{Lem:nonunit_alg_property}
The algebra $A_{\Ring,\T_0}$ satisfies assumptions (i) and (ii) of Section \ref{SSS_loc_unit}.
Moreover, for every $x\in \T_0$, the $A_{\Ring,\T_0}$-module $A_{\Ring,\T_0} e_x$ is Noetherian.
\end{Lem}
\begin{proof}
(i) is vacuous. (ii) follows because Hom's in $D^b(\OCat^-_\Ring)$ are finitely generated over $\Ring$
and $\Hom_{\OCat_\Ring^-}(P^{st}_{\T_0,\Ring}(x),P^{st}_{\T_0,\Ring}(y))$ is an $\Ring$-flat deformation 
of $\Hom_{\OCat^-}(P^{st}_{\T_0}(x),P^{st}_{\T_0}(y))$. 

Thanks to the equivalence  $A_{\T_0}\operatorname{-mod}\xrightarrow{\sim} \Coh^G(\pi^*\Acal)_{\T_0}$,
the module $A_{\T_0}e_x$ is Noetherian.  The claim that $A_{\Ring,\T_0} e_x$ is Noetherian now follows from
Lemma \ref{Lem:deformation_Noetherian}. 
\end{proof}

\subsubsection{Full embedding}
Consider the category $A_{\Ring,\T_0}\operatorname{-mod}$, see Section \ref{SSS_loc_unit}. 
We have the functor
\begin{equation}\label{eq:F_functor}\mathcal{F}:D^-(A_{\Ring,\T_0}\operatorname{-mod})\rightarrow D^-(\OCat^-_\Ring), 
M\mapsto \left(\bigoplus_{\tau\in \T_0}P^{st}_{\T_0,\Ring}(x)\right)\otimes^L_{A_{\Ring,\T_0}}M
\end{equation}
%and 
%$$\mathcal{G}: D^+(\OCat^-_\Ring)\rightarrow D^+(A_{\Ring,\T_0}\operatorname{-mod}),
%N\mapsto \bigoplus_{x\in \T_0}R\Hom_{\OCat^-_\Ring}(P^{st}_{\Ring,\T_0}(x),N).$$
Note that the image of $\mathcal{F}$ is indeed in $D^-(\OCat^-_\Ring)$ as every object in 
$A_{\Ring,\T_0}\operatorname{-mod}$ is a quotient of a finite direct sum of objects of the form 
$A_{\Ring,\T_0}e_\tau$. When we need to indicate the dependence of $\mathcal{F}$ on $\T_0$ we write 
$\mathcal{F}_{\T_0}$.

The following lemma summarizes properties of $\mathcal{F}$.

\begin{Lem}\label{Lem:full_embedding_deformed}
The following claims hold:
\begin{enumerate}
\item The functor $\mathcal{F}$ sends $D^b(A_{\Ring,\T_0}\operatorname{-mod})$ to $D^b(\OCat^-_\Ring)$.
%The functor $\mathcal{G}$ sends $D^b(\OCat^-_\Ring)$ to $D^b(A_{\Ring,\T_0}\operatorname{-mod})$. 
\item The functor $\mathcal{F}$ is a full embedding $D^b(A_{\Ring,\T_0}\operatorname{-mod})\rightarrow D^b(\OCat^-_\Ring)$.
\item The essential image  $\mathcal{F}(D^b(A_{\Ring,\T_0}\operatorname{-mod}))$ contains the objects $\Delta^{st}_\Ring(x), x\in \T_0$.
\end{enumerate}
\end{Lem}
\begin{proof}
(1): Consider the functor $\mathcal{F}_0:D^-(A_{\T_0}\operatorname{-mod})\rightarrow D^-(\OCat^-)$ given by 
$$\left(\bigoplus_{x\in \T_0}P^{st}_{\T_0}(x)\right)\otimes^L_{A_{\T_0}}\bullet.$$
Its restriction to $D^b(A_{\T_0}\operatorname{-mod})$ is the composition of the t-exact functor 
$D^b(A_{\T_0}\operatorname{-mod})\rightarrow D^b(\Coh^G(\pi^*\Acal))$, the equivalence 
$D^b(\Coh^G(\pi^*\Acal))\xrightarrow{\sim} D^b(\OCat^-)$, and the full embedding 
$D^b(\OCat^-)\hookrightarrow D^-(\OCat^-)$. So $\mathcal{F}_0$ sends $D^b(A_{\T_0}\operatorname{-mod})$
to $D^b(\OCat^-)$.
%The equivalence $D^b(\Coh^G\pi^*\Acal)\xrightarrow{\sim} D^b(\OCat^-)$ extends to the $D^-$ categories.
%Indeed, we have $D^-(\Coh^G\pi^*\Acal)\xrightarrow{\sim} D^-(\Acal\otimes_{\C[\g^*]}\Acal_0^{opp}\operatorname{-mod}^G)$
%via taking the global sections. 

%Moreover $\mathcal{F}_0:D^-(A_{\T_0}\operatorname{-mod})\rightarrow D^-(\Coh^G\pi^*\Acal)$ is t-exact.
%It follows that $\mathcal{F}_0$ restricts to $D^b(A_{\T_0}\operatorname{-mod})\rightarrow D^b(\OCat^-)$. 

Now note that the evaluation functors $$\C\otimes^L_\Ring\bullet:
D^-(A_{\Ring,\T_0}\operatorname{-mod})\rightarrow D^-(A_{\T_0}\operatorname{-mod}), 
D^-(\OCat^-_\Ring)\rightarrow D^-(\OCat^-)$$ 
intertwine $\mathcal{F}$ and $\mathcal{F}_0$. 
Using (ii) of Lemma \ref{Lem:evaluation_properties}, we get  (1).

(2): It is enough to show that $\mathcal{F}: D^-(A_{\Ring,\T_0}\operatorname{-mod})\rightarrow D^-(\OCat^-_\Ring)$
is a full embedding. Every object in $ D^-(A_{\Ring,\T_0}\operatorname{-mod})$ is represented by a complex
whose terms are finite direct sums of $A_{\Ring,\T_0} e_\tau$'s. Each $A_{\Ring,\T_0} e_\tau$ is sent to an object in $\OCat^-_\Ring$.
So, it is enough to show that $\mathcal{F}$ is fully faithful on the objects $A_{\Ring,\T_0} e_\tau,\tau\in \T_0$.
But this is clear from the construction of the functor $\mathcal{F}$.

(3): The essential image of $\mathcal{F}$
contains the objects $P^{st}_{\T_0,\Ring}(x)$ for all $x\in \T_0$ and hence, by Lemma \ref{Lem:proj_deformation}, 
the objects $\Delta^{st}_\Ring(x)$ as well. 
\end{proof}

\subsubsection{Highest weight structure}
The goal of this part is to establish a highest weight structure on $A_{\Ring,\T_0}\operatorname{-mod}$.
We start with candidates for the standard objects. For $x\in \T_0$, let $\Delta^{st,A}_{\Ring}(x)$
be the object in $D^b(A_{\Ring,\T_0}\operatorname{-mod})$ with $\mathcal{F}(\Delta^{st,A}_{\Ring}(x))
\cong \Delta^{st}_\Ring(x)$, this makes sense thanks to (3) of Lemma \ref{Lem:full_embedding_deformed}.

\begin{Lem}\label{Lem:stand_image}
For each $x\in \T_0$, the complex $\Delta^{A,st}_\Ring(x)$ is concentrated in homological degree $0$.
\end{Lem}
\begin{proof}
This boils down to checking that $\operatorname{Ext}^i_{\OCat^-_\Ring}(P^{st}_{\T_0,\Ring}(y),\Delta^{st}_\Ring(x))=0$
for $i>0$ and $x,y\in \T_0$. Recall, Lemma \ref{Lem:derived_full_embedding_interval finite}, 
that $D^b(\Coh^G(\pi^*\Acal)_{\T_0})\rightarrow D^b(\Coh^G(\pi^*\Acal))$ is a full embedding. 
Since $P^{st}_{\T_0}(y)$ is a projective object in $\Coh^G(\pi^*\Acal)_{\T_0}$
and $\Delta^{st}(y)$ is an object of that category, we get 
$\operatorname{Ext}^i_{\OCat^-}(P^{st}_{\T_0}(y),\Delta^{st}(x))=\operatorname{Ext}^i_{\Coh^G(\pi^*\Acal)}(P^{st}_{\T_0}(y),\Delta^{st}(x))=0$ for all $i>0$. Recall that $P^{st}_{\Ring,\T_0}(y),
\Delta^{st}_\Ring(x)$ are $\Ring$-flat deformations of $P^{st}_{\T_0}(y), \Delta^{st}(x)$. It follows that
$\operatorname{Ext}^i_{\OCat^-_\Ring}(P^{st}_{\Ring,\T_0}(y),\Delta^{st}_\Ring(x))=0$ for $i>0$. 
\end{proof}

Then we have the following result.

\begin{Prop}\label{Prop:ART0_hw}
The category $A_{\Ring,\T_0}\operatorname{-mod}$ is highest weight with poset $(\T_0,\leqslant^{st})$
and standard objects $\Delta^{A,st}_\Ring(x)$.
\end{Prop}
\begin{proof}
It remains to check axioms (i)-(v) from Section \ref{SSS_hw_finite_def}. Axioms (iv) and (v) follow
from Lemma \ref{Lem:proj_deformation}. 

Next we need the following isomorphism, which follows from parts (2) and (3) of  
Lemma \ref{Lem:full_embedding_deformed}
\begin{equation}\label{eq:same_hom}
\operatorname{Hom}_{D^b(\OCat^-_\Ring)}(\Delta^{st}_\Ring(x), \Delta^{st}_\Ring(y)[i])\xrightarrow{\sim}
\operatorname{Hom}_{D^b(A_{\Ring,\T_0})}(\Delta^{A,st}_\Ring(x), \Delta^{A,st}_\Ring(y)[i]).
\end{equation}
Using (\ref{eq:same_hom}) combined with Lemma \ref{Lem:cl_stand_properties} we get axioms (ii) and (iii) from Lemma \ref{Lem:cl_stand_properties}. 

It remains to establish axiom (i): that the $A_{\Ring,\T_0}$-modules $\Delta^{A,st}_\Ring(x)$
is flat over $\Ring$, equivalently, that $\C\otimes^L_\Ring \Delta^{A,st}_\Ring(x)$ does
not have the higher homology. Note that $\mathcal{F}(\C\otimes^L_\Ring \Delta^{A,st}_\Ring(x))\cong \Delta^{st}(x)$.
It follows that  $\C\otimes^L_\Ring \Delta^{A,st}_\Ring(x)\in D^b(A_{\T_0}\operatorname{-mod})\subset 
D^b(\Coh^G(\pi^*\Acal))$ lies in $\Coh^G(\pi^*\Acal)$ and hence in $A_{\T_0}\operatorname{-mod}$ finishing the check of (i). 
\end{proof}

\subsection{Deformation of $\Coh^G(\pi^*\Acal)$}
In this section we define a deformed version of the entire category $\Coh^G(\pi^*\Acal)$. This deformation will be a highest weight category over $\Ring$
with interval finite poset $(W^{a}, \leqslant^{st})$. 
\subsubsection{Full embeddings}
The goal of this part is to prove the following claim.

\begin{Prop}\label{Prop:full_embeddings_hw_subs}
Let $\T_1\subset\T_0\subset \T$ be two coideal finite poset ideals in $(W^{a},\leqslant^{st})$. Then there is an equivalence of $A_{\Ring,\T_1}\operatorname{-mod}$ with a highest weight subcategory
in $A_{\Ring,\T_0}\operatorname{-mod}$ associated to the poset ideal $\T_1\subset \T_0$ that intertwines the full
embeddings $A_{\Ring,\T_i}\operatorname{-mod}\hookrightarrow D^b(\OCat^-_\Ring)$. 
\end{Prop}
\begin{proof}
For all $x\in \T_1$, the object $P^{st}_{\T_1,\Ring}(x)$ is the quotient of $P^{st}_{\T_0,\Ring}(x)$ by 
the subobject associated to the poset coideal $\T_0\setminus \T_1\subset \T_0$
as in Remark \ref{Rem:filtr_delta_cl}, denote it here by 
$P^{st}_{\T_0,\Ring}(x)_{\T_0\setminus \T_1}$. Inside $\Hom_{\OCat^-_\Ring}(P^{st}_{\T_0,\Ring}(x), P^{st}_{\T_0,\Ring}(y))$
consider the $\Ring$-submodule $\Hom_{\OCat^-_\Ring}(P^{st}_{\T_0,\Ring}(x), P^{st}_{\T_0,\Ring}(y))_{\T_0\setminus \T_1}$
of all homomorphisms with image in $P^{st}_{\T_0,\Ring}(y)_{\T_0\setminus \T_1}$. Set 
$$I:=\bigoplus_{x,y\in \T_0}\Hom_{\OCat^-_\Ring}(P^{st}_{\T_0,\Ring}(x), P^{st}_{\T_0,\Ring}(y))_{\T_0\setminus \T_1}.$$
From (1) of Lemma \ref{Lem:cl_stand_properties} (the $\Hom$ vanishing part) it follows that $I$ is a two-sided ideal
in $A_{\Ring,\T_0}$. Note that $e_\tau\in I$ for $\tau\in \T_0\setminus \T_1$, moreover $I$ is generated by these elements. 
It follows that we have a homomorphism
$A_{\Ring,\T_0}/I\rightarrow A_{\Ring,\T_1}$ that sends $e_x$ to $e_x$ for all $x\in \T_1$. We claim that this 
homomorphism is an isomorphism. The injectivity is clear from the construction. The surjectivity follows from 
$\operatorname{Ext}^1_{\OCat^-_\Ring}(P^{st}_{\T_0,\Ring}(x), \Delta^{st}_\Ring(y))=0$ for all $x,y\in \T_0$
that was established in the proof of Lemma \ref{Lem:stand_image}.

The isomorphism $A_{\Ring,\T_0}/I\xrightarrow{\sim} A_{\Ring,\T_1}$ gives rise to a full embedding 
$A_{\Ring,\T_1}\operatorname{-mod}\hookrightarrow A_{\Ring,\T_0}\operatorname{-mod}$, to be denoted by $\iota$. 
Since $e_x$
goes to $e_x$ for all $x\in \T_1$, the image is exactly the highest weight subcategory of $A_{\Ring,\T_0}\operatorname{-mod}$
associated to $\T_1$. 

Consider the functors $\mathcal{F}_{\T_i}:A_{\Ring,\T_i}\operatorname{-mod}\rightarrow D^b(\OCat^-_\Ring), i=0,1$.
We need to show that $\mathcal{F}_1$ and $\mathcal{F}_0\circ \iota$ (functors from $A_{\Ring,\T_1}\operatorname{-mod}$) 
are isomorphic. This will follow if we check that the non-derived functors (where we use the usual tensor product 
instead of the derived one) are isomorphic. This boils down to showing that the epimorphism
$$(\bigoplus_{x\in \T_0}P^{st}_{\T_0,\Ring}(x))/(\bigoplus_{x\in \T_0}P^{st}_{\T_0,\Ring}(x))I\twoheadrightarrow 
\bigoplus_{y\in \T_1}P^{st}_{\T_1,\Ring}(y)$$
is an isomorphism. This follows because of the Ext vanishing between $P^{st}_{\T_0,\Ring}(?)$'s and $\Delta^{st}_{\Ring}(?)$
proved in Lemma \ref{Lem:stand_image}.   
\end{proof}

\begin{Rem}\label{Rem:algebra_epim_compat}
By the construction, the specialization of the epimorphism $A_{\Ring,\T_0}\twoheadrightarrow A_{\Ring,\T_1}$ to the closed point of 
$\operatorname{Spec}(\Ring)$ is the epimorphism $A_{\T_0}\twoheadrightarrow A_{\T_1}$ coming from the inclusion
of highest weight subcategories $\Coh^G(\pi^*\Acal)_{\T_1}\subset \Coh^G(\pi^*\Acal)_{\T_0}$.
\end{Rem}

\subsubsection{Category $\OCat^{st}_\Ring$}
We define the full subcategory $\OCat^{st}_\Ring\subset \Hecke=D^b(\OCat^-_\Ring)$ as the union of the images of 
the full embeddings $\mathcal{F}_{\T_0}:A_{\Ring,\T_0}\operatorname{-mod}\hookrightarrow \Hecke$,
where $\T_0$ runs over the coideal finite poset ideals in $(W^{a},\leqslant^{st})$ and $\mathcal{F}_{\T_0}$
is the functor given by (\ref{eq:F_functor}). This is an 
$\Ring$-linear category. A morphism in $\OCat^{st}_\Ring$ between two objects $M,N$
is an epimorphism (resp., monomorphism) if it is so in every $\mathcal{F}_{\T_0}(A_{\Ring,\T_0}\operatorname{-mod})$
containing both $M$ and $N$. It is clear that $\OCat^{st}_\Ring$ is an abelian category. 

\begin{Lem}\label{Lem:Ost_hw}
The category $\OCat^{st}_\Ring$ is highest weight with poset $(W^{a},\leqslant^{st})$ and 
standard objects $\Delta^{st}_\Ring(x)$ with $x\in W^{a}$. 
\end{Lem}
\begin{proof}
It follows 
from Proposition \ref{Prop:ART0_hw} (together with the claim that each $\mathcal{F}_{\T_0}$ is a full embedding
from Lemma \ref{Lem:full_embedding_deformed}) that each  $\OCat^{st}_{\T_0,\Ring}:=\mathcal{F}_{\T_0}(A_{\Ring,\T_0}\operatorname{-mod})$
is a highest weight category with coideal finite poset $\T_0$ and standard objects $\Delta^{st}_\Ring(x), x\in \T_0$.
By Proposition \ref{Prop:full_embeddings_hw_subs}, for an inclusion of poset ideals $\T_1\subset \T_0$
with finite $\T_0\setminus \T_1$, we have the inclusion $\OCat^{st}_{\T_1,\Ring}\subset \OCat^{st}_{\T_0,\Ring}$
and it is the inclusion of the highest weight subcategory associated to the poset ideal $\T_1\subset \T_0$.
From here it easily follows that $\OCat^{st}_\Ring$ is a highest weight category in the sense of 
Definition \ref{defi:hw_interval_finite}.
\end{proof}

The following lemma shows $\OCat^{st}_\Ring$ is a deformation of $\Coh^G(\pi^*\Acal)$. Recall that we can talk about base changes of 
$\OCat^{st}_\Ring$ to $\Ring$-algebras, see Remark \ref{Rem:interval_finite_base_change}.

\begin{Lem}\label{Lem:Ost_spec}
The base change $\OCat^{st}$ of $ \OCat^{st}_\Ring$ to $\C$ is equivalent to $\Coh^G(\pi^*\Acal)$ in such a way 
that $\Delta^{st}(x):=\C\otimes_\Ring \Delta^{st}_\Ring(x)\in \OCat^{st}$ is sent to $\Delta^{st}(x)\in \Coh^G(\pi^*\Acal)$.
\end{Lem}
\begin{proof}
We can use Remark \ref{Rem:algebra_epim_compat} 
and the construction in Section \ref{SSS_interval_finite_recovery} to identify both $\OCat^{st}$
and $\Coh^G(\pi^*\Acal)$ with the category of modules over the same inverse limit of locally unital 
algebras. This identification preserves the standard objects. 
\end{proof}

\subsection{Stabilized t-structure on $\Hecke$}
The goal of this section is to finish the proof of Theorem \ref{Thm:stabilized_t_structure}
and show that the heart of the t-structure there is $\OCat^{st}_\Ring$.

\subsubsection{Generators}
We start with the following result. 

\begin{Lem}\label{Lem:generators}
The objects $\Delta^{st}_\Ring(x), x\in W^{a},$ generate the triangulated category $\Hecke=D^b(\OCat^-_\Ring)$.
\end{Lem}
\begin{proof}
We remark that $D^b(\OCat^-_\Ring)$ is generated by the standard objects $\Delta^-_\Ring(x),x\in W^{a}$.
%It is enough to show that it is generated by $\Delta^{st}_\Ring(x)=\Delta_\Ring(x t_{-\lambda})*J_{\lambda}$
%for $\lambda$ sufficiently dominant. 
Let $T_{\Ring}(s)$ denote the tilting object in $D^b(\OCat^-_\Ring)$ corresponding to
a simple reflection $s$. Since $\Delta^{st}_\Ring(x)=\Delta^-_\Ring(xt_{-\lambda})*J_\lambda$ for $\lambda$
sufficiently dominant, we see that $T_{\Ring}(s)*\Delta_\Ring^{st}(x)$ fits into a distinguished triangle 
with $\Delta^{st}_\Ring(x)$ and $\Delta^{st}(sx)$ (in some order). So the triangulated 
subcategory in $\Hecke$ spanned by the objects $\Delta^{st}_\Ring(x)$ is closed under convolutions 
on the left with $T_{\Ring}(s)$ and hence with convolutions with $\mathsf{T}_s$ and $\mathsf{T}_s^{-1}$ (that are given by cones of 
$\operatorname{id}\rightarrow T_\Ring(s)*?$ and $T_{\Ring}(s)*?\rightarrow \operatorname{id}$). 
It follows that every $\Delta^-_\Ring(x)$ is in the subcategory generated by the $\Delta^{st}_\Ring(?)$ 
finishing the proof. 
\end{proof}

Now note that, for every coideal finite poset ideal $\T_0$ in $(W^{a},\leqslant^{st})$, we can view $D^b(\OCat^{st}_{\T_0,\Ring})=\mathcal{F}_{\T_0}(A_{\Ring,\T_0}\operatorname{-mod})$
as a full subcategory in $\Hecke$.

\begin{Cor}\label{Cor:category_union}
The category $\Hecke$ is the union of its full subcategories $D^b(\OCat^{st}_{\T_0,\Ring})$. 
\end{Cor}
\begin{proof}
The union contains all objects $\Delta^{st}_\Ring(x)$ and we can use Lemma \ref{Lem:generators}. 
\end{proof}

\subsubsection{t-structure}
We set 
\begin{equation}\label{eq:st_t_str_neg_pos}
\Hecke^{st,\leqslant 0}:=\bigcup D^{b,\leqslant 0}(\OCat^{st}_{\T_0,\Ring}), 
\Hecke^{st,\geqslant 0}:=\bigcup D^{b,\geqslant 0}(\OCat^{st}_{\T_0,\Ring}).
\end{equation}

\begin{Prop}\label{Prop:stabilized_t_structure}
The following claims are true:
\begin{enumerate}
\item (\ref{eq:st_t_str_neg_pos})  defines a t-structure on $D^b(\OCat_\Ring^-)$.
\item The heart is $\OCat^{st}_\Ring$.
\item This t-structure is bounded. 
\end{enumerate}
\end{Prop}
\begin{proof}
(1) follows from the observation that the inclusions $D^b(\OCat_{\T_1,\Ring}^{st})\subset D^b(\OCat_{\T_0,\Ring}^{st})$
(for an inclusion of poset ideals with finite complement) and Corollary \ref{Cor:category_union}. (2) follows from the observation that the heart of the union is the union of the hearts.

Let us show (3). Take $M\in \Hecke$. It lies in $D^b(\OCat^{st}_{\T_0,\Ring})$ for some $\T_0$. The restriction of 
our t-structure to  $D^b(\OCat^{st}_{\T_0,\Ring})$ is the tautological t-structure, which is bounded. So $M$ only has finitely 
many non-vanishing cohomology modules. It follows that our t-structure is bounded.  
\end{proof}

Note that this finishes the proof of (2) of Theorem \ref{Thm:stabilized_t_structure}.

\subsubsection{Derived equivalence}
Thanks to Proposition \ref{Prop:stabilized_t_structure}, the full inclusion of the heart $\OCat^{st}_\Ring\hookrightarrow \Hecke$ gives rise to the realization functor 
$D^b(\OCat^{st}_\Ring)\rightarrow \Hecke$.

\begin{Prop}\label{Prop:deformed_derived_equiv}
The functor $D^b(\OCat^{st}_\Ring)\rightarrow \Hecke$ is an equivalence. 
\end{Prop}
\begin{proof}
We need to show that 
\begin{equation}\label{eq:Ext_morphism}\operatorname{Ext}^i_{\OCat^{st}_\Ring}(M,N)\rightarrow \operatorname{Hom}_{\Hecke}(M,N[i])
\end{equation}
for all $M,N\in \OCat^{st}_\Ring$. Note that $M,N\in \OCat^{st}_{\T_0,\Ring}$ for some coideal finite poset ideal $\T_0$. 
Then (\ref{eq:Ext_morphism}) intertwines the maps 
$$\operatorname{Ext}^i_{\OCat^{st}_{\T_0,\Ring}}(M,N)\rightarrow \operatorname{Ext}^i_{\OCat^{st}_\Ring}(M,N),
\operatorname{Ext}^i_{\OCat^{st}_{\T_0,\Ring}}(M,N)\rightarrow \operatorname{Hom}_{\Hecke}(M,N[i]).$$
The former is an isomorphism by Lemma \ref{Lem:derived_full_embedding_interval finite}, and the latter is an 
isomorphism by Lemma \ref{Lem:full_embedding_deformed}. So (\ref{eq:Ext_morphism}) is an isomorphism. 
\end{proof}

This finishes the proof of (3) of Theorem \ref{Thm:stabilized_t_structure}.

\subsubsection{Completion of the proof of Theorem \ref{Thm:stabilized_t_structure}}\label{SSS_proof_stab_compl}
It remains to prove (1) of the theorem.

\begin{proof}
First, we claim that the evaluation functor $\C\otimes^L_\Ring\bullet:D^b(\OCat^-_\Ring)\rightarrow D^b(\OCat^-)$ intertwines the right braid group actions, i.e., we need to show that 
$\C\otimes^L_\Ring(\mathcal{F}*\mathsf{T}_x)\cong (\C\otimes^L_\Ring\mathcal{F})*\mathsf{T}_x$ for all $x$. To see this notice that the evaluation functor intertwines the actions of $\Hecke$ (from the left). The left hand side above is isomorphic to $\mathcal{F}*(\C\otimes^L_\Ring \nabla^-_\Ring(x))$,
while the right hand side is isomorphic to $\mathcal{F}*\nabla^-(x)$.  So our claim reduces to checking 
that $\C\otimes^L_\Ring \nabla_\Ring^-(x)\cong \nabla^-(x)$, which is a tautology.

Now (1) of the theorem easily follows by applying (1) of Lemma \ref{Lem:evaluation_properties} to the categories 
$D^b(\OCat^{st}_\Ring)$ and $\Hecke$. The details are left as an exercise to the reader. 
\end{proof}

\subsection{Properties of $\OCat^{st}_\Ring$}
\subsubsection{Soergel functor}
Recall the Soergel functor $\Vfun: \Hecke\rightarrow D^b(\Ring\operatorname{-bimod})$ from 
Section \ref{SSS_Soergel_fun}. We are interested
in the restriction of this functor to the category of standardly filtered objects in $\OCat^{st}_\Ring$.

\begin{Lem}\label{Lem:Vfun_st}
The following claims hold:
\begin{enumerate}
\item $\Vfun(\Delta^{st}_\Ring(x))\cong \Ring_x$.
\item $\Vfun$ is faithful on standardly filtered objects in $\OCat^{st}$.  
\end{enumerate}
\end{Lem}
\begin{proof}
By (1) of Lemma \ref{Lem:V_O_negative_properties}, $\Vfun$ is monoidal. Now (1) follows from (2) of 
Lemma \ref{Lem:V_O_negative_properties} and the construction of $\Delta^{st}_\Ring(x)$ in 
Section \ref{SSS:Delta_st}.

We proceed to (2). The functor $?*J_{\lambda}$ is a self-equivalence of $\OCat^{st}_\Ring$, this follows, for example,
from (1) of Theorem \ref{Thm:stabilized_t_structure}. Similarly, $?*J_\lambda$ is a self-equivalence of 
$\OCat^{st}$, this follows from results of \cite{BLin} recalled in Section \ref{SSS_BLin}. 
We have $\Vfun(?*J_\lambda)\cong \Vfun(?)\otimes_{\Ring}\Ring_{t_\lambda}$. So the homomorphism
$$\Hom_{\OCat^{st}_\Ring}(M,N)\rightarrow \Hom_{D^b(\Ring\operatorname{-bimod})}(\Vfun(M),\Vfun(N))$$
is injective if and only if the induced homomorphism
$$\Hom_{\OCat^{st}_\Ring}(M*J_\lambda,N*J_\lambda)\rightarrow \Hom_{D^b(\Ring\operatorname{-bimod})}(\Vfun(M*J_\lambda),\Vfun(N*J_\lambda))$$
is injective for some $\lambda$. 

Now pick $x_1,x_2\in W^{a}$. For a sufficiently dominant $\lambda$, we have $\Delta^{st}(x_i)*J_{-\lambda}=\Delta(x_it_{-\lambda})$
The injectivity in question follows from (4) of Lemma \ref{Lem:V_O_negative_properties}. This proves (2). 
\end{proof}

\subsubsection{Base change to $\operatorname{Frac}(\Ring)$}
We write $\F$ for $\operatorname{Frac}(\Ring)$. Consider the base change $\OCat^{st}_\F$, see Remark \ref{Rem:interval_finite_base_change}
for the discussion of base change for interval finite highest weight categories. 

\begin{Lem}\label{Lem:ss_base_change}
The category $\OCat^{st}_\F$ is semisimple.
\end{Lem}
\begin{proof}
After base change to $\F$, every Bott-Samelson bimodule becomes the direct sum of graph bimodules $\F_x, x\in W^{a}$
(the localizations of the graph bimodules $\Ring_x$). It follows that $\OCat^-_\F$ is a semisimple category, in particular,
there are no higher Ext's between the standard objects. From here and the construction of the standard objects 
$\Delta^{st}_\Ring(x), x\in W^{a},$ we deduce that there are no higher Ext's between the standard objects
in $\OCat^{st}_\F$, hence this category is semisimple. 
\end{proof}

\begin{Cor}\label{Cor:Vfun_ff_Ost}
The restriction of $\Vfun$ to the category of standardly filtered objects in $\OCat^{st}_\Ring$ is a fully
faithful embedding into $\Ring\operatorname{-bimod}$.
\end{Cor}
\begin{proof}
The claim that $\Vfun$ sends standardly filtered objects to $\Ring\operatorname{-bimod}$ (rather than just complexes of bimodules)
follows from (1) of Lemma \ref{Lem:Vfun_st}. Now we prove that it is fully faithful on standardly filtered objects. 
Note that $\Vfun$ is 
\begin{itemize}
\item[(a)] faithful on standardly filtered objects in $\OCat^{st}$, (2) of Lemma \ref{Lem:Vfun_st}, 
\item[(b)] fully faithful on $\OCat^{st}_\F$, this follows from (1) of Lemma \ref{Lem:Vfun_st}
combined with Lemma \ref{Lem:ss_base_change}.
\end{itemize} 
Now the claim of this lemma follows similarly to the proof of Lemma \ref{Lem:RS_0_faith}. Namely, take 
$M,N\in \OCat_\Ring^{st,\Delta}$. The objects $M,N,\Vfun(M),\Vfun(N)$ are flat over $\Ring$. It follows 
that $\Hom_{\OCat^{st}_\Ring}(M,N), \Hom_{\Ring\operatorname{-bimod}}(\Vfun(M),\Vfun(N))$ are 
reflexive $\Ring$-modules. So it is enough to show that 
\begin{equation}\label{eq:localized_Hom_iso11}
\Hom_{\OCat^{st}_\Ring}(M,N)_\p\xrightarrow{\sim} \Hom_{\Ring\operatorname{-bimod}}(\Vfun(M),\Vfun(N))_\p
\end{equation}
for every prime ideal $\p\subset \Ring$ of height $1$. Let $\kf_\p$ denote the residue field of 
$R_\p$. Then, by (a) combined with Remark \ref{Rem:RS4}, we get
\begin{equation}\label{eq:localized_Hom_emb11}\Hom_{\OCat^{st}_{\kf_\p}}(M_{\kf_\p},N_{\kf_\p})\hookrightarrow \Hom_{\Ring\otimes\kf_\p\operatorname{-mod}}(\Vfun(M_{\kf_\p}),\Vfun(N_{\kf_\p})).
\end{equation}
Here we write $M_{\kf_\p}:=\kf_\p\otimes_\Ring M, N_{\kf_\p}:=\kf_{\p}\otimes_\Ring N$.
By (b), we have
\begin{equation}\label{eq:localized_Hom_iso12}
\Hom_{\OCat^{st}_{\F}}(M_{\F},N_{\F})\xrightarrow{\sim} 
\Hom_{\Ring\otimes\F\operatorname{-mod}}(\Vfun(M_{\F}),\Vfun(N_{\F})).
\end{equation}
Similarly to the proof of Lemma \ref{Lem:RS_0_faith}, (\ref{eq:localized_Hom_emb11}) and (\ref{eq:localized_Hom_iso12}) imply (\ref{eq:localized_Hom_iso11}). 
\end{proof}

\subsubsection{Ext's between standards: result}
We would like to understand $\Ext^1$ between standard objects in $\OCat^{st}_\Ring$, more precisely, its localization 
to $\Ring_{\p}$ for height $1$ prime ideals. Here is our main result. 

\begin{Prop}\label{Prop:Ext1_stand_Ost}
We have $\Ext^1_{\OCat^{st}_\Ring}(\Delta^{st}_\Ring(x),\Delta^{st}_\Ring(y))_{\p}\neq 0$ if and only if the following two conditions
hold:
\begin{itemize}
\item[(a)] $xs_\alpha=y$ for a real root $\alpha$ and $x<^{st}y$.
\item[(b)] $\p=(\alpha)$.
\end{itemize}
If these (a) and (b) hold, then $\Ext^1_{\OCat^{st}_\Ring}(\Delta^{st}_\Ring(x),\Delta^{st}_\Ring(y))_{\p}$ is isomorphic 
to the residue field of $\Ring_\p$.
\end{Prop}

This proposition will be proved in the subsequent parts of this section.

\subsubsection{Non-vanishing implies (a) and (b)}\label{SSS_non_vanish_Ext}
Here we prove the $\Rightarrow$ part of the theorem. 

\begin{proof}
Suppose $\Ext^1_{\OCat^{st}_\Ring}(\Delta^{st}_\Ring(x),\Delta^{st}_\Ring(y))_{\p}\neq 0$. 
By 
Corollary \ref{Cor:Vfun_ff_Ost}, we have 
\begin{equation}\label{eq:Ext1_inclusion_Ost}
\Ext^1_{\OCat^{st}_\Ring}(\Delta^{st}_\Ring(x),\Delta^{st}_\Ring(y))\hookrightarrow 
\Ext^1_{\Ring\operatorname{-bimod}}(\Vfun(\Delta^{st}_\Ring(x)),\Vfun(\Delta^{st}_\Ring(y))).
\end{equation}
By (1) of Lemma \ref{Lem:Vfun_st}, $\Vfun(\Delta^{st}_\Ring(x))\cong \Ring_x, 
\Vfun(\Delta^{st}_\Ring(y)))\cong \Ring_y$, so the target of 
(\ref{eq:Ext1_inclusion_Ost}) is $\operatorname{Ext}^1_{\Ring\operatorname{-bimod}}(\Ring_x,\Ring_y)$.
The latter Ext is nonvanishing after the localization to $R_\p$ (with $\p$ of height $1$) 
if and only if the intersection of the graphs of $x$ and $y$ has codimension 
$1$ in each, which is equivalent to $x^{-1}y$ being a reflection, say $s_\alpha$.  In this case we have 
$\operatorname{Ext}^1_{\Ring\operatorname{-bimod}}(\Ring_x,\Ring_y)\cong \Ring/(\alpha)$, as a $\Ring$-module. 
So its base change to $\Ring_\p$ is the residue field of $\Ring_\p$ if $\p=(\alpha)$ and vanishes otherwise. 
\end{proof}

The proof also shows that it remains to show that 
\begin{itemize}
\item[(*)] if $x^{-1}y=s_\alpha$ and $\p=(\alpha)$, then 
$$\Ext^1_{\OCat^{st}_\Ring}(\Delta^{st}_\Ring(x),\Delta^{st}_\Ring(y))_{\p}\neq 0.$$
\end{itemize}

\subsubsection{Proof of non-vanishing}
\begin{proof}[Proof of (*)]
The proof is in several steps. 

{\it Step 1}. As in the proof of Lemma \ref{Lem:Vfun_st}, we can replace $x,y$ with $xt_{-\lambda},yt_{-\lambda}$
for $\lambda$ sufficiently dominant. This changes the $\Ring$-module structure by twisting it with $t_{-\lambda}$.
With this, we can achieve that $\Delta^{st}_\Ring(x)\cong \Delta^-_\Ring(x), 
\Delta^{st}_\Ring(y)\cong \Delta^-_\Ring(y)$. We then further reduce to proving the analog of (*)
for $\Ext^1_{\OCat^-_\Ring}(\Delta^-_\Ring(x), \Delta^-_\Ring(y))$. 

{\it Step 2}. The main part of the proof is to reduce to the case when $\alpha$ is a simple root.
We assume that $\alpha$ is positive and not simple.  
Pick a simple reflection $s$ and set $x':=xs, y':=ys, \alpha':=s(\alpha)$ and $\p'=(\alpha')$. We claim that 
\begin{equation}\label{eq:Ext_localized_iso}
\Ext^1_{\OCat^-_\Ring}(\Delta^-_\Ring(x), \Delta^-_\Ring(y))_\p\xrightarrow{\sim} \Ext^1_{\OCat^-_\Ring}(\Delta^-_\Ring(x'), \Delta^-_\Ring(y'))_{\p'},
\end{equation}
an $\Ring$-semilinear isomorphism (with respect to the twist of the $\Ring$-action by $s$). Let $\alpha_s$ denote the 
simple root corresponding to $s$. 

Note that $\Delta^-_\Ring(x')=\Delta^-_\Ring(x)*\Delta^-_\Ring(s)$ if $\ell(x')>\ell(x)$, else  
$\Delta^-_\Ring(x')=\Delta^-_\Ring(x)*\nabla^-_\Ring(s)$. We have a homomorphism $\Delta^-_\Ring(s)\rightarrow \nabla^-_\Ring(s)$
coming from the highest weight structure on $\OCat^-_\Ring$. It is an isomorphism outside the divisor $(\alpha_s)$. 
Since $(\alpha)\neq (\alpha_s)$, we have $\Delta^-_{\Ring_\p}(s)\xrightarrow{\sim} \nabla^-_{\Ring_\p}(s)$. This implies 
(\ref{eq:Ext_localized_iso}).

Assume that $\alpha$ is conjugate to a simple root $\beta$, so that $s_\alpha= us_{\beta}u^{-1}$ for $u\in W^a$,
which we can take to be the shortest possible. Let $s$ be such that $\ell(su)<\ell(u)$. Then $s_{\alpha'}=(su)s_\beta (su)^{-1}$.
To complete the reduction to the case when $\alpha$ is simple we can now argue by induction on $\ell(u)$. 

{\it Step 3}. Now we have $y=xs$ and, by our assumption on the order, $x<y$ in the Bruhat order (otherwise,
$\Ext^1$ vanishes). So, $\Delta^-_{\Ring}(y)=\Delta^-_{\Ring}(x)*\Delta^-_\Ring(s)$. We then have an $\Ring$-linear 
isomorphism 
$$\Ext^1_{\OCat^-_\Ring}(\Delta^-_\Ring(x), \Delta^-_\Ring(y))\xrightarrow{\sim} \Ext^1_{\OCat^-_\Ring}(\Delta^-_\Ring(1), \Delta^-_\Ring(s)).$$
The latter $\Ring$-module is easily seen to be isomorphic to $\Ring/(\alpha_s)$, which finishes the proof. 
\end{proof}

\subsection{Stabilized t-structure for singular blocks}\label{SS_stab_tstr_singular}
The goal of this section is to establish a singular version of Theorem \ref{Thm:stabilized_t_structure}.

\begin{Prop}\label{Prop:stabilized_t_structure_singular}
The following claims are true:
\begin{enumerate}
\item
There is unique, \underline{stabilized}, t-structure on $\,_J\Hecke$ 
such that its negative part $\,_J\Hecke^{st,\leqslant 0}$ coincides with
the full subcategory of $\Hecke$ consisting of all objects $\mathcal{F}$ such that
$\mathcal{F}*J_\lambda\in \,_J\Hecke^{\leqslant 0}$ (for the perverse t-structure).
The stabilized t-structure is bounded. 
The functors $\pi_J$ and $\pi^J$ are t-exact for the stabilized t-structure. 
\item
Moreover, the heart of the t-structure, to be denoted by $\,_J\OCat_\Ring^{st}$, is a highest weight category with the following interval finite poset and
standard objects:
\begin{itemize}
\item The poset is $W_J\backslash W^{a}$ with order $W_Jx\preceq W_Jy$ if $W_Jxt_{-\lambda}\leqslant W_Jyt_{-\lambda}$
for all $\lambda$ sufficiently dominant (below we will see that this gives a well-defined
order).
\item The standard object corresponding to $W_Jx$ is $\pi_J(\Delta_\Ring(xt_{-\lambda}))*J_{\lambda}$ for all $\lambda$
    sufficiently dominant (again, below we will see that this is well-defined).
\end{itemize}
\item We have a derived equivalence $D^b(\,_J \OCat_\Ring^{st})\xrightarrow{\sim} \,_J\Hecke$
given by the realization functor.
\end{enumerate}
\end{Prop}

\subsubsection{Reflection functors}
The goal of this section is to prove the following result. We write $T_{\Ring,J}$ for the indecomposable tilting in $\OCat_\Ring^-$
corresponding to the longest element in $W_J$, it corresponds to the Soergel bimodule $\Ring\otimes_{\Ring^J}\Ring$. 
We note that 
\begin{equation}\label{eq:convolution_projection}
T_{\Ring,J}*\bullet\cong \pi^J\circ \pi_J.
\end{equation}
%let $\tilde{\alpha}$ be the maximal 
%short root, $\tilde{s}=s_{\tilde{\alpha}}$. We write $\tilde{s}_0$ for $\tilde{s}t_{\tilde{\alpha}}$
%(note that the simple reflection $s_0$ is $\tilde{s}t_{-\tilde{\alpha}}$). Observe that the reflection 
%hyperplane for $\tilde{s}_0$ is $\tilde{\alpha}^\vee=-1$ (while for $s_0$ it is $\tilde{\alpha}^\vee=1$).
%Also note that $\tilde{\alpha}^\vee=-1-p$ is a wall of the $p$-alcove containing $-2\rho$.

\begin{Lem}\label{Lem:reflection_Ost} 
\begin{itemize}
\item[(i)] 
the  endo-functor $T_{\Ring,J}*\bullet$
of $D^b(\OCat^-_\Ring)$ restricts to an (automatically t-exact) endo-functor of $\OCat^{st}_\Ring$
\item[(ii)]
and, moreover, $T_{\Ring,J}*\bullet$ sends $\Delta_\Ring^{st}(x)$ to an object filtered by 
$\Delta_{\Ring}^{st}(ux)$ for $u\in W_J$ (in the order dictated 
by the highest weight structure) and each such $u$ occurs exactly once.  
\end{itemize}
\end{Lem}
\begin{proof}
The functor $T_{\Ring,J}*?$ is self-biadjoint, so, as long as we know that it preserves $D^b(\OCat_\Ring)^{st,\leqslant 0}$, 
the claim of (i) will follow. Note that $T_{\Ring,J}*?$ is t-exact for the usual t-structure. 
To prove that $T_{\Ring,J}*\bullet$ preserves $D^b(\OCat_\Ring)^{st,\leqslant 0}$, we use the fact that $T_{\Ring,J}*?$ commutes with 
$?*J_\lambda$ for all $\lambda\in \Lambda$ and (1) of Theorem \ref{Thm:stabilized_t_structure}. This proves (i).

To prove (ii) observe that for sufficiently dominant $\lambda$ all elements of the form $uxt_{-\lambda}$ with $u\in W_J$
have the form $wt_\mu$, where $\mu$ is anti-dominant. It follows that $\Delta^{st}_\Ring(uxt_{-\lambda})=
\Delta_\Ring(uxt_{-\lambda})$. Now we just use the standard fact about the standard filtration on $T_{\Ring,J}*\Delta_\Ring(xt_{-\lambda})$.
\end{proof} 

\subsubsection{Proof of (1) of Proposition \ref{Prop:stabilized_t_structure_singular}}
\begin{proof}
The proof is in several steps. 

{\it Step 1}. We claim that for $\mathcal{F}\in \,_J\Hecke$ the following two conditions are equivalent:
\begin{itemize}
\item[(a)] $\mathcal{F}$ lies in the Karoubian envelope of $ \pi_J(\Hecke^{st,\leqslant 0})$,
\item[(b)] and $\pi^J(\mathcal{F})\in \Hecke^{st,\leqslant 0}$.
\end{itemize}
(b)$\Rightarrow$(a) follows from (\ref{eq:piJ_identity}). (a)$\Rightarrow$(b) follows from
(\ref{eq:convolution_projection}) combined with (i) of Lemma \ref{Lem:reflection_Ost}. 
Similarly, the following two conditions 
\begin{itemize}
\item[(a')] $\mathcal{F}$ lies in the Karoubian envelope of $\pi_J(\Hecke^{st,\geqslant 0})$,
\item[(b')] and $\pi^J(\mathcal{F})\in \Hecke^{st,\geqslant 0}$.
\end{itemize}  

{\it Step 2}. Define $\,_J\Hecke^{st,\leqslant 0}$ (resp., $\,_J\Hecke^{st,\geqslant 0}$) as the full
subcategory of $\,_J\Hecke$ of objects satisfying the equivalent conditions (a),(b) (resp., 
the equivalent conditions (a'),(b')). These are Karoubian subcategories.
Since $\pi_J$ and $\pi^J$ commute with $?*J_\lambda$ for all $\lambda$, we see that 
$$\,_J\Hecke^{st,\leqslant 0}=\{\mathcal{F}\in \,_J\Hecke| \mathcal{F}*J_\lambda\in \,_J\Hecke^{\leqslant 0},\forall \lambda\}.$$

{\it Step 3}. We claim that the subcategories  $\,_J\Hecke^{st,\leqslant 0},\,_J\Hecke^{st,\geqslant 0}$ constitute the non-positive
and the non-negative parts of a t-structure. A check that these subcategories are stable under appropriate 
homological shifts is trivial. To check that there are no homomorphisms from  
$\mathcal{F}\in\,_J\Hecke^{st,\leqslant 0}$ to $\mathcal{G}\in\,_J\Hecke^{st,>0}$
we use that $\mathcal{F}$ lies in the Karoubian envelope of $ \pi_J(\Hecke^{st,\leqslant 0})$, while $\pi^J(\mathcal{G})\in\Hecke^{st,>0}$,
and the claim that $\pi_J$ is left adjoint to $\pi^J$. 

It remains to show that any $\mathcal{F}\in \,_J\Hecke$ fits into a distinguished triangle 
of the form $\mathcal{F}^{\leqslant 0}\rightarrow \mathcal{F}\rightarrow \mathcal{F}^{>0}\xrightarrow{+1}$
with $\mathcal{F}^{\leqslant 0}\in\,_J\Hecke^{st,\leqslant 0}, \mathcal{F}^{> 0}\in\,_J\Hecke^{st,> 0}$.
Let $\mathcal{C}$ denote the full subcategory of $\Hecke$ of all $\mathcal{F}$ such that 
such a triangle exists. Then the assignment $\,_J\tau_{\leqslant 0}:\mathcal{C}\rightarrow \,_J\Hecke^{st,\leqslant 0},\mathcal{F}\mapsto\mathcal{F}^{\leqslant 0},$
is a part of a functor. On morphisms, $\,_J\tau_{\leqslant 0}$ sends $\varphi:\mathcal{F}\rightarrow \mathcal{F}$
 to a unique morphism $\varphi_{\leqslant 0}: \mathcal{F}^{\leqslant 0}\rightarrow \mathcal{F}^{\leqslant 0}$
such that $\iota\circ\varphi_{\leqslant 0}= \varphi\circ \iota$, where $\iota:\mathcal{F}^{\leqslant 0}
\rightarrow \mathcal{F}$ is the natural morphism. Similarly, we can define a functor
$\,_J\tau_{> 0}: \mathcal{C}\rightarrow \,_J\Hecke^{>0}$. From this construction applied to an idempotent $\varphi$, 
it is easy to see that $\mathcal{C}$ is a Karoubian subcategory.  

Let $\tau_{\leqslant 0},\tau_{>0}$ be the truncation functors for the stabilized t-structure on $\Hecke$.
Then we have a distinguished triangle
$$\pi_J \tau_{\leqslant 0}\pi^J\mathcal{F}\rightarrow \pi_J\pi^J\mathcal{F}\cong\mathcal{F}^{\oplus |W_J|}
\rightarrow \pi_J \tau_{> 0}\pi^J\mathcal{F}\xrightarrow{+1}.$$
So $\mathcal{F}^{\oplus |W_J|}\in \mathcal{C}$. Since $\mathcal{C}$ is Karoubian, we get $\mathcal{F}\in \mathcal{C}$.
We have checked that $(\,_J\Hecke^{st,\leqslant 0},\,_J\Hecke^{st,\geqslant 0})$ is a t-structure.

{\it Step 4}. The claim that the functors $\pi_J,\pi^J$ are t-exact with respect to the stabilized 
t-structure easily follows from the construction. Now let us prove that the t-structure is bounded. 
Take $M\in \,_J\Hecke$. Then $\pi^J M$ has only finitely many nonzero homology groups because 
the t-structure on $\Hecke$ is bounded. Then we use the claim that $\pi_J$ is t-exact combined with 
(\ref{eq:piJ_identity}) to conclude that $M$ has only finitely many nonzero homology groups with respect
to the stabilized t-structure. 
%The claim that the t-structure is bounded easily follows by combining the 
%claim that the stabilized t-structure on $\OCat^{st}_\Ring$ is bounded with (\ref{eq:piJ_identity})
\end{proof}

\subsubsection{Proof of (2) of Proposition \ref{Prop:stabilized_t_structure_singular}}
\begin{proof}
{\it Step 1}.
Note that $\,_J\OCat_\Ring^{st}$ can be characterized as the Karoubian envelope of $\pi_J(\OCat^{st}_\Ring)\subset \,_J\Hecke$
or, equivalently, as the full subcategory of all objects $\mathcal{F}\in \,_J\Hecke$ with $\pi^J(\mathcal{F})\in \OCat^{st}_\Ring$.
This follows from Step 1 of the proof of part (1). Recall that $\OCat^{st}_\Ring$ is a Noetherian category with Hom's finitely 
generated over $\Ring$. (\ref{eq:piJ_identity}) now implies that $\,_J\OCat_\Ring^{st}$ has analogous properties. 

{\it Step 2}. We claim that the transitive closure of the relation 
$\leqslant^{st}$ on $W_J\backslash W^{a}$ given by $W_Jx\leqslant^{st}W_Jy$ if there are $x'\in W_Jx, y'\in W_Jy$ 
with $x'\leqslant^{st}y'$ is a partial order. Indeed, what we need to prove is that 
\begin{enumerate}
\item 
there no cycles, i.e., collections of elements 
$W_Jx_1,\ldots,W_Jx_k\in W_J\backslash W^{a}$ with $W_Jx_1<^{st}W_Jx_2<^{st}\ldots<^{st}W_Jx_k<^{st}W_Jx_1$.  
\end{enumerate}
For this, note that $\leqslant^{st}$ on $W_J\backslash W^{a}$ is invariant under $W_Jx\mapsto W_Jxt_\lambda$ for all
$\lambda\in \Lambda$. So we can assume that all elements in the cosets $W_Jx_i$ have the form $wt_\mu$ with 
anti-dominant $\mu$. Now our claim follows from the fact that the singular Bruhat order is well-defined. 
Note also $W_Jx\leqslant^{st}W_Jy$ is equivalent to $W_Jxt_{-\lambda}\leqslant W_Jyt_{-\lambda}$ for sufficiently dominant 
$\lambda$. 

{\it Step 3}. Fix $W_Jx\in W_J\backslash W^{a}$. Recall that $\Delta^-_\Ring(W_Jx):=\pi_J(\Delta^-_\Ring(x))$ 
only depends on $W_Jx$. Set $\Delta^{st}_\Ring(W_Jx):=\Delta^-_\Ring(W_Jxt_{-\lambda})*J_\lambda$
for $\lambda$ sufficiently dominant, equivalently, $\Delta^{st}_\Ring(W_Jx):=\pi_J(\Delta^{st}_\Ring(x))$.
As these definitions are equivalent, they are independent of the choice of $\lambda$ or choice of $x$ in $W_Jx$.

Now we will show that $\,_J\OCat^{st}_\Ring$ is a highest weight category with poset 
$(W_J\backslash W^{a}, \leqslant^{st})$.

Axiom (i): $\Delta^{st}_\Ring(W_Jx)$ is flat over $\Ring$ because it is isomorphic to $\pi_J(\Delta^{st}_\Ring(x))$,
the object $\Delta^{st}_\Ring(x)$ is flat over $\Ring$ ((ii) of Theorem \ref{Thm:stabilized_t_structure}) 
and $\pi_J$ is exact, part (i) of this proposition.   

Axioms (ii) and (iii): follow from $\Delta^{st}_\Ring(W_Jx)=\Delta^-_\Ring(W_Jxt_{-\lambda})*J_\lambda$
for sufficiently dominant $\lambda$.

Axiom (iv): Let $M\in \,_J\OCat^{st}_{\Ring,\T_0}$ for a coideal finite poset ideal $\T_0\subset W_J\backslash W^{a}$.
Let $\tilde{\T}_0\subset W^{a}$ be its preimage. This is also a coideal finite poset and $\pi^J(M)\in 
\OCat^{st}_{\Ring,\tilde{\T}_0}$. Note that the latter object is nonzero thanks to (\ref{eq:piJ_identity}). 
So we can find $x\in \T_0$ with 
$$\Hom_{\,_J\OCat^{st}_{\Ring,\T_0}}(\pi_J\Delta^{st}_\Ring(x),M)=
\Hom_{\OCat^{st}_{\Ring,\tilde{\T}_0}}(\Delta^{st}_\Ring(x),\pi^JM)\neq 0.$$

Axiom (v): The functors $\pi_J,\pi^J$ restrict to the subcategories $\OCat^{st}_{\Ring,\tilde{\T}_0}$
and $\,_J\OCat^{st}_{\Ring,\T_0}$. So $\pi_J$ sends projective objects in 
$\OCat^{st}_{\Ring,\tilde{\T}_0}$ to projective objects in $\,_J\OCat^{st}_{\Ring,\T_0}$.
Axiom (v) easily follows from the construction of  the objects $\Delta^{st}_\Ring(W_Jx)$.  
\end{proof}

\subsubsection{Proof of (3) of Proposition \ref{Prop:stabilized_t_structure_singular}}
\begin{proof}
As in the proof of Proposition \ref{Prop:deformed_derived_equiv}, we need to show that 
\begin{equation}\label{eq:Ext_isom_singular}
\operatorname{Ext}^i_{\,_J\OCat^{st}_\Ring}(M,N)\xrightarrow{\sim} \operatorname{Hom}_{\,_J\Hecke}(M,N[i]).
\end{equation}
We know the analogous statement for the regular blocks. Suppose $\tilde{N}\in \OCat^{st}_\Ring$
is such that $N\cong \pi_J \tilde{N}$.
Then the isomorphism
$$\operatorname{Ext}^i_{\OCat^{st}_\Ring}(\pi^J M,\tilde{N})\xrightarrow{\sim} \operatorname{Hom}_{\Hecke}(\pi^J M,\tilde{N}[i]).$$
is intertwined with (\ref{eq:Ext_isom_singular}) by the following two isomorphisms
$$\operatorname{Ext}^i_{\,_J\OCat^{st}_\Ring}(M,N)\xrightarrow{\sim} \operatorname{Ext}^i_{\OCat^{st}_\Ring}(\pi^J M,\tilde{N}), 
\operatorname{Hom}_{\,_J\Hecke}(M,N[i])\xrightarrow{\sim} \operatorname{Hom}_{\Hecke}(\pi^J M,\tilde{N}[i]).$$
So (\ref{eq:Ext_isom_singular}) is an isomorphism. 

In general, every object $N$ in $\,_J\OCat^{st}_\Ring$ is a direct summand in the object of the form $\pi_J\tilde{N}$
(with $\tilde{N}=\pi^J N$). This shows (3). 
\end{proof}

\subsubsection{Soergel functor}
Similar to the regular case we have the functor $$\Vfun: \,_J\Hecke\rightarrow D^b(\Ring^J\operatorname{-}\Ring\operatorname{-bimod}).$$
We have the following analog of Lemmas \ref{Lem:Vfun_st}, \ref{Lem:ss_base_change} and Corollary \ref{Cor:Vfun_ff_Ost}. 

Note that we can consider the base change  $\,_J\OCat^{st}$ of $\,_J\OCat_\Ring^{st}$ to $\C$. We still have biadjoint functors $\pi_J:\OCat^{st}\rightarrow 
\,_J\OCat^{st}, \pi^J:\,_J\OCat^{st}\rightarrow \OCat^{st}$.

\begin{Lem}\label{Lem:Vfun_singular}
The following claims are true:
\begin{enumerate}
\item For every $x\in W^{a}$, we have $\Vfun(\Delta^{st}_\Ring(W_Jx))\cong \Ring_x$ as an $\Ring^J$-$\Ring$-bimodule.
\item The functor $\Vfun$ is faithful on standardly filtered objects in $\,_J \OCat^{st}$. 
\item The base change of $\,_J\OCat_\Ring$ to $\F:=\operatorname{Frac}(\Ring)$ is semisimple.
\item The restriction of $\Vfun$ to the category of standardly filtered objects in $\,_J\OCat^{st}_\Ring$ is a fully
faithful embedding into $\Ring^J\operatorname{-}\Ring\operatorname{-bimod}$.
\end{enumerate}
\end{Lem}
\begin{proof}
Note that we can consider the restriction functor $\pi_J:\Ring\operatorname{-bimod}\rightarrow \Ring^J\operatorname{-}\Ring\operatorname{-bimod}$ and its biadjoint $\pi^J:\Ring^J\operatorname{-}\Ring\operatorname{-bimod}$.
We also have 
\begin{align}\label{eq:Soergel_projection}
&\Vfun\circ\pi_J\cong \pi_J\circ \Vfun,\\\label{eq:Soergel_projection_adjoint}
&\Vfun\circ \pi^J\cong \pi^J\circ \Vfun.
\end{align}

(i): This follows from (\ref{eq:Soergel_projection}) combined with (i) of Lemma \ref{Lem:Vfun_st}.

(ii): We need to show that $\Vfun$ does not kill a nonzero homomorphism $\Delta^{st}(W_Jx)\rightarrow \Delta^{st}(W_Jy)$.
Note that $\Delta^{st}(W_Jx)=\pi_J(\Delta^{st}(x))$, while $\pi^J(\Delta^{st}(W_Jy))$ is standardly filtered.
By (ii) of Lemma \ref{Lem:Vfun_st}, $\Vfun$ gives an injective map 
$$\Hom_{\OCat^{st}}(\Delta^{st}(x), \pi^J(\Delta^{st}(W_Jy)))\rightarrow \Hom_{\Ring\operatorname{-mod}}(\Vfun(\Delta^{st}(x)),
\Vfun\pi^J(\Delta^{st}(W_Jy))).$$
Now we can use (\ref{eq:Soergel_projection}) and  (\ref{eq:Soergel_projection_adjoint}) to establish the claim in the first sentence. 

(iii) and (iv): the proofs repeat those of Lemma \ref{Lem:ss_base_change} and Corollary \ref{Cor:Vfun_ff_Ost}, respectively. 
\end{proof}

\subsubsection{Ext's between standards}
We finish this section with a singular analog of Proposition \ref{Prop:Ext1_stand_Ost}. 

\begin{Lem}\label{Lem:Ext1_stand_Ost_sing}
Let $\p$ be a height $1$ prime ideal in $\Ring$. 
We have $\Ext^1_{\,_J\OCat^{st}_\Ring}(\Delta^{st}_\Ring(W_Jx),\Delta^{st}_\Ring(W_Jy))_{\p}\neq 0$ if and only if the following two conditions
hold:
\begin{itemize}
\item[(a)] There is a real root $\alpha$ such that $W_Jxs_\alpha=W_Jy$ with $W_Jx<^{st}W_Jy$ and
\item[(b)] $\p=(\alpha)$.
\end{itemize}
If these (a) and (b) hold, then $\Ext^1_{\OCat^{st}_\Ring}(\Delta^{st}_\Ring(x),\Delta^{st}_\Ring(y))_{\p}$ is isomorphic 
to the residue field of $\Ring_\p$.
\end{Lem}
\begin{proof}
Note that 
\begin{equation}\label{eq:Ext_equality_sing}
\Ext^1_{\,_J\OCat^{st}_\Ring}(\Delta^{st}_\Ring(W_Jx),\Delta^{st}_\Ring(W_Jy))_\p\xrightarrow{\sim} 
\Ext^1_{\OCat^{st}_\Ring}(\Delta^{st}_\Ring(x),\pi^J\Delta^{st}_\Ring(W_Jy))_\p.
\end{equation}
If the left hand side is nonzero, so is $\Ext^1_{\OCat^{st}_\Ring}(\Delta^{st}_\Ring(x),\Delta^{st}_\Ring(y'))_\p$
for some $y'\in W_J y$. It follows from Proposition \ref{Prop:Ext1_stand_Ost} that 
$xs_\alpha=y',x<^{st}y'$ and $\p=(\alpha)$. Note that $W_Jy\neq W_J ys_\alpha$ and that taking $\Ext^1$ and then
localizing at $\p$ gives the same result as localizing at $\p$ and then taking $\Ext^1$. So for $y''\in W_J y'$ different
from $y'$ we get, thanks to Proposition \ref{Prop:Ext1_stand_Ost}, that there are no extensions between 
$\Delta^{st}(y')_\p$ and $\Delta^{st}(y'')_\p$ and also between $\Delta^{st}(x)_\p$ and $\Delta^{st}(y'')_\p$.
It follows that the right hand side of (\ref{eq:Ext_equality_sing}) is $\Ext^1_{\OCat^{st}_\Ring}(\Delta^{st}_\Ring(x),\Delta^{st}_\Ring(y'))_\p$.
Now the claim of the lemma follows from Proposition \ref{Prop:Ext1_stand_Ost}. 
\end{proof}

\begin{Rem}\label{Rem:Ext1_stand_Ost_sing}
%Similarly to Section \ref{SSS_non_vanish_Ext}, we have 
%$$\Ext^1_{\,_J\OCat^{st}_\Ring}(\Delta^{st}_\Ring(W_Jx),\Delta^{st}_\Ring(W_Jy))_\p\hookrightarrow 
%\Ext^1_{\Ring^J\otimes \Ring}(\Ring_x,\Ring_y)_\p.$$
Similarly to the proof of Lemma \ref{Lem:Ext1_stand_Ost_sing}, we have
$$\Ext^1_{\Ring^J\otimes \Ring}(\Ring_x,\Ring_y)\xrightarrow{\sim} \Ext^1_{\Ring\otimes \Ring}(\Ring_x,\Ring\otimes_{\Ring^J}\Ring_y)$$
and, for $\p=(\alpha)$ and $y=xs_\alpha$, we have 
$$\Ext^1_{\Ring^J\otimes \Ring}(\Ring_x,\Ring_y)_\p\xrightarrow{\sim} \Ext^1_{\Ring\otimes \Ring}(\Ring_x,\Ring_y)_\p.$$
The latter $\Ring_\p$-module is isomorphic to the residue field of $\Ring_\p$.
\end{Rem}

\subsection{Towards the coherent realization of $\OCat^{st}_\Ring$}\label{SS_coherent_full_deformation}
Recall, Proposition \ref{Prop:new_t_structure_properties}, that $\OCat^{st}$ is equivalent to
$\Coh^G(\pi^*\Acal)$. One can ask if there is an analogous description of $\OCat^{st}_\Ring$.

Realizing a smaller deformation of $\Coh^G(\pi^*\Acal)$ in coherent terms is straightforward. 
Namely, set $\underline{\Ring}:=\Ring/(\hbar)=\C[\h^*]^{\wedge_0}$. Set $\tilde{\g}^\wedge:=\operatorname{Spec}(\C[\h^*]^{\wedge_0})\times_{\h^*}\tilde{\g}$.
We write $\hat{\pi}$ for the natural morphism $\tilde{\g}^{\wedge}\rightarrow \g^*$. 
Then one can show that $\OCat^{st}_{\underline{\Ring}}\cong \Coh^G(\hat{\pi}^*\Acal)$. 

To explain a conjectural realization of $\OCat^{st}_\Ring$ we will need quantizations. 
The variety $\tilde{\g}$ admits a formal quantization: the $\hbar$-adic completion of 
the $T$-invariants in the homogenized differential operators $D_{G/U,\hbar}$, where 
$U$ stands for the unipotent radical of $B$. Denote this sheaf by $\tilde{\mathcal{D}}_\hbar$.
Then, compare to \cite{SRA_der,quant_res}, one can uniquely quantize the vector bundle $\mathcal{E}$
on $\tilde{\g}$ to a right $\tilde{\mathcal{D}}_\hbar$-module $\mathcal{E}_\hbar$. 
Set $\Acal_\hbar:=\operatorname{End}_{\mathcal{D}_\hbar}(\mathcal{E}_\hbar)$, this is  a
flat deformation of $\Acal$ over $\C[[\hbar]]$.

We consider the sheaf of algebras $\tilde{\mathcal{D}}_\hbar^{opp}\widehat{\otimes}_{\C[[\hbar]]}\Acal_\hbar$ on $\tilde{\g}^\wedge$.
It is acted on by $G$ and so it makes sense to talk about $G$-equivariant coherent sheaves of 
$\tilde{\mathcal{D}}_\hbar^{opp}\widehat{\otimes}_{\C[[\hbar]]}\Acal_\hbar$-modules, denote the corresponding 
category by $\Coh^G(\tilde{\mathcal{D}}_\hbar^{opp}\widehat{\otimes}_{\C[[\hbar]]}\Acal_\hbar)$. 
We expect that there is a full embedding $\OCat^{st}_\Ring\hookrightarrow \Coh^G(\tilde{\mathcal{D}}_\hbar^{opp}\widehat{\otimes}_{\C[[\hbar]]}\Acal_\hbar)$ whose essential 
image consists of all objects $\mathcal{F}$ such that $\mathcal{F}/\hbar \mathcal{F}$
lies in $\Coh^G(\hat{\pi}^*\Acal)$, equivalently, the $\C[\g^*]$-actions on $\mathcal{F}/\hbar \mathcal{F}$
coming from $\hat{\pi}^*$ and $\C[\g^*]\hookrightarrow \Acal$ coincide.

\section{Quantum category $\OCat$}
\subsection{Quantum group}
\subsubsection{Basic definition}\label{SSS_quantum_basic}
Let $\g$ be a semisimple Lie algebra with a fixed triangular decomposition
$\g=\mathfrak{n}^-\oplus \mathfrak{h}\oplus \mathfrak{n}$. We write $W$ for the Weyl group. 
Let $\Lambda\supset \Lambda_0$ denote the weight and root lattices. 
Let $I$ be an indexing set for simple (co)roots, we write $\alpha_i$
for the simple root indexed by $i\in I$ and $\varpi_i$ for the corresponding fundamental weight. 
There is a unique $W$-invariant symmetric bilinear form on $\h^*$ such that $(\alpha,\alpha)=2$ for all 
short roots $\alpha$ in all simple summands of $\g$. Set $d_i:=(\varpi_i,\alpha_i)$, this is an element of 
$\{1,2,3\}$, and $(\varpi_j,\alpha_i)=\delta_{ij}d_i$.   
 
Let $v$ be an indeterminate. We can consider the quantum group $\UU$ over $\C(v)$. It is generated by the elements $\underline{E}_i,\underline{F}_i$ for $i\in I$ and $K_\nu, \nu\in \Lambda,$ subject to the usual relations, see, e.g., 
\cite[Chapter 4]{Jantzen}. 
Consider the $\C(v)$-subalgebras $\underline{\UU}^-,\UU^0,\underline{\UU}^+\subset \UU$ generated by the elements
$\underline{F}_i$ (resp., $K_\nu$, $\underline{E}_i$). We have the triangular decomposition
\begin{equation}\label{eq:triang_generic}
\underline{\UU}^-\otimes_{\C(v)}\UU^0\otimes_{\C(v)}\underline{\UU}^{+}_v\xrightarrow{\sim} \UU.
\end{equation}

Inside $\UU$ consider the {\it mixed form} $U^{mix}_{v}$, the $\C[v^{\pm 1}]$-subalgebra
generated by the elements $\underline{F}_i,K_\nu$, and the divided powers $\underline{E}_i^{(n)}$ for $i\in I$ and $n>0$. %, and the $v$-bimomial coefficients
%$${K^\nu \brack n}:=\prod_{i=1}^{n}\frac{K^\nu v^{i-1}-K^{-\nu}v^{1-i}}{v^i-v^{-i}}.$$
Let $\underline{U}^-_v$ (resp., $U^0_v,\underline{U}^{mix,+}_v$) denote the intersections of $U^{mix}_v$ with $\underline{\UU}^-$
(resp., $\UU^0,\underline{\UU}^+$). It is easy to see that we have
the triangular decomposition
\begin{equation}\label{eq:triang_mix}
\underline{U}^-_v\otimes_{\C[v^{\pm 1}]}U^0_v\otimes_{\C[v^{\pm 1}]}\underline{U}^{+,mix}_v\xrightarrow{\sim} U^{mix}_v.
\end{equation}
The subalgebra $\underline{U}^-_v$ is generated by the elements $\underline{F}_i, i\in I$, while $\underline{U}^{+,mix}_v$
is generated by the elements $\underline{E}_i^{(n)}$.
Also define the subalgebra $\underline{U}_v^{\geqslant 0,mix}:=U^0_v\otimes_{\C[v^{\pm 1}]}\underline{U}^{+,mix}_v$.

We recall the PBW basis in $\underline{U}^-_v$.  Fix a reduced expression $w_0=s_{i_1}\ldots s_{i_\ell}$ of the longest element $w_0\in W$. 
Define the positive roots $\alpha_j=s_{i_j}s_{i_{j+1}}\ldots s_{i_{\ell-1}}\alpha_{i_\ell}$. The braid group $\Br^a$ acts on 
$\UU_v$ by automorphisms. Set $F_{\alpha_j}=\mathsf{T}_{i_j}\mathsf{T}_{i_{j+1}}\ldots \mathsf{T}_{i_{\ell-1}}F_{\alpha_{i_\ell}}$.
Then the elements $F_{\alpha_j}$ lie in $\underline{U}^-_v$ and, moreover, the elements $\prod_{i=1}^\ell F_{\alpha_i}^{d_i}$, 
where $d_1,\ldots,d_\ell$ are nonnegative integers, form a basis on the $\C[v^{\pm 1}]$-module $\underline{U}^-_v$.  

We will need to describe the subalgebra $U^0_v$. 

\begin{Lem}\label{Lem:U0_generators}
The $\C[v^{\pm 1}]$-subalgebra $U^0_v$ is generated by the elements $K_\nu, \nu\in \Lambda,$ and the elements $\frac{K_i-K_i^{-1}}{v^{d_i}-v^{-d_i}}$
for $i\in I$ (where, as usual, $K_i$ is the shorthand for $K_{\alpha_i}$).  
\end{Lem}
\begin{proof}
This is a consequence of the formula for $\underline{E}_i^{(r)}\underline{F}_i^r$ (that can be deduced, say, from \cite[Lemma 1.7]{Jantzen})
and (\ref{eq:triang_generic}).
\end{proof}

Note that $\UU^0$ is identified
$\C(v)\otimes_{\C[v^{\pm 1}]}U^0_v$, etc.

We will also consider the usual De Concini-Kac form $U^{DK}_v$, the subalgebra of $\UU$ generated by $\underline{E}_i,\underline{F}_i$ and $K^\nu$.
We denote its positive and negative parts by $\underline{U}^+_v$ and $\underline{U}^-_v$. Note that we have a natural 
homomorphism $U^{DK}_v\rightarrow U^{mix}_v$. 

\subsubsection{Sevostyanov's modification}\label{SS_Sevostyanov}
An important feature of the negative part $U(\mathfrak{n}^-)$ extensively used in Representation theory is that this 
algebra has a ``nondegenerate character'': a homomorphism $\psi:U(\mathfrak{n}^-)\rightarrow \C$ that is nonzero
on the generator $f_\alpha$ corresponding to any simple root $\alpha$. A direct analog of this claim is false 
for $\underline{U}^-_v$ because of the form of the quantum Serre relations.  

It was observed by Sevostyanov, \cite[Section 1]{Sevostyanov}, that one can modify the subalgebras $\underline{\UU}^{\pm}, 
\underline{U}^\pm_v$ so that the modified versions have non-degenerate characters. Namely, one considers  
elements $E_i:=K_{\nu_i}\underline{E}_i, F_i:=K_{-\nu_i}\underline{F}_i$ for suitable 
$\nu_i\in \Lambda$ for $i\in I$. Let $\UU^{\pm}$ be the subalgebra in $\UU$ generated 
by the elements $E_i$ (for +) and $F_i$ (for $-$), and let $U^-_v, U^{+,mix}\subset U^{mix}_v$ have the similar meaning
(where $U^{+,mix}_v$ is generated by the divided powers). 
The choice of $\nu_i$ is such that $\UU^{-}$ admits a homomorphism $\psi$ to $\C(v)$ that sends $F_i$ to $1$
(and the same is true for $\UU^+$). We note that the image of $U_v^-$ under $\psi$ lies in $\C[v^{\pm 1}]$. 

\subsection{Quantum category $\OCat$}\label{SS_O_quant}
The setting is as follows. 
Fix $\epsilon$, a root of $1$ of odd order $d$ (and coprime to $3$ if $\g$ has simple summands of type $G_2$). 

Consider  the ring $\Ring:=\C[[\h^*,\hbar]]$ that already appeared in 
Section \ref{SSS_Hecke_setting}. In particular, we have the embedding $\iota:\h^*\rightarrow \Ring$
sending $\nu\in \h^*$ to the element $(\nu,\cdot)\in \h\subset \Ring$.  
Fix also a group homomorphism $\zeta:\Lambda\rightarrow \C^\times$.

\subsubsection{Definition}
Thanks to Lemma \ref{Lem:U0_generators}, for each $\lambda\in \Lambda_0$, there is a unique algebra 
homomorphism $U^0_v\rightarrow \Ring$ satisfying
\begin{equation}\label{eq:chilambda_def}
v\mapsto q:=\epsilon e^{2\pi\sqrt{-1}\hbar/d}, K_\nu\mapsto q^{(\lambda,\nu)}e^{2\pi\sqrt{-1}\iota(\nu)}\zeta(\nu).
\end{equation}

For an $\Ring$-algebra $R$, we abuse the notation and write $\chi_\lambda$ for the composed homomorphism
$U^0_v\rightarrow \Ring\rightarrow R$.

Now let $M$ be an $R$-module together with an $R$-module decomposition $M=\bigoplus_{\lambda\in \Lambda_0} M_\lambda$
to be called the {\it weight decomposition}. We turn $M$ into $U^0_v$-module by requiring that $U^0_v$ acts on 
$M_\lambda$ via $\chi_\lambda$.

By a {\it deformed weight module} over  $R$ and $U_v^{mix}$ we mean a $U_v^{mix}\otimes_{\C} R$-module
together with the decomposition $M=\bigoplus M_\lambda$ so that the action of $U_v^0$ on $M$
is as before, and $F_i M_\lambda\subset M_{\lambda-\alpha_i}$, while $E_i^{(\ell)}M_\lambda\subset M_{\lambda+\ell \alpha_i}$ for all $\lambda\in \Lambda,\ell\in \Z_{>0}$ and $i\in I$.
A homomorphism of deformed weight modules is a $U^{mix}_v\otimes_\C R$-linear map
preserving the weight decompositions.

The following definition appeared in \cite[Section 2]{Situ1} (and for $R=\Ring$ in \cite{BBASV}), compare to \cite{Gaitsgory}.

\begin{defi}\label{defi:deformed_quant_O}
By definition, the category $\OCat^\zeta_R$ is a full subcategory in the category of finitely generated deformed weight $U^{mix}_v\otimes_\C R$-modules $M$ satisfying the following conditions:
\begin{itemize}
\item Each $M_\lambda$ is a finitely generated $R$-module.
\item The weights are bounded from above: there is $\lambda_0\in \Lambda_0$
such that $M_\lambda\neq \{0\}$ implies $\lambda\leqslant \lambda_0$ with respect to the dominance order on $\Lambda_0$.
\end{itemize}
\end{defi}

When $\zeta=1$ (the ``integral block''), we write $\OCat_R$ instead of $\OCat^1_R$. 

Here is an example of an object in $\OCat^\zeta_R$. Take a $U_v^{mix,\geqslant 0}$-module $\underline{M}$
satisfying the following condition
\begin{itemize}
\item[(*)] $\underline{M}$ admits a finite weight decomposition $\underline{M}=\bigoplus_{\lambda\in \Lambda_0}\underline{M}_\lambda$,
where each $\underline{M}_\lambda$ is a finitely generated $R$-module. Moreover, $U_v^0$ acts on $\underline{M}_\lambda$
by $\chi_\lambda$. 
\end{itemize}
Then the induced
module $U_v^{mix}\otimes_{U_v^{mix,\geqslant 0}}\underline{M}$ is in $\OCat^\zeta_R$. An important special
case is the Verma module $\Delta^\zeta_\Ring(\lambda)$. It arises via this construction from the free rank $1$
$\Ring$-module $\underline{M}$ (to be denoted by $\Ring_\lambda$), where all $E_i^{(\ell)}$ act by $0$, and $U^0_v$ acts via $\chi_\lambda$.

\subsubsection{Highest weight structure}\label{SSS_quantum_hw_structure}
Recall that we equip $\Lambda_0$ with the usual poset structure as in Example \ref{Ex:poset_highest_weights}.

\begin{Lem}\label{Lem:hw_O_quantum}
The category $\OCat^\zeta_\Ring$ is highest weight with poset $\Lambda_0$ and standard objects $\Delta_\Ring^\zeta(\lambda),\lambda\in \Lambda$. 
\end{Lem}
\begin{proof}
First, we need to check that the category $\OCat^\zeta_\Ring$ is Noetherian and the Hom modules are finitely generated over 
$\Ring$. Note that every module in $\OCat^\zeta_\Ring$ is finitely generated over $\U_{v}^-\otimes_{\C}\Ring$. By 
\cite[Proposition 1.7]{DCK}, the algebra $\U_v^-$ admits an ascending bounded below filtration whose associated graded 
is a twisted polynomial algebra. The usual proof of the Hilbert basis theorem applies and we see that a twisted polynomial
algebra over a Noetherian ring is Noetherian. It follows that $\OCat^{\zeta}_\Ring$ is Noetherian. The claim that 
the Hom modules are finitely generated over $\Ring$ is an easy consequence of the definition.

Now we need to verify that the conditions of Definition \ref{defi:hw_interval_finite} hold. 
Let $\T_0\subset \Lambda_0$ be a coideal finite poset ideal. The full subcategory $\OCat^{\zeta}_{\T_0,\Ring}$ consists exactly 
of the modules $M$ such that $M_\lambda\neq \{0\}\Rightarrow \lambda\in \T_0$. It remains to check axioms (i)-(v) from 
Section \ref{SSS_hw_finite_def}. (i) follows from (\ref{eq:triang_mix}) combined with the observation that every 
weight $\C[v^{\pm 1}]$-submodule in $U^-_v$ is free over $\C[v^{\pm 1}]$ thanks to the quantum PBW theorem. Axioms
(ii), (iii) and (iv) are standard and are left as exercise. 

Let us check (v). 
For $\beta\in \operatorname{Span}_{\Z_{\geqslant 0}}(\alpha_i)$, we write $U^{+,mix}_{\Ring,\beta}$
for the $\beta$-weight space in $U_v^{+,mix}\otimes_{\C[v^{\pm 1}]}\Ring$. This is a finitely generated free $\Ring$-module,
thanks to the PBW theorem. For $\lambda\leqslant \mu$, let $U^{+,mix}_\Ring[\lambda,\T_0]$
denote the quotient of $U_v^{+,mix}\otimes_{\C[v^{\pm 1}]}\Ring$ by the direct sum of all $U^{+,mix}_{\Ring,\beta}$ with $\lambda+\beta\not\in \T_0$. This direct sum is a two-sided ideal in $U_v^{+,mix}\otimes_{\C[v^{\pm 1}]}\Ring$, so $U^{+,mix}_\Ring[\lambda,\T_0]$ is a $U_v^{+,mix}\otimes_{\C[v^{\pm 1}]}\Ring$-module.  We equip it with a $U^{\geqslant 0,mix}\otimes_{\C[v^{\pm 1}]}\Ring$-module structure
by requiring that $U^0_v$ acts on $U^{+,mix}_{\Ring,\beta}$ by $\lambda+\beta$. Since $\T_0$ is coideal finite, only finitely many weight submodules in 
$U^{+,mix}_\Ring[\lambda,\T_0]$ are nonzero.
 In particular,
$U^{+,mix}_\Ring[\lambda,\T_0]$ becomes a deformed weight module.

Set $$P^\zeta_{\T_0,\Ring}(\lambda):=U_v^{mix}\otimes_{U_v^{\geqslant 0,mix}}U^+_\Ring[\lambda,\T_0].$$
It is easy to see that  $P^\zeta_{\T_0,\Ring}(\lambda)$ is a projective object in $\OCat^\zeta_{\T_0,\Ring}$
(the Hom functor in $\OCat^\zeta_{\T_0,\Ring}$ from this object sends $M\in \OCat^\zeta_{\T_0,\Ring}$
to the weight submodule $M_\lambda$). It admits an epimorphism onto $\Delta_\Ring^\zeta(\lambda)$.
Thanks to the construction, the kernel is filtered by objects of the form 
$$U_{v}^{mix}\otimes_{U_v^{\geqslant 0,mix}}(\Ring_\lambda\otimes_\Ring U^{+,mix}_{\Ring,\beta})\cong \Delta_\Ring^\zeta(\lambda+\beta)^{\oplus N}$$
for $\beta>0$ and a suitable positive integer $N$ (note that $U^{+,mix}_{\Ring,\beta}$ is a free finite rank $\Ring$-module). This finishes checking (v).
\end{proof}

\begin{Rem}
The same argument implies that $\OCat^\zeta_R$ is a highest weight category over $R$ for any Noetherian $\Ring$-algebra
$R$. Moreover, it is obtained by base changing $\OCat^\zeta_\Ring$ to $R$.
\end{Rem}

\subsection{Structural results about quantum groups}
We write $\Sring$ for the completion of $\C[v^{\pm 1}]$ at $\epsilon$, where $\epsilon$ is as in the beginning of 
Section \ref{SS_O_quant}, and let $\C_\epsilon$ denote the residue field of $\Sring$. We consider
the algebras 
\begin{align*}
&U^{DK}_{\Sring}:=\Sring\otimes_{\C[v^{\pm 1}]}U^{DK}_v,U^{DK}_{\epsilon}:=\C_\epsilon\otimes_{\C[v^{\pm 1}]}U^{DK}_v.
\end{align*} 
Note that $U^{DK}_{\Sring}$ naturally acts on any module from $\OCat^\zeta_\Ring$.

In this section we recall some results from \cite{DCKP} on the structure of the centers of these algebras 
(and also about the structure of $U^{DK}_\epsilon$ over its center).

\subsubsection{Harish-Chandra center of $U^{DK}_{\Sring}$}\label{SSS_HC_center}
Let $U^{0,ev}_\Sring$ denote 
$\operatorname{Span}_{\Sring}(K_{2\nu}| \nu\in \Lambda)$. Note that the group $W$ acts 
on $U^{0,ev}_\Sring$ by automorphisms. Also note that $(\rho,2\nu)\in \Z$ for all $\nu\in \Lambda$.

Let $\Zcal_{\Sring}$ denote the center of $U^{DK}_\Sring$. The triangular decomposition 
$$U^{-}_{\Sring}\otimes_{\Sring}U^0_{\Sring}\otimes_{\Sring}U^+_{\Sring}\xrightarrow{\sim} U^{DK}_\Sring$$
gives rise to a homomorphism $\pi: Z_{\Sring}\rightarrow U^0_{\Sring}$ so that any element $z\in Z_{\Sring}$
acts on $\Delta^{\zeta}_\Ring(\lambda)$ by $\chi_\lambda\circ \pi(z)$. We note that, while we use Sevostyanov's generators
for the triangular decomposition, the homomorphism $\pi$ we have defined coincides with the analogous morphism for 
the standard generators. Let $\gamma_{\rho}$ denote the automorphism of $U^{0,ev}_\Sring$ 
that sends $K_{\nu}$ to $v^{(\rho,\nu)}K_\nu$. 

\begin{Lem}\label{Lem:HC_center_quantum}
The homomorphism $\pi$ is injective and its image coincides with $\gamma_{\rho}(\left(U_\Sring^{0,ev}\right)^W)$.
\end{Lem}
\begin{proof}
The claim is classical when we base change to $\operatorname{Frac}(\Sring)$, see, e.g., \cite[Section 6.2]{DCKP}. This isomorphism easily implies that $\pi$ restricts to an embedding $\Zcal_{\Sring}\hookrightarrow \gamma_{\rho}(\left(U_\Sring^{0,ev}\right)^W)$. 
The claim that this restriction is surjective follows from the argument of the proof (absence of poles) in 
\cite[Proposition 6.2]{DCKP}.   
\end{proof}     

Set $\Zcal^{HC}_\epsilon:=\C_{\epsilon}\otimes_{\Sring}\Zcal_\Sring$. This algebra embeds into the center of 
$U^{DK}_\epsilon$ and it is known as the Harish-Chandra center of $U^{DK}_\epsilon$. 

\subsubsection{$\epsilon$-center}\label{SSS_epsilon_center}
In \cite{DCKP}, another central subalgebra of $\U^{DK}_\epsilon$ was introduced. Here we call it the {\it $\epsilon$-center} 
(by the analogy with the $p$-center in the universal enveloping algebras over characteristic $p$ fields; another name is the {\it Frobenius center}) and denote it by $\Zcal^\epsilon$. By the construction, \cite[Section 4]{DCKP}, this algebra splits as the tensor product
$$\Zcal^{\epsilon,-}\otimes \Zcal^{\epsilon,0}\otimes \Zcal^{\epsilon,+},$$
where
\begin{itemize}
\item $\Zcal^{\epsilon,-}$ is the polynomial algebra on $\underline{F}_\alpha^d$, where $\alpha$ runs over the set of positive roots, and  the $\underline{F}_\alpha$'s are the usual root vectors in $\underline{U}^-_\epsilon$ recalled in Section \ref{SSS_quantum_basic}. 
\item $\Zcal^{\epsilon,0}$ is the span of $K_\nu$ with $\nu\in d\Lambda$.
\item $\Zcal^{\epsilon,+}$ is the polynomial algebra on the elements $\mathsf{T}_{w_0}\underline{F}_\alpha^d$. We note that 
$T_{w_0}\underline{F}_\alpha^d$ is the product of a suitable element $K_?$ and an element from $\underline{U}^+$. 
\end{itemize} 

Let $G$ be the simply connected semisimple group with Lie algebra $\g$.
Let $N^{\pm}$ denote the positive and negative unipotent subgroups of $G$, so that $G^0:=N^+\otimes T\otimes N^-$
is the open Bruhat cell in $G$. As explained in \cite[Section 4]{DCKP}, we have identifications $\C[N^\pm]\xrightarrow{\sim}
\Zcal^{\epsilon,\mp}$ (note that the role of $N^+,N^-$ here is swapped compared to \cite{DCKP}, see, e.g., Section 0.6
there). We will need to recall the construction of the isomorphism $\operatorname{Spec}(\C[\Zcal^{\epsilon,-}])\xrightarrow{\sim} N^+$,
see \cite[Section 0.6]{DCKP}: it is given by 
\begin{equation}\label{eq:isomorphism_N+}
z\mapsto \prod_{i=1}^\ell \exp([F_{\alpha_i}^d](z)e_{\alpha_i}).
\end{equation}

Note that the isomorphisms $\C[N^{\pm}]\xrightarrow{\sim} \Zcal^{\epsilon,\mp}$ are $T$-equivariant. 
Next, consider the embedding $\C[T]\hookrightarrow \Zcal^{\epsilon,0}$ that sends the function on 
$T$ corresponding to the character $\lambda$ to $K_{2d\lambda}$. So we get an etale finite morphism $\operatorname{Spec}(\Zcal^\epsilon)
\twoheadrightarrow G^0$ of degree $2^{|I|}$. 

Finally, we can get a description of the full center of $U^{DK}_\epsilon$. The following is a consequence of 
\cite[Theorem 6.4]{DCKP} and its proof. 

\begin{Prop}\label{Prop:UDK_center_structure}
The natural homomorphism 
$$\Zcal^{HC}_\epsilon\otimes_{\Zcal^{HC}_\epsilon\cap \Zcal^{\epsilon}}\Zcal^\epsilon\rightarrow U^{DK}_\epsilon$$
is an embedding whose image coincides with the center $Z(U^{DK}_\epsilon)$ of $U^{DK}_\epsilon$. The algebra 
$\Zcal^{HC}_\epsilon\cap \Zcal^\epsilon$ is identified with $\C[G]^G$.
\end{Prop}

In particular, we see that $Z(U^{DK}_\epsilon)$ is a finitely generated algebra, and its spectrum is an affine 
variety. 

\subsubsection{Azumaya locus}
The algebra $U^{DK}_\epsilon$ is a free rank $d^{\dim \g}$-module over $\Zcal^\epsilon$. In particular, 
it is a finitely generated module over $Z(U^{DK}_\epsilon)$. So, it makes sense to speak about the 
{\it Azumaya locus} of $U^{DK}_\epsilon$, the maximal open subset in $\operatorname{Spec}(Z(U^{DK}_\epsilon))$
over which $U^{DK}_\epsilon$ is Azumaya. We write $\operatorname{Spec}(Z(U^{DK}_\epsilon))^{pr}$ for 
the preimage of the locus of regular elements in $G$ under the composition 
$$\operatorname{Spec}(Z(U^{DK}_\epsilon))\rightarrow \operatorname{Spec}(\Zcal^\epsilon)\rightarrow G^0\hookrightarrow G.$$ 

The following was essentially obtained in \cite{DCKP_degen} but we provide a proof for reader's convenience. 

\begin{Prop}\label{Prop:Azumaya_locus}
$\operatorname{Spec}(Z(U^{DK}_\epsilon))^{pr}$ is contained in the Azumaya locus.
\end{Prop}
\begin{proof}
For $z\in \operatorname{Spec}(Z(U^{DK}_\epsilon))$ we write $U^{DK}_{\epsilon,z}$ for the fiber of 
$U^{DK}_\epsilon$ at $z$.
Set $n:=\dim \mathfrak{n}$. According to \cite[Theorem 5.1]{DCKP_degen}, for $z\in \operatorname{Spec}(Z(U^{DK}_\epsilon))^{pr}$ 
the dimension of every irreducible $U^{DK}_{\epsilon,z}$-module is $d^n$ (note that the theorem talks about orbits of a certain infinite
dimensional group $\tilde{G}$, according to \cite[Theorem 6.6]{DCKP}, these orbits are preimages of the intersections 
of adjoint $G$-orbits with $G^0$). From \cite[Sections 1.7]{DCK} it follows that $U^{DK}_\epsilon$ has no zero divisors. 
Using \cite[(3.8.3)]{DCK} we conclude that 
$$\operatorname{Frac}(Z(U^{DK}_\epsilon))\otimes_{Z(U^{DK}_\epsilon)}U^{DK}_\epsilon$$
is the matrix algebra of dimension  $d^{2n}$ over $\operatorname{Frac}(Z(U^{DK}_\epsilon))$. 
It follows that $U^{DK}_\epsilon$ is a PI algebra of rank $d^{2n}$. Now we can combine the claim that the dimension 
of every irreducible $U^{DK}_{\epsilon,z}$-module is $d^n$ with the Artin-Procesi theorem, 
\cite[Theorem 13.7.14]{MR}, to deduce the claim of the proposition. 
\end{proof}

\subsection{Whittaker coinvariants}
The goal of this section is to produce a functor $\OCat^\zeta_\Ring\rightarrow \Zcal_v\otimes_{\C[v^{\pm 1}]}\Ring\operatorname{-mod}$
(where the homomorphism $\C[v^{\pm 1}]\rightarrow \Ring$ is given by $v\mapsto q:=\epsilon e^{2\pi\sqrt{-1}\hbar/d}$)
and study its properties. In particular, we will see that, for $\zeta=1$, this functor is 
\begin{itemize}
\item Faithful on standardly filtered objects in $\OCat^{\zeta}_\epsilon$, the specialization of 
$\OCat^\zeta_\Ring$ to the closed point of $\operatorname{Spec}(\Ring)$. 
\item Fully faithful on standardly filtered objects in $\OCat^\zeta_\Ring$. 
\end{itemize}
We will also compute the images of standard objects under the functor. 

\subsubsection{Definition}
Recall, Section \ref{SS_Sevostyanov}, that we have a homomorphism $\psi: U^-_v\rightarrow \C[v^{\pm 1}]$ defined on the generators 
by $\psi(F_i)=1$ for all $i\in I$. Set $B:=\Zcal_v\otimes_{\C[v^{\pm 1}]}\Ring$.

\begin{defi}\label{defi:Wh_coinv}
Define a functor $\Wh: \OCat^\zeta_{\Ring}\rightarrow B\operatorname{-mod}$
by sending $M\in \OCat_{\Ring}$ to
$$\C[v^{\pm 1}]\otimes_{U^-_v}M,$$
where the homomorphism $U^-_v\rightarrow \C[v^{\pm 1}]$ is $\psi$. Note that $\Wh(M)$ is a quotient of 
$M$, and the kernel is $\Zcal_v$ and $\Ring$-stable. Hence we get a natural action of 
$B$ on $\Wh(M)$.
\end{defi}

Note that by the very definition, $\Wh$ is right exact and $\Ring$-linear. 

\begin{Ex}\label{Ex:Wh_Verma}
Note that $\Delta_\Ring^\zeta(\lambda)$ is a free rank $1$ module over $\Ring\otimes_{\C[v^{\pm 1}]}U_v^-$.
It follows that one has an $\Ring$-linear identification of $\Wh(\Delta_\Ring^\zeta(\lambda))$ with $\Ring$.
The algebra $\Zcal_v$ acts on $\Wh(\Delta_\Ring^\zeta(\lambda))$ via a homomorphism to $\Ring$, the same as
for the action on $\Delta_\Ring^\zeta(\lambda)$ itself. Under the identification $\Zcal_v\cong (U^{0,ev}_\Sring)^{W}$,
an element $z\in (U^{0,ev}_\Sring)^{W}$ acts by $\chi_{\lambda}\circ \gamma_\rho(z)$. One can write 
$\chi_{\lambda}\circ \gamma_\rho$ as $\chi_{\lambda+\rho}$: while $\rho$ may fail to be in  $\Lambda_0$,
formula (\ref{eq:chilambda_def}) is well-defined for $\lambda+\rho$ and $\nu\in 2\Lambda$.
\end{Ex}

\begin{Rem}\label{Rem:Wh_acyclic}
The functor $\Wh$ is acyclic on the Verma modules because they are free over $U^-_v\otimes_{\C[v^{\pm 1}]}\Ring$.
\end{Rem}

\subsubsection{Faithfulness}
Consider the functor $\Wh_\epsilon: \OCat^{\zeta}_\epsilon\rightarrow \Zcal^{HC}_\epsilon\operatorname{-mod}$
obtained by restricting $\Wh$. The goal of this part is to prove the following claim. 

\begin{Prop}\label{Prop:Wh_faithful}
The functor $\Wh$ is faithful on the standardly filtered objects in $\OCat^1_\epsilon$. 
\end{Prop}
\begin{proof}
It is enough to show that $\Wh$ gives an injective map 
$$\operatorname{Hom}_{\OCat^1_\epsilon}(\Delta^1_\epsilon(\lambda),
\Delta^1(\mu))\rightarrow  \operatorname{Hom}_{\Zcal^{HC}_\epsilon}(\Wh(\Delta^1_\epsilon(\lambda)),
\Wh(\Delta^1(\mu)))$$ for all $\lambda,\mu\in \Lambda_0$. The proof is in several steps. 

%{\it Step 1}. We abuse the notation and write $\Wh$ for the functor $\OCat^\zeta_\epsilon\rightarrow \operatorname{Vect}$.
%Note that the elements of $$
%For $\chi\in \operatorname{Spec}(\Zcal^{HC}_\epsilon)$, let $\OCat^{\zeta}_{\epsilon,\chi}$
%denote the full subcategory of $\OCat^\zeta_\epsilon$ consisting of all modules 

{\it Step 1}. %Now we determine the central character of $\Delta^\zeta_\epsilon(\lambda)$, i.e., the scalar by which 
%$Z(U^{DK}_\epsilon)$ acts on $\Delta^\zeta_\epsilon(\lambda)$. 
First, we are going to identify scalars by which various central subalgebras act on $\Delta^1_\epsilon(\lambda)$.
The HC center $\Zcal_\epsilon$ (identified with 
$(U^{0,ev}_\epsilon)^W$ via the HC isomorphism, see Section \ref{SSS_HC_center}) acts  on $\Delta_\epsilon^1(\lambda)$ 
via $\chi_{\lambda,\epsilon}$, the homomorphism $U^0_\epsilon\rightarrow \C$ induced 
by $\chi_\lambda:U^0_\Sring\rightarrow \Sring$.
% let $x^0_\lambda$ denote the corresponding 
%point in $\operatorname{Spec}(U^0_\epsilon)$ by $x^0$. 

Now we describe the action of the $\epsilon$-center $\Zcal^\epsilon$.
The subalgebra $\Zcal^{\epsilon,+}\cong \C[N^-]$ acts via evaluation at $1$. The subalgebra $\Zcal^{\epsilon,0}$
acts by the restriction of $\chi_\lambda: U^0_\epsilon\rightarrow \C$ to $\Zcal^{\epsilon,0}$, since $\zeta=1$, this restriction is 
the evaluation at $1$.
The subalgebra $\Zcal^{\epsilon,-}=\C[N^+]$ does not act by a scalar, but we note that for each element $x^+\in N^+$,
the image of $(1,1,x^+)\in N^-\times T_0\times N^+$ in the adjoint quotient $G/\!/G$ coincides with the image of $1$. So the point 
\begin{equation}\label{eq:point_def}
p:=\left((1,1,x^+),\chi_\lambda\right)\in G^0\times \operatorname{Spec}(\Zcal^{HC}_\epsilon)
\end{equation}
actually lies in the subvariety $\operatorname{Spec}(Z(U^{DK}_\epsilon))$, see Proposition \ref{Prop:UDK_center_structure}. 

{\it Step 2}. 
Now take an element $\underline{F}_\alpha^d\in \Zcal^{\epsilon,-}$ for a positive root $\alpha$. It is of the form 
$z_\alpha F_\alpha^d$, where $z_\alpha$ is a nonzero multiple of $K_?$ with $?\in d\Lambda$. 
Define a point $\psi^+\in N^+=\operatorname{Spec}(\Zcal^{\epsilon,-})$ by sending $\underline{F}_\alpha^d$ 
to $\psi(F_\alpha)^d$.  
Let $p\in \operatorname{Spec}(Z(U^{DK}_\epsilon))$ be the point $\left((1,1,\psi^+),\chi_\lambda\right)$. 
We claim that $z\in \operatorname{Spec}(Z(U^{DK}_\epsilon))^{pr}$. This will follow once we check that 
$\psi^+$ is a principal unipotent element. This is an easy consequence of  
(\ref{eq:isomorphism_N+}): thanks to this formula the expansion of $\ln(\psi^+)\in \mathfrak{n}^+$ in the root vectors 
$e_\alpha$ contains all root vectors corresponding to simple roots with nonzero coefficients.

{\it Step 3}. Take  $\chi\in \operatorname{Spec}(\Zcal^{HC}_\epsilon)$ and write 
$\OCat^1_\epsilon(\chi)$ for the full subcategory in $\OCat^1_\epsilon$ consisting of all
objects, where $\Zcal^{\epsilon,0}$, $\C[N^-]$ act via the evaluation at $1$,
and $\Zcal^{HC}_\epsilon$ acts by $\chi$. Note that $\Delta^1_\epsilon(\lambda)\in \OCat^1_\epsilon(\chi_\lambda)$. 
If $\mu\in \Lambda^0$ is such that $\chi_\mu\neq \chi_\lambda$,
then there are no homomorphisms between $\Delta^1_\epsilon(\lambda)$ and $\Delta^1_\epsilon(\mu)$
and there is nothing to prove. So we can assume that both $\Delta^1_\epsilon(\lambda)$ and $\Delta^1_\epsilon(\mu)$ lie in
the same subcategory $\OCat^1_\epsilon(\chi)$.

{\it Step 4}. We note that $\Wh: \OCat^1_\epsilon\rightarrow \operatorname{Vect}$ factors
through $\operatorname{Wh}^\epsilon: M\mapsto M/(\underline{F}_\alpha^d-\psi^+(\underline{F}_\alpha^d))M$.
For $M\in \OCat^1_\epsilon(\chi)$, the algebra $Z(U^{DK}_\epsilon)$ acts on $\operatorname{Wh}^\epsilon(M)$
by evaluation at $p$ given by (\ref{eq:point_def}). So we get a functor  $\operatorname{Wh}^\epsilon: \OCat^1_\epsilon(\chi)\rightarrow 
\mathsf{u}_p\operatorname{-mod}$, where we write $\mathsf{u}_p$ for the fiber of $U^{DK}_\epsilon$
at $p$. It is a matrix algebra of dimension $d^{2N}$, where $N=\dim \mathfrak{n}$, see 
Proposition \ref{Prop:Azumaya_locus}. Note that $\operatorname{Wh}^\epsilon(\Delta^1_\epsilon(\lambda))$ is of dimension $d^N$. Since 
$\Wh(\Delta^1_\epsilon(\lambda))\cong \C$, we see that $\Wh$ is the composition of 
$\Wh^\epsilon$ and an equivalence $\mathsf{u}_p\operatorname{-mod}\rightarrow \operatorname{Vect}$. 
So it remains to show that $\Wh^\epsilon$ is faithful on the objects $\Delta^1_\epsilon(\lambda)\in \OCat^1_\epsilon(\chi)$.
Let $\varphi:\Delta^1_\epsilon(\lambda_1)\rightarrow \Delta^1_\epsilon(\lambda_2)$ be a homomorphism.
Let $M$ be its cokernel. Since $\Wh^\epsilon(\Delta^1_\epsilon(\lambda_i))$ is the irreducible $\mathsf{u}_p$-module, 
the claim that $\Wh^\epsilon(\varphi)=0$ is equivalent to $M_p\neq 0$, equivalently, $p$ lies in the support of 
the $M\in \OCat^1_\epsilon$ viewed as a $Z(U^{DK}_\epsilon)$-module.

{\it Step 5}.  The elements $E_i^{(d)}$ (and hence $\underline{E}_i^{(d)}$) act on $M$. 
We can view $M$ as a module over $\Zcal^\epsilon$. As was shown in \cite[Section 7.3]{DCKP}, $\Zcal^\epsilon$ has a natural Poisson bracket. For $a\in \Zcal^{\epsilon}, m\in M$, we have $E_i^{(d)}(am)=\{E_i^d,a\}m+ a(E_i^{(d)}m)$. The derivation 
$\{E_i^d,\bullet\}$ of $\Zcal^\epsilon$ has been computed in \cite[Theorem 5.4]{DCKP}, it is equal to
$z_i e_i$, where $z_i$ is an invertible element of $\Zcal^\epsilon$, and $e_i$ is the left-invariant 
vector field on $G$ corresponding to the Cartan generator $e_i\in \g$ (the roles of $e_i,f_i$
in this paper are switched from those in \cite{DCKP}, compare to Section \ref{SSS_epsilon_center}).
So we can define the operators $e_i,i\in I,$ on $M$ that are derivations for the $\Zcal^\epsilon$-module structure.

Note that $\C[G^0]\subset \Zcal^\epsilon$ is stable under the derivations $e_i$. We view modules from 
$\OCat^1_\epsilon(\chi)$ as graded $\C[G^0]$-modules equipped additionally with the derivations $e_i$. 
In these terms, the functor $\Wh^\epsilon$ takes the fiber of such a module at the principal point $\psi^+\in N^+\subset G^0$. 
The support of $M$ in $B$ has the following properties:
\begin{enumerate}
\item it is contained in the closed subvariety $N^+\subset G^0$, 
\item it is stable under the left-invariant constant vector fields $e_i$ on $N^+$,
\item it is invariant under the diagonal action of $T$,
\item and it contains $\psi^+$.
\end{enumerate} 
Note that the vectors $e_i\in \mathfrak{n}^+$ generate the Lie algebra $\mathfrak{n}^+$. The group $N^+$
acts on $N^+$ by $n.n':=n'n^{-1}$, while $T$ acts by $t.n'=t^{-1}n't$,
together these actions give an action of $B$. Thanks to (2), (3) and (4), the support of $M$ contains 
$B\psi^+$. But $\psi^+$ is principal and so $B\psi^+$ is open in $N^+$. By the construction of $M$,
it follows that the image of $\varphi$ is supported on  $N^+\setminus B\psi^+$. This is impossible
because $\Delta^1_\epsilon(\lambda_2)$ is a free $U^-_\epsilon$-module and hence free over 
$\mathcal{Z}^{\epsilon,-}$.
\end{proof}

\begin{Rem}\label{Rem:Wh_nonintegral}
One can generalize Proposition \ref{Prop:Wh_faithful} as long as a suitable element of $B=T\ltimes N^+$
is principal. Namely, we can define $\psi^+$ similarly to the above. Also, we consider $\zeta'$, the restriction of
$\zeta: U^0_\epsilon\rightarrow \C$ to $\C[T]\hookrightarrow U^0_\epsilon$. The element we need is $(\zeta',\psi^+)$.
\end{Rem}

\subsection{Ext's between standards}
From now on we are going to assume that $\zeta=1$. We will write $\OCat_\Ring$ instead of $\OCat^1_\Ring$
and $\Delta_\Ring(\lambda)$ instead of $\Delta^1_\Ring(\lambda)$.
In this section we are interested in the (non)vanishing of the $\Ring'$-modules $\Ext^i_{\OCat^{\zeta}_{\Ring'}}(\Delta_{\Ring'}(\lambda), \Delta_{\Ring'}(\mu))$ for suitable local $\Ring$-algebras $\Ring'$. 

%The main result of this section is the following proposition. 
%
%\begin{Prop}\label{Prop:quantum_stand_Ext}
%The following claims are true:
%\begin{enumerate}
%\item Let $\F:=\operatorname{Frac}(\Ring)$. Then the category $\OCat^\zeta_\F$ is semisimple.
%\item Let $\p$ be a height $1$ prime ideal in $\Ring$. Then the following conditions are equivalent:
%\begin{itemize} 
%\item $\Ext^1_{\OCat^{\zeta}_{\Ring}}(\Delta_{\Ring}(\lambda), \Delta_{\Ring}(\mu))_\p\neq 0$, 
%\item 
%\end{enumerate}
%\end{Prop}

Let $\mathfrak{m}'\subset \Ring'$ denote the maximal ideal
and $\kf'$ be the residue field. Recall that $q$ stands for $\epsilon e^{2\pi\sqrt{-1}\hbar/d}$.
 
\subsubsection{Central characters}\label{SSS_centr_char}
%For $\lambda\in \Lambda_0, \nu\in \Lambda$, define 
%\begin{equation}\label{eq:chi_spec}
%\chi_{\lambda,\nu}:=\chi_{\lambda+\rho}(\sum_{w\in W}K_{2 w\nu})=\sum_{w\in W}q^{(\lambda+\rho,2w\nu)}e^{(2w\nu,\cdot)}\zeta(2w\nu). 
%\end{equation}
For time being, let $\Ring'$ be a local $\Ring$-algebra. The following lemma is standard. 

\begin{Lem}\label{Lem:Ext_non_vanish_necessary}
Suppose $\Ext^i_{\OCat_{\Ring'}}(\Delta^\zeta_{\Ring'}(\lambda), \Delta_{\Ring'}(\mu))\neq 0$ for some $i=0,1$. 
Then for all $\nu$ in $\Lambda$ we have $\chi_{\lambda+\rho}(\sum_{w\in W}K_{2w\nu})-\chi_{\mu+\rho}(\sum_{w\in W} K_{2w\nu})\in \mathfrak{m}'$.
\end{Lem}

This lemma can be interpreted as follows. Let $T_{ev}$ denote the torus $\operatorname{Spec}(\operatorname{Span}(K_{2\nu}|\nu\in \Lambda))$.
For $\lambda\in\Lambda_0$, define a $\kf'$-point $p_\lambda$ in $T_{ev}$ by sending $K_{2\nu}$ to the image of $\chi_{\lambda+\rho}(K_{2\nu})$ in 
$\kf'$. The (usual, not shifted) action of $W$ on $2\Lambda$ gives rise to an action of $W$ on $T_{ev}$. The condition of the lemma means that 
$p_\lambda$ and $p_\mu$ lie in the same $W$-orbit. 

We still need a more convenient equivalent formulation. For this, we define a $\kf'$-point $\hat{p}_\lambda$ of $\h^*$. Let $a\in \{0,1,\ldots,d-1\}$ be such that $\epsilon=\exp(2\pi\sqrt{-1}a/d)$. Define the lattice $\Lambda_0^\vee\subset \h^*$ as $\{\lambda\in \h^*| (\lambda,\Lambda)\subset \Z\}$, note that $\Lambda_0\subset \Lambda_0^\vee$. 
Then define $\hat{p}_\lambda$ by
\begin{equation}\label{eq:hat_p_point}
\hat{p}_\lambda:=\frac{a+\hbar}{d}(\lambda+\rho)+\sum_{i\in I}\iota(\varpi_i)\alpha^\vee_i, 
\end{equation} 
where $\iota$ is defined in the beginning of Section \ref{SS_O_quant}.
Here we abuse the notation and write $\hbar,\iota(\varpi_i)$ for their images in $\kf'$. 

Consider the natural action of $W\ltimes (\frac{1}{2}\Lambda_0^\vee)$ on $\h^*(\kf')=\h^*\otimes_{\C}\kf'$.   The following is an immediate consequence of Lemma \ref{Lem:Ext_non_vanish_necessary}. 

\begin{Cor}\label{Cor:Ext_non_vanish_necessary}
Suppose $\Ext^i_{\OCat^{\zeta}_{\Ring'}}(\Delta_{\Ring'}(\lambda), \Delta_{\Ring'}(\mu))\neq 0$ for some $i=0,1$. 
Then $\hat{p}_\lambda,\hat{p}_\mu$ lie in the same $W\ltimes (\frac{1}{2}\Lambda_0^\vee)$-orbit.
\end{Cor}  

Explicitly, the condition of Corollary \ref{Cor:Ext_non_vanish_necessary} means that there is $w\in W$ such that 
\begin{equation}\label{eq:Ext_non_vanish_nec_explicit}
d^{-1}\hbar(w(\lambda+\rho)-(\mu+\rho))+\sum_{i\in I}\iota(\varpi_i)(w\alpha^\vee_i-\alpha^\vee_i)\in  -\frac{a}{d}(w(\lambda+\rho)-(\mu+\rho))+ \frac{1}{2}\Lambda_0^\vee.  
\end{equation}

\subsubsection{Main result}
Here is the main result of this section. We use the notation introduced in Section \ref{SSS_centr_char}. 

\begin{Prop}\label{Prop:quantum_stand_extension}
Let $\p$ be a height $1$ prime ideal in $\Ring$. Then the following two conditions are equivalent:
\begin{enumerate}
\item $\Ext^1_{\OCat_{\Ring}}(\Delta_{\Ring}(\lambda), \Delta_{\Ring}(\mu))_\p\neq 0$,
\item $\p=(\alpha^\vee-k\hbar)$ and $\mu=\lambda-((\lambda+\rho,\alpha^\vee)+dk)\alpha$,
where $k\in \Z$ and $\alpha$ is a finite root. 
\end{enumerate}
Moreover, if this condition holds then the Ext in (1) is isomorphic to the residue field of $\Ring_\p$
(as an $\Ring_\p$-module).
\end{Prop}

\begin{Rem}\label{Rem:Wa_lattice_action}
Define the action of $W^a$ on $\Lambda_0$ so that $W$ acts by the $\cdot$-action, while
$t_\lambda\in \Lambda_0\subset W^a$ acts by $t_\lambda \mu=\mu+d\lambda$. Then $\mu$
is obtained from $\lambda$ by the reflection about the hyperplane given by $\alpha^\vee-k\hbar$,
compare to Section \ref{SSS_Hecke_setting}. 
\end{Rem}

\subsubsection{Necessary conditions}
The next lemma gives a necessary condition for (1) of Proposition \ref{Prop:quantum_stand_extension}. 

\begin{Lem}\label{Lem:nonvanishing_necessary_main}
Let $\Ring'=\kf'$ be  $\operatorname{Frac}(\Ring/\mathfrak{p})$ for a height at most $1$ prime ideal $\p\subset\Ring$. Suppose (\ref{eq:Ext_non_vanish_nec_explicit}) holds for 
some $\lambda,\mu, w$. Then one of 
the following conditions holds: 
\begin{itemize}
\item[(a)] $\p=(\hbar),w=1,$ and $\lambda-\mu\in d\Lambda_0$ or
\item[(b)] $w=s_\alpha$, and (2) of Proposition \ref{Prop:quantum_stand_extension} holds. 
\end{itemize} 
\end{Lem} 
\begin{proof}
Note that $\hbar,\iota(\varpi')$ lie in the maximal ideal, $\mathfrak{m}$, of $\Ring/\mathfrak{p}$.
So the left hand side of (\ref{eq:Ext_non_vanish_nec_explicit}) lies in $\mathfrak{m}\otimes_\C\h^*$,
while the right hand side is in $\h^*$. Since $\Ring/\mathfrak{p}=\C\oplus \mathfrak{m}$ we see that
(\ref{eq:Ext_non_vanish_nec_explicit}) is equivalent to the following two conditions
\begin{align}\label{eq:Ext_non_vanish_nec_explicit1}
& d^{-1}\hbar(w(\lambda+\rho)-(\mu+\rho))+\sum_{i\in I}\iota(\varpi_i)(w\alpha^\vee_i-\alpha^\vee_i)=0,\\\label{eq:Ext_non_vanish_nec_explicit2}
&\frac{a}{d}(w(\lambda+\rho)-(\mu+\rho))\in \frac{1}{2}\Lambda_0^\vee. 
\end{align}  

We first analyze (\ref{eq:Ext_non_vanish_nec_explicit2}). By our choice of $d$, it is coprime to the order of $(\frac{1}{2}\Lambda_0^\vee)/\Lambda$ (which is a power of $2$ if $\g$ has no summands of type $G_2$ and the product of $3$ and a power of $2$ otherwise).
So (\ref{eq:Ext_non_vanish_nec_explicit2}) is equivalent to $w(\lambda+\rho)-(\mu+\rho)\in d\Lambda_0$. 

Now we proceed to (\ref{eq:Ext_non_vanish_nec_explicit1}).  
First of all, notice that for $\p=\{0\}$, (\ref{eq:Ext_non_vanish_nec_explicit1}) implies that $w=1$ and $\lambda=\mu$:
this is because $\hbar$ and the elements $\iota(\varpi_i)$ are linearly independent over $\C$. So we only need to 
consider the case when $\mathfrak{p}$ has height $1$.

Assume that $w=1$. Since $\lambda\neq\mu$,
we get $\hbar=0$.  We conclude that option (a) holds.

Now consider the case when $w\neq 1$. Observe that if collections of complex numbers $(a^j,a^j_i),j=1,2$ are such that 
$a^j\hbar+\sum_{i\in I}a^j_i \iota(\varpi_i)=0$, then the collections $(a^1,a^1_i)$ and $(a^2,a^2_i)$
are proportional -- because $\p$ has height $1$. It then follows from (\ref{eq:Ext_non_vanish_nec_explicit1}) that
the vectors $w(\lambda+\rho)-(\mu+\rho),w\alpha_i^\vee-\alpha_i^\vee$ are proportional to each other.
Since $\alpha_i^\vee, i\in I$, span $\h^*$, we conclude that $\operatorname{rk}(w-\operatorname{id})=1$,
i.e., $w$ is a reflection, say $s_\alpha$.  So we see that there is  $k\in \C$ such that 
\begin{align*}
\lambda-\mu=((\lambda+\rho,\alpha^\vee)+dk)\alpha, k\hbar-\sum_{i\in I}(\alpha_i^\vee,\alpha^\vee)\iota(\varpi_i)=0.
\end{align*}
Note that $\sum_{i\in I}(\alpha_i^\vee,\alpha^\vee)\iota(\varpi_i)=\iota(\alpha^\vee)$, and so
$\p=(\alpha^\vee-k\hbar)$. Further,  $(\lambda+\rho,\alpha^\vee)$ is an integer and $\lambda-\mu$
is in the root lattice, so $dk$ must be an integer. So, if $w\neq 1$, then (\ref{eq:Ext_non_vanish_nec_explicit1})
holds if and only if $w=s_\alpha, \p=(\alpha^\vee-k\hbar)$ and $\mu=\lambda-((\lambda+\rho,\alpha^\vee)+k)\alpha$,
where $k\in d^{-1}\Z$.
Condition (\ref{eq:Ext_non_vanish_nec_explicit2}) now translates to $k\in \Z$.
\end{proof}

\begin{Cor}\label{Cor:generic_semisimple}
Let $\F:=\operatorname{Frac}(\Ring)$. Then $\OCat_\F$ is semisimple.  
\end{Cor}
\begin{proof}
This is because $\OCat_\F$ is a highest weight category with no $\Hom$'s and  $\Ext^1$'s between the standard objects. 
\end{proof}

Corollary \ref{Cor:generic_semisimple} implies the following claim that will be used later. 

\begin{Cor}\label{Cor:Ext_Hom_relation}
Let $\p$ be a height $1$ prime ideal in $\Ring$, and $\kf'$ be the residue field of $\Ring/\p$.
Then the following two conditions are equivalent:
\begin{itemize}
\item $\Ext^1_{\OCat_{\Ring}}(\Delta_{\Ring}(\lambda), \Delta_{\Ring}(\mu))_\p\neq 0$,
\item $\Hom_{\OCat_{\kf'}}(\Delta_{\kf'}(\lambda), \Delta_{\kf'}(\mu))\neq 0$.
\end{itemize} 
\end{Cor}

%\subsubsection{Sufficient condition for simplicity}
%\begin{Lem}\label{Lem:generic_simplicity}
% 
%\end{Lem}
%\begin{proof}
%It is enough to prove that there are no $\operatorname{Ext}^1$ between the standard objects. This follows from 
%Corollary \ref{Cor:Ext_non_vanish_necessary} because for $\kf'=\F$, the elements $\hbar, \iota(\varpi_i)/(2\pi\sqrt{-1})$
%are linearly independent over $\mathbb{Q}$. On the other hand, to have a nonzero $\Ext^1$ between standards,
%$\hat{p}_\lambda$ needs to lie in a rational affine hyperplane in $\h^*$. 
% \end{proof}
 
\subsubsection{Reduction to categories over $\C$}
To prove the equivalence of (1) and (2) in Proposition \ref{Prop:quantum_stand_extension} 
it is convenient to reduce the question to categories over $\C$. 
%We are going to pick points 
%$\underline{q}\in \C^\times, \underline{\zeta}\in \operatorname{Spec}(\operatorname{Span}_\C (K_\nu|\nu\in \Lambda))$
%as follows. 

Assume that (a) or (b) of  Lemma \ref{Lem:nonvanishing_necessary_main} hold. Set $T:=\operatorname{Spec}(\operatorname{Span}_\C(K_\nu|\nu\in \Lambda))$. We are going to define a codimension $1$ affine subtorus $Y\subset T\times \C^\times$ (a translate of a subtorus). 

Suppose (a) holds. Then $Y:=T\times \{\epsilon\}$. 

Suppose (b) holds. We can view $\alpha^\vee$ as a homomorphism $T\rightarrow \C^\times$. Set $Y:=\{(t,z)| \alpha^\vee(t)=z^k\}$
(so, in this case we get a genuine subtorus).

Set $Y':=Y\cap \{(\underline{\zeta},\underline{q})| \underline{q}^{2d_i}\neq 1,\forall i\}$. 
Note that for $(\underline{\zeta},\underline{q})\in T\times \C^\times$ such that $\underline{q}^{2d_i}\neq 1$ 
we can define the category $\OCat^{\underline{\zeta}}_{\underline{q}}$. 

The main result of this part is the following lemma. 

\begin{Lem}\label{Lem:homom_nonvanish}
Let $Y^0\subset Y'$ be a Zariski dense subset. 
The following claims are equivalent:
\begin{enumerate}
\item $\Hom_{\OCat_{\kf'}}(\Delta_{\kf'}(\lambda), \Delta_{\kf'}(\mu))\neq 0$,
\item $\Hom_{\OCat^{\underline{\zeta}}_{\underline{q}}}(\Delta^{\underline{\zeta}}_{\underline{q}}(\lambda), 
\Delta^\zeta_{\underline{q}}(\mu))\neq 0$ for all but finitely many elements $(\underline{\zeta},\underline{q})$
of $Y^0$.
\end{enumerate}
\end{Lem}
\begin{proof}
Let $R$ be a ring equipped with a homomorphism  $\eta: U^0_v=\C[T\times \C^\times]\rightarrow R$. Then we can consider 
the induced $U_v^{mix}\otimes_{\C[v^{\pm 1}]}R$-module $\Delta^\eta_R:=U_v^{mix}\otimes_{U_{v}^{mix,\geqslant 0}}R$,
it is naturally $\Lambda_0$-graded. For a given homomorphism $\eta$ and $\lambda\in \Lambda_0$ we can consider
a new homomorphism to $R$, it sends $v$ to $\eta(v)$ and  $K_\nu$ to $\eta(v)^{(\lambda,\nu)}\eta(K_\nu)$. 
We write $\Delta^\eta_R(\lambda)$ for $\Delta^{\eta_\lambda}_R$. 

Take $R=\C[Y']$ with a natural homomorphism $\eta:U^0_v\rightarrow \C[Y']$ and form modules $\Delta^\eta_R(\lambda),\Delta^\eta_R(\mu)$.
The space of $\Lambda$-graded $U_v^{mix}\otimes_{\C[v^{\pm 1}]}R$-linear maps 
$\Hom(\Delta^\eta_R(\lambda),\Delta^\eta_R(\mu))$ is identified with the $R$-module of solutions
of finitely many $R$-linear equations (given by $E_i^{(\ell)}$'s) 
in a finite rank $R$-module (the weight module in $U^-_v\otimes_{\C[v^{\pm 1}]}R$ of weight $\mu-\lambda$),
i.e, with $\ker\Phi_R$ for an $R$-linear map $\Phi_R$ between free finite rank $R$-modules $F^1_R\rightarrow F^2_R$.

Note that $\kf'$ is a faithfully flat $R$-module, so 
$\kf'\otimes_R M_R:=\Hom_{\OCat_{\kf'}}(\Delta_{\kf'}(\lambda), \Delta_{\kf'}(\mu))$. On the other hand,
$\Delta^{\underline{\zeta}}_{\underline{q}}(\lambda)=\C\otimes_R \Delta^\eta_R(\lambda)$, where the homomorphism 
$R=\C[Y']\rightarrow \C$ is given by evaluation at $(\underline{\zeta},\underline{q})\in Y'$. The space 
$\Hom_{\OCat^{\underline{\zeta}}_{\underline{q}}}(\Delta^{\underline{\zeta}}_{\underline{q}}(\lambda), 
\Delta^{\underline{\zeta}}_{\underline{q}}(\mu))$ is the kernel of $\C\otimes_R\Phi_R: \C\otimes_R F^1_R\rightarrow 
\C\otimes_R F^2_R$. 

So (1) is equivalent to the claim that $\ker\Phi_R$ is not torsion over $R$, and the equivalence of 
(1) and (2) is standard. 
\end{proof}

\subsubsection{Real root hyperplanes}\label{SSS:root_real}
Our goal here is to show that (2) of Lemma \ref{Lem:homom_nonvanish} holds in case (b) for a suitable Zariski dense 
subset $Y^0\subset Y'$.

Consider the locus $Y^0\subset Y'$
consisting of all $(\underline{\zeta},\underline{q})$ such that there is $\kappa\in \Lambda_0$ with 
$\underline{\zeta}(K_\nu)=\underline{q}^{(\kappa,\nu)}$ for all $\nu\in \Lambda$. The condition $(\underline{\zeta},\underline{q})\in Y$
is equivalent to $(\kappa,\alpha^\vee)=k$. The subset $Y^0$ is Zariski dense in $Y$ because $\underline{q}$
does not need to be  a root of unity. Set $\lambda':=\lambda+\kappa, \mu':=\mu+\kappa$. 
Observe that $s_\alpha\cdot \lambda'=\mu'$. Also observe that $\Delta^{\underline{\zeta}}_{\underline{q}}(\lambda)=\Delta_{\underline{q}}(\lambda'), 
\Delta^{\underline{\zeta}}_{\underline{q}}(\mu)=\Delta_{\underline{q}}(\mu')$. 
To establish (2) of Lemma \ref{Lem:homom_nonvanish} it is therefore enough to show that 
there is a nonzero homomorphism $\Delta_{\underline{q}}(\lambda')\rightarrow \Delta_{\underline{q}}(\mu')$
for a Weil generic element $\underline{q}$.
The existence of a homomorphism between Verma modules with this label in the usual category $\OCat$
is classical. The existence a homomorphism in $\OCat_{\underline{q}}$  then follows 
from \cite[Theorem 4.7]{EK}. 

So we see that indeed (2) (and hence (1)) of Lemma \ref{Lem:homom_nonvanish} holds in 
the setting of (b) of Lemma \ref{Lem:nonvanishing_necessary_main}. 

\subsubsection{Affine root hyperplane}
Here we show that, under condition (a) of Lemma \ref{Lem:nonvanishing_necessary_main}, 
(2) of Lemma \ref{Lem:homom_nonvanish} is not satisfied for $Y'=Y^0$.
Together with Section \ref{SSS:root_real} this will complete the proof of equivalence 
between (1) and (2) in Proposition \ref{Prop:quantum_stand_extension}. 

Let $\underline{\zeta}$ be the generic point of $\operatorname{Spec}(\operatorname{Span}(K_\nu| \nu\in \Lambda))$
and $\underline{\zeta}'$ and $\chi$ denote the images of $\underline{\zeta}$ in  
$\operatorname{Spec}(\Zcal^{\epsilon,0})$ and $\operatorname{Spec}(\Zcal^{HC}_\epsilon)$.
%If the objects $\Delta^{\underline{\zeta}}_\epsilon(\lambda), \Delta^{\underline{\zeta}}_\epsilon(\mu)$ 
%do not lie in the same subcategory $\OCat_\epsilon(\chi)$ (defined in Step 3 of the proof of Proposition
%(\ref{Prop:Wh_faithful}), there is nothing 
%to prove. The condition that $\Delta^{\underline{\zeta}}_\epsilon(\lambda), \Delta^{\underline{\zeta}}_\epsilon(\mu)$ 
%do not lie in the same subcategory $\OCat^\zeta_\epsilon(x^0,\chi)$ means that $\lambda-\mu\in d\Lambda_0$
%(just by comparing the $x^0$-components). In other words, $\Delta_\epsilon(\lambda),
%\Delta_\epsilon(\mu)$ are the same $U^{mix}_\epsilon$-modules but with different 
%gradings. 
  
Consider the point $p=((1,\underline{\zeta}',1),\chi)\in \operatorname{Spec}(Z(U^{DK}_\epsilon))$.
We note that the fiber $M_p$ of $M\in \OCat^{\underline{\zeta}}_\epsilon$ is $\Lambda_0$-graded as well. 
For the same reason as in Step 5 of the proof of Proposition \ref{Prop:Wh_faithful}, 
the functor $M\mapsto M_p$ is faithful on standardly filtered objects, compare to 
Remark \ref{Rem:Wh_nonintegral}.
But $\Delta^{\underline{\zeta}}_\epsilon(\lambda),
\Delta^{\underline{\zeta}}_\epsilon(\mu)$ have different gradings, so  the fibers 
$\Delta^{\underline{\zeta}}_\epsilon(\lambda)_p,
\Delta^{\underline{\zeta}}_\epsilon(\mu)_p$ are isomorphic modules over the matrix algebra $\mathsf{u}_p$ that have different gradings as well. There can be no nonzero 
graded homomorphisms between such modules finishing the proof.

\subsubsection{Infinitesimal blocks and functors to bimodules}\label{SSS_inf_blocks_bimod}
We now proceed to computing the $\Ring_\p$-modules  $\Ext^1_{\OCat_{\Ring}}(\Delta_{\Ring}(\lambda), \Delta_{\Ring}(\mu))_\p$,
where $\p,\lambda,\mu$ satisfy (2) of Proposition \ref{Prop:quantum_stand_extension}.

We start by introducing the infinitesimal blocks of the category $\OCat_\Ring$. Recall that to $\lambda\in \Lambda_0$ and
a local $\Ring$-algebra $\Ring'$ with residue field $\kf'$ we assign $\hat{p}_\lambda\in \h^*\otimes_{\C}\kf'$ by 
(\ref{eq:hat_p_point}). We apply this construction to $\Ring':=\Ring$ (and hence $\kf'=\C$) getting
\begin{equation}\label{eq:hat_p_point_special}
\hat{p}_\lambda:=\frac{a}{d}(\lambda+\rho). 
\end{equation} 
Define the equivalence relation $\sim$ on $\Lambda_0$ if the following equivalent (by Section \ref{SSS_centr_char}) conditions hold:
\begin{itemize}
\item For any $z\in \Zcal^{HC}_\epsilon$, the scalars of the action of $z$ on $\Delta_\epsilon(\lambda), \Delta_\epsilon(\mu)$
coincide.
\item $\hat{p}_\lambda,\hat{p}_\mu$ lie in the same orbit for the action of $W\ltimes (\frac{1}{2}\Lambda_0^\vee)$ on $\h^*$.  
\end{itemize}

Note that the latter condition is equivalent to the following (compare to the proof of Lemma \ref{Lem:nonvanishing_necessary_main}) 
\begin{itemize}
\item $\lambda,\mu$ lie in the same orbit for the action of $W^a$ on $\Lambda_0$ (see Remark \ref{Rem:Wa_lattice_action}).
\end{itemize}

Let $\Xi$ denote an orbit of $W^a$  in $\Lambda_0$. 
Let $\OCat^{\Xi}_\Ring$ denote the Serre span 
of $\Delta_\Ring(\lambda)$ with $\lambda\in \Xi$ so that we have 
$\OCat_\Ring=\bigoplus \OCat_\Ring^{\Xi}$, where the sum is taken over all equivalence classes in 
$\Lambda_0$. The summands $\OCat_\Ring^{\Xi}$ are called {\it infinitesimal blocks}
(in fact, they are also blocks in the usual sense).  

We note that the action of $W^a$ on $\Lambda_0$ has fundamental domain 
$$\Lambda_0^{fund}:=\{\lambda\in \Lambda_0| (\lambda+\rho,\alpha_i^\vee)\leqslant 0, (\lambda+\rho,\alpha_0^\vee)\geqslant -d\},$$
here $\alpha_0^\vee$ is the maximal coroot. Let $\lambda^\circ$ denote the unique point in $\Lambda_0^{fund}\cap \Xi$. 
Let $W_\circ$ be the stabilizer of $\lambda^\circ$ in $W^a$.

Recall the homomorphism $\gamma_{\lambda+\rho}:\C[T_{ev}\times \C^\times]\rightarrow \Ring,\lambda\in \Lambda_0,$  defined 
by (\ref{eq:chilambda_def}). 
%explicitly: 
%$$v\mapsto q:=\epsilon e^{2\pi\sqrt{-1}\hbar}, K_{2\nu}\mapsto q^{(\lambda+\rho,2\nu)}e^{4\pi\sqrt{-1}\iota(\nu)}\zeta(2\nu).$$

\begin{Lem}\label{Lem:center_action}
The following claims hold:
\begin{enumerate}
\item 
Let $x\in W^a$.
Then 
$\gamma_{x\cdot \lambda^\circ +\rho}=x\gamma_{\lambda^\circ+\rho}$. In particular, 
$\operatorname{im}\gamma_{\lambda^\circ+\rho}\in \Ring^{W_\circ}$.
\item The homomorphism $\gamma_{\lambda^\circ+\rho}:\Zcal_v\rightarrow \Ring^{W_\circ}$ (uniquely) extends to
an isomorphism $\Zcal_v^{\wedge_{z^\circ}}\xrightarrow{\sim} \Ring^{W_\circ}$.
\end{enumerate}
\end{Lem}
\begin{proof}
(1): we need to show that for all $\nu\in 2\Lambda$ we have 
$$\gamma_{x\cdot \lambda^\circ+\rho}(\sum_{w\in W}K_{w\nu})=(x\gamma_{\lambda^\circ+\rho})(\sum_{w\in W}K_{w\nu}).$$
Set $\lambda=\lambda^\circ+\rho$ to unload the notation. It's sufficient to prove the claim for $x=u\in W$
and $x=t_\mu, \mu\in \Lambda_0,$ separately. 

Let $x=u$. Then 
%Suppose that $x=t_\mu u$. Then $x\cdot \lambda^\circ+\rho=u(\lambda^\circ+\rho)+\mu$ and so 
\begin{align*}
&\gamma_{x\cdot \lambda^\circ+\rho}(\sum_{w\in W}K_{w\nu})=\sum_{w\in W}\exp(2\pi\sqrt{-1}[(\frac{a}{d}+\hbar)(u\lambda,w\nu)+ \iota(w\nu)])\\
&=\sum_{w\in W}\exp(2\pi\sqrt{-1}[\frac{a+\hbar}{d}(u\lambda,uw\nu)+ \iota(uw\nu)])\\
&=\sum_{w\in W}\exp(2\pi\sqrt{-1}[\frac{a+\hbar}{d}(\lambda,w\nu)+ u\iota(w\nu)])=x\gamma_{\lambda^\circ+\rho}(\sum_{w\in W}K_{w\nu}).
\end{align*} 

Let $x=t_\mu$. Then
\begin{align*}
&\gamma_{x\cdot \lambda^\circ+\rho}(\sum_{w\in W}K_{w\nu})=\sum_{w\in W}\exp(2\pi\sqrt{-1}[\frac{a+\hbar}{d}(\lambda+d\mu,w\nu)+ \iota(w\nu)])\\
&=\sum_{w\in W}\exp(2\pi\sqrt{-1}[\frac{a+\hbar}{d}(\lambda,w\nu)+\hbar(\mu,w\nu)+ \iota(w\nu)])=\\
&%\sum_{w\in W}\exp(2\pi\sqrt{-1}[\frac{a+\hbar}{d}(\lambda,w\nu)+\iota(w\nu+\hbar \mu)])
=x\gamma_{\lambda^\circ+\rho}(\sum_{w\in W}K_{w\nu}).
\end{align*}
This finishes the proof of (1). 

(2): Note that $\gamma_{\lambda^\circ+\rho}: \C[T_{ev}\times \C^\times]\rightarrow \Ring$ (uniquely) extends to an 
isomorphism between the completion of the source at $(\lambda^\circ,\epsilon)$, where $\lambda^\circ$ is viewed as the homomorphism $K_{2\nu}
\mapsto \exp\left(2\pi\sqrt{-1}(\lambda^\circ+\rho,2\nu)\right)$, and $\Ring$. We remark that projection $W^a\twoheadrightarrow W$
identifies $W_\circ$ with the stabilizer of $p_{\lambda^\circ}\in T_{ev}$ under the action of $W$.
It follows that $\Ring$ is a finite birational extension of $\Zcal_v^{\wedge_{z^\circ}}$. The latter is normal, finishing the proof. 
\end{proof}

We identify $\Zcal_v^{\wedge_{z^\circ}}\cong \Ring^{W_\circ}$ using the isomorphism from (2). So we can view 
$\Wh$ as a functor $\OCat^\Xi_\Ring\rightarrow \Ring^{W_\circ}$-$\Ring\operatorname{-bimod}$.

\begin{Cor}\label{Cor:Wh_stand_image}
The following claims are true:
\begin{enumerate}
\item The functor $\Wh$ is fully faithful on standardly filtered objects. 
\item
The $\Ring^{W_\circ}$-$\Ring$-bimodule $\Wh(\Delta_\Ring(x^{-1}\cdot \lambda^\circ))$
is identified with the graph bimodule $\Ring_x$.
\end{enumerate} 
\end{Cor}
\begin{proof}
Corollary \ref{Cor:generic_semisimple} combined with (1) of Lemma \ref{Lem:center_action} shows that 
$\Wh$ is fully faithful after specializing to $\F=\operatorname{Frac}(\Ring)$.
Proposition \ref{Prop:Wh_faithful} shows that $\Wh$ is faithful on standardly filtered objects specialized to $\C$.  
Repeating the argument of the proof of Corollary \ref{Cor:Vfun_ff_Ost}
we get (1).

(2) follows from (1) of Lemma \ref{Lem:center_action}. 
\end{proof}

\subsubsection{Computation of the localized $\Ext^1$}
Here we finish the proof of Proposition \ref{Prop:quantum_stand_extension}. 

Let $\lambda=x^{-1}\cdot \lambda^\circ$, then $\mu=(xs_{\tilde{\alpha}})^{-1}\cdot \lambda^\circ$, where $s_{\tilde{\alpha}}=t_{-k\alpha}s_{\alpha}$. Thanks to Corollary \ref{Cor:Wh_stand_image},
$$\Ext^1_{\OCat_\Ring}(\Delta_\Ring(\lambda),\Delta_\Ring(\mu))\hookrightarrow 
\Ext^1_{\Ring^{W_\circ}\otimes \Ring}(\Ring_x,\Ring_{xs_{\tilde{\alpha}}}).$$ 
Thanks to Section \ref{SSS:root_real}, the source is nonzero after localization at $\p$.
By Remark \ref{Rem:Ext1_stand_Ost_sing},
$\Ext^1_{\Ring^{W_\circ}\otimes \Ring}(\Ring_x,\Ring_{xs_{\tilde{\alpha}}})_\p\cong \kf'$,
where we write $\kf'$ for the residue field of $\Ring_\p$. This finishes the proof.

\section{Main equivalence theorem}
\subsection{Statement and proof}
The goal of this section is to prove the following result, which is the main result of this paper.
Let $\lambda^\circ,W_\circ,\Xi$ have the same meaning as in Section \ref{SSS_inf_blocks_bimod}. 
Then $W_\circ$ is a standard parabolic subgroup, let $J\subset I^a$ denote the corresponding finite set.

\begin{Thm}\label{Thm:main_equivalence_deformed}
We have an $\Ring$-linear equivalence of abelian categories $\,_J\OCat^{st}_\Ring\xrightarrow{\sim} \OCat^\Xi_\Ring$
sending $\Delta^{st}_\Ring(W_Jx)$ to $\Delta_\Ring(x^{-1}\cdot \lambda)$ for all $x\in W^a$.
\end{Thm}

Note that this theorem implies Theorem \ref{Thm:main_basic} by specializing at the closed point of $\operatorname{Spec}(\Ring)$.

\begin{proof}
The proof is in several steps.

{\it Step 1}. 
Thanks to Remark \ref{Rem:interval_finite_stand_to_all}, what we need to establish is an equivalence 
$(\,_J\OCat^{st}_\Ring)^\Delta\xrightarrow{\sim} (\OCat^\Xi_\Ring)^\Delta$ sending $\Delta^{st}_\Ring(W_Jx)$ to $\Delta_\Ring(x^{-1}\cdot \lambda)$.
Note that we have full embeddings
\begin{align*}
&
\Vfun: (\,_J\OCat^{st}_\Ring)^\Delta\rightarrow \Ring^{J}\operatorname{-}\Ring\operatorname{-bimod}, \Vfun(\Delta^{st}_\Ring(W_Jx))\cong \Ring_x,\\
&\Wh:(\OCat^\Xi_\Ring)^\Delta\rightarrow \Ring^{J}\operatorname{-}\Ring\operatorname{-bimod}, \Wh(\Delta_\Ring(x^{-1}\cdot \lambda^\circ))\cong \Ring_x.
\end{align*}
See Lemma \ref{Lem:Vfun_singular} for $\Vfun$, and Corollary \ref{Cor:Wh_stand_image} for $\Wh$.  It follows that 
we only need to show that 
\begin{equation}\label{eq:full_images_coincide}
\Vfun\left((\,_J\OCat^{st}_\Ring)^\Delta\right)=\Wh\left((\OCat^\Xi_\Ring)^\Delta\right).  
\end{equation} 

{\it Step 2}. We claim that $W_Jx\leqslant^{st}W_Jy$ implies $x^{-1}\cdot \lambda^\circ\leqslant y^{-1}\cdot \lambda^\circ$. 
The check is similar to that of Lemma \ref{Lem:order_implication}. It is enough to assume that $y=xs_\beta$ for a real root 
$\beta$ and $x<^{st}y$. Let $x=wt_\lambda$ for $w\in W, \lambda\in \Lambda_0$. Then $y$ takes the form $ws_\alpha t_{\lambda-m\alpha}$
for some positive finite type root $\alpha$ and some $m\in \Z$. The condition $x<^{st}y$ translates to the condition that either 
$m>0$ or $m=0$ but $w\alpha>0$, compare to the proof of Lemma \ref{Lem:order_implication}. Note that $y^{-1}\cdot \lambda^\circ-x^{-1}\cdot \lambda^\circ=k\alpha$, where 
\begin{equation}\label{eq:alpha_coefficient}
k:=dm+\langle\rho,\alpha^\vee\rangle-\langle\lambda^\circ+\rho, w\alpha^\vee\rangle.
\end{equation}
We have $\langle\rho,\alpha^\vee\rangle>0$.
If $m=0$, then $w\alpha^\vee>0$, so $\langle\lambda^\circ+\rho, w\alpha^\vee\rangle\leqslant 0$ and $k>0$.
Suppose $m>0$. Then  $\langle\lambda^\circ+\rho, w\alpha^\vee\rangle$ achieves the maximal value when 
$-w\alpha^\vee$ is the maximal coroot, hence  $\langle\lambda^\circ+\rho, w\alpha^\vee\rangle\leqslant d$.
So in this case $k>0$ as well. This completes the proof of the claim in the beginning of the step. 

{\it Step 3}. We consider both $\OCat^1_\Ring:=\,_J\OCat^{st}_\Ring, \OCat^2_\Ring:=\OCat^\Xi_\Ring$ as highest weight categories over $\Ring$
with poset $\T:=W_J\backslash W^a$ and the order defined by $W_Jx\leqslant W_Jy$ if $x^{-1}\cdot \lambda^\circ\leqslant y^{-1}\cdot \lambda^\circ$. 
Pick a finite poset interval $\underline{\T}\subset \T$ and let $\OCat^i_{\Ring,\underline{\T}}$ denote the corresponding highest weight 
subquotients. Note that $(\OCat^i_{\Ring,\underline{\T}})^\Delta\hookrightarrow (\OCat^i_{\Ring})^\Delta$ and so  (\ref{eq:full_images_coincide})
will follow once we show that 
\begin{equation}\label{eq:full_images_coincide1}
\Vfun((\OCat^1_{\Ring,\underline{\T}})^\Delta)=\Wh((\OCat^2_{\Ring,\underline{\T}})^\Delta)
\end{equation}
for all finite poset intervals $\underline{\T}$. 

{\it Step 4}. Consider projective generators $P^i_\Ring\in \OCat^i_{\Ring,\underline{\T}}$. They are standardly filtered and so 
can be viewed as objects of $\OCat^i_\Ring$. Let $I_\Ring$ be the intersection of the kernels 
of the actions of $\Ring^{J}\otimes \Ring$ on $\Vfun(P^1_\Ring),\Wh(P^2_\Ring)$, and set $\underline{A}_\Ring:=\Ring^{J}\otimes \Ring/I_\Ring,
\underline{\Cat}_\Ring:=\underline{A}_\Ring\operatorname{-mod}$. Observe that since $\Vfun(P^1_\Ring),\Wh(P^2_\Ring)$
are finitely generated over $\Ring$, the algebra $\underline{A}_\Ring$ is a finitely generated $\Ring$-module.
We have right exact functors $\pi^i_\Ring: \OCat^i_\Ring\rightarrow \underline{\Cat}_\Ring$ for $i=1,2$ that extend
$\Vfun:\OCat^{1,\Delta}_\Ring\rightarrow \underline{\Cat}_\Ring,\Wh:\OCat^{2,\Delta}_\Ring\rightarrow \underline{\Cat}_\Ring$. 

We claim that these functors are RS functors in the sense of Definition \ref{defi:RS_functor}. The properties (RS1),(RS2),(RS4) 
follow directly from the analogous properties of $\Vfun, \Wh$ established above. Namely, (RS1) and (RS2) for
$\Vfun$ were established in (1) of Lemma \ref{Lem:Vfun_singular}, while (RS4) follows from (2) of the same lemma combined with 
Remark \ref{Rem:RS4}. For $\Wh$, (RS1) follows from Remark \ref{Rem:Wh_acyclic}, (RS2) follows from Example 
\ref{Ex:Wh_Verma}, while (RS4) follows from Proposition \ref{Prop:Wh_faithful} combined with Remark \ref{Rem:RS4}.

Let us explain why (RS3) holds. 
Let $\F=\operatorname{Frac}(\Ring)$. The categories $\OCat^i_{\F,\underline{\T}}$ are split semisimple (see (3) of 
Lemma \ref{Lem:Vfun_singular} for $i=1$ and Corollary \ref{Cor:generic_semisimple} for $i=2$). The objects 
$P^i_\F$ are direct sums of the $\Ring^{J}\otimes \F$-bimodules $\F_x$ for $x\in \underline{\T}$. The union of graphs 
of the elements $x\in \underline{\T}$ is a closed subscheme in $\operatorname{Spec}(\Ring^{J})\times \operatorname{Spec}(\F)$
isomorphic to the disjoint union of $|\underline{\T}|$ many copies of $\operatorname{Spec}(\F)$. The algebra  
$\Ring^{J}\otimes \F$ acts on the objects $\pi^i_{\F}(P^i_\F)$ via the projection to the algebra of functions on 
the union of graphs. Since every standard $\Delta^i_\Ring(x)$ appears in the standard filtration of $P^i_\Ring$,
we see that $\underline{A}_\F\cong \F^{\oplus \underline{\T}}$ and that the functors $\pi^i_\F$ are equivalences. 

{\it Step 5}. To prove (\ref{eq:full_images_coincide1}) we will use Theorem \ref{Thm:hw_equiv}. Condition (a) 
follows from Step 1, so we only need to establish (b): that 
\begin{equation}\label{eq:full_images_coincide2}
\pi^1_{\Ring_\p}((\OCat^1_{\Ring_\p,\underline{\T}})^\Delta)=\pi^2_\Ring((\OCat^2_{\Ring_\p,\underline{\T}})^\Delta)
\end{equation}
for all height $1$ prime ideals $\p\subset \Ring$. Let us list consequences of Lemma
\ref{Lem:Ext1_stand_Ost_sing} (for the category $\OCat^1_\Ring$) and Proposition \ref{Prop:quantum_stand_extension}
(for the category $\OCat^2_\Ring$) for a pair of elements $W_Jx,W_Jy\in \T$:
\begin{itemize}
\item[(i)] If $\p$ is of the form $(\beta^\vee:=\alpha^\vee-k\hbar)$ for a Dynkin root $\alpha$ and $k\in \Z$, and $W_J x,W_Jy\in \underline{\mathcal{T}}$ 
are such that $y= xs_\beta, x<y$ (for some choice of $x,y$ in the respective cosets), then $\operatorname{Ext}^1_{\OCat^i_{\Ring_\p,\T_0}}(\Delta^i_{\Ring_\p}(W_Jx),\Delta^i_{\Ring_\p}(W_Jy))=\kf$,
\item[(ii)] otherwise, $\operatorname{Ext}^1_{\OCat^i_{\Ring_\p,\T_0}}(\Delta^i_{\Ring_\p}(W_Jx),\Delta^i_{\Ring_\p}(W_Jy))=0$.
\end{itemize}
We also know, Remark \ref{Rem:Ext1_stand_Ost_sing}, that 
\begin{itemize} 
\item 
$\operatorname{Ext}^1_{\Ring^J\otimes \Ring}(\Ring_x,\Ring_y)_\p=\kf$ if $W_Jy=W_Jxs_\beta$ and $\p=(\beta^\vee)$, and zero otherwise.
\end{itemize} 
We conclude that $\OCat^i_{\Ring_\p,\underline{\T}}$ splits into the direct sum of highest weight categories, where each summand 
has one or two standard objects. And since all nonzero $\Ext^1$'s are isomorphic to $\kf$, we use Remark \ref{Rem:proj_construction}
to conclude that the images of the projectives 
in $\OCat^1_{\Ring_\p,\underline{\T}}$ under $\pi^1_{\Ring_\p}$ coincide with the images of the projectives in  
$\OCat^2_{\Ring_\p,\underline{\T}}$ under $\pi^2_{\Ring_\p}$. This implies (\ref{eq:full_images_coincide2}) and finishes the proof. 
\end{proof}

\subsection{Remarks}
\begin{Rem}\label{Rem:non_integral}
Let us explain modifications needed to handle the non-integral case. Our category $\OCat^2_\Ring$ is still $\OCat^{\zeta,\Xi}_\Ring$. 
For $\OCat^1_\Ring$ we take a suitable direct sum of categories of the form $\,_J\OCat^{st}_\Ring$ for a suitable pseudo-Levi of 
$G$ and suitable subsets $J$. We expect the details to be addressed elsewhere.   
\end{Rem}

\begin{Rem}\label{Rem:projective_image}
One can still talk about projectives for the categories $\OCat^{st}_\Ring,\OCat_\Ring$, but they are pro-objects. 
It still makes sense to talk about their images under the functors $\Vfun,\Wh$. We believe that these images should be 
some stabilized version of (singular) Soergel bimodules for $W^a$. More precisely, we expect that the indecomposable Soergel
bimodules $B_{xt_{-\lambda}}$ should stabilize (in a certain precise sense) as $\lambda$ goes to $-\infty$. The inverse limit 
of these bimodules should be the image of the indecomposable projectives labelled by $x$ in the regular blocks of 
$\OCat^{st}_\Ring,\OCat_\Ring$. Note that similar phenomena in a closely related context of sheaves on moment graphs
were observed and proved in \cite{Lanini1}, see also \cite{Lanini2}.  
\end{Rem}

\end{document}